\documentclass[american]{article}
\usepackage[T1]{fontenc}
\usepackage[utf8]{inputenc}
\usepackage{babel}
\usepackage{pifont}
\usepackage{wrapfig}
\usepackage{enumitem}
\usepackage{amsmath}
\usepackage{amsthm}
\usepackage{amssymb}
\usepackage{graphicx}
\usepackage{geometry}
\usepackage{xargs}[2008/03/08]
\usepackage[pdfusetitle,
 bookmarks=true,bookmarksnumbered=false,bookmarksopen=false,
 breaklinks=true,pdfborder={0 0 1},backref=false,colorlinks=false]
 {hyperref}

\makeatletter
\numberwithin{figure}{section}
\theoremstyle{plain}
\newtheorem{thm}{\protect\theoremname}[section]
\theoremstyle{definition}
\newtheorem{defn}[thm]{\protect\definitionname}
\theoremstyle{plain}
\newtheorem{conjecture}[thm]{\protect\conjecturename}
\theoremstyle{plain}
\newtheorem{cor}[thm]{\protect\corollaryname}
\theoremstyle{definition}
\newtheorem{example}[thm]{\protect\examplename}
\theoremstyle{remark}
\newtheorem{claim}[thm]{\protect\claimname}
\theoremstyle{plain}
\newtheorem{lem}[thm]{\protect\lemmaname}
\theoremstyle{remark}
\newtheorem{rem}[thm]{\protect\remarkname}
\newlist{casenv}{enumerate}{4}
\setlist[casenv]{leftmargin=*,align=left,widest={iiii}}
\setlist[casenv,1]{label={{\itshape\ \casename} \arabic*.},ref=\arabic*}
\setlist[casenv,2]{label={{\itshape\ \casename} \roman*.},ref=\roman*}
\setlist[casenv,3]{label={{\itshape\ \casename\ \alph*.}},ref=\alph*}
\setlist[casenv,4]{label={{\itshape\ \casename} \arabic*.},ref=\arabic*}
\theoremstyle{plain}
\newtheorem{prop}[thm]{\protect\propositionname}
\theoremstyle{plain}
\newtheorem{lyxalgorithm}[thm]{\protect\algorithmname}

\@ifundefined{date}{}{\date{}}

\usepackage{flafter}

\usepackage{fancyhdr}
\renewcommand{\emph}[1]{\textbf{#1}}  

\usepackage{tikz-cd}

\ifdefined\showcaptionsetup
 \PassOptionsToPackage{caption=false}{subfig}
\fi
\usepackage{subfig}
\AtBeginDocument{
  
}

\makeatother

\providecommand{\algorithmname}{Algorithm}
\providecommand{\casename}{Case}
\providecommand{\claimname}{Claim}
\providecommand{\conjecturename}{Conjecture}
\providecommand{\corollaryname}{Corollary}
\providecommand{\definitionname}{Definition}
\providecommand{\examplename}{Example}
\providecommand{\lemmaname}{Lemma}
\providecommand{\propositionname}{Proposition}
\providecommand{\remarkname}{Remark}
\providecommand{\theoremname}{Theorem}

\begin{document}
\global\long\def\smallparens#1{\mathopen{{(}}#1\mathclose{{)}}}%

\global\long\def\largeparens#1{\mathopen{{}}\left(#1\right)\mathclose{{}}}%

\global\long\def\mathemph#1{\boldsymbol{#1}}%

\global\long\def\spi{\operatorname{s\pi}}%

\global\long\def\scl{\operatorname{scl}}%

\global\long\def\sql{\operatorname{sql}}%

\global\long\def\ssql{\operatorname{ssql}}%

\global\long\def\cc{\operatorname{CC}}%
\global\long\def\cycles{\operatorname{Cycles}}%

\global\long\def\imm{\looparrowright}%

\global\long\def\boundary{\partial}%

\global\long\def\U#1{U\smallparens{#1}}%

\global\long\def\aut{\operatorname{Aut}}%

\global\long\def\skeleton#1{{#1}^{\smallparens 1}}%

\global\long\def\universalCover#1{\widetilde{{#1}}}%

\global\long\def\fundamental#1{\pi_{1}\smallparens{#1}}%

\newcommandx\num[2][usedefault, addprefix=\global, 1=w]{\#_{#1}\smallparens{#2}}%

\newcommandx\expectation[2][usedefault, addprefix=\global, 1=w]{\mathbb{E}_{#1}\left[#2\right]}%

\newcommandx\heightened[3][usedefault, addprefix=\global, 1=X, 2=H, 3=L]{\smallparens{#1,#2,#3}}%

\global\long\def\inducedGamma#1{\Gamma\smallparens{#1}}%

\global\long\def\inducedprimeGamma#1{\tilde{\Gamma}\smallparens{#1}}%

\global\long\def\Gammaprime{\tilde{\Gamma}}%

\global\long\def\Dprime{\tilde{D}}%

\global\long\def\s{s}%

\newcommandx\sn[3][usedefault, addprefix=\global, 1=\lambda, 2=\mu, 3=n]{s_{#1,#2}^{\smallparens{#3}}}%

\title{Cycle Counting and Character Expectations Using Alternating Structures}
\author{Noam Ta Shma}
\maketitle
\begin{abstract}
Recently, three related papers \cite{cassidy,Hanany2020,MageeDeLaSalle2024}
found a connection between two subjects: the $w$-cycle theorem, which
is a theorem about counting appearances of cycles reading out a word
$w$ in certain graphs, and character expectations on word measures.

The $w$-cycle theorem was proven independently by \cite{Louder_Wilton_2017}
using stackings and by \cite{helferwise} using bislim structures.
In the current work, we generalize stackings and bislim structures
to alternating stackings and alternating bislim structures. We show
how this significantly strengthens the $w$-cycle theorem for words
admitting such alternating structures, and as a result, also strengthens
the recent results of \cite{cassidy} and \cite{MageeDeLaSalle2024}.
We show that generic words admit alternating bislim structures, and
therefore, the strengthened results hold for generic words.

Using our new machinery, we address conjectures of Wilton, of Hanany--Puder
and of Puder--Shomroni. We prove that all three conjectures hold
for generic words, but we also find counterexamples for the first
two.

\tableofcontents{}
\end{abstract}

\section{Introduction}

The $w$-cycle theorem was originally conjectured to hold in \cite[Conjecture 1.1]{Wise2005}
as part of a program to prove that one relator groups are coherent,
a goal which was recently achieved in \cite{JaikinZapirain2023}.
This theorem was proven independently by \cite{Louder_Wilton_2017}
using ``stackings'' and by \cite{helferwise} using ``bislim structures''.

The $w$-cycle theorem was first connected to word measures in \cite{Hanany2020};
more recently, \cite{cassidy} and \cite{MageeDeLaSalle2024}\footnote{To be precise, a primitive version of the $w$-cycle theorem was used
in \cite{MageeDeLaSalle2024}. We show that the full $w$-cycle theorem
could have been used.} strongly connected the two subjects.

In the current work, we generalize stackings and bislim structures
to ``alternating stackings'' and ``alternating bislim structures''.
We show how this significantly strengthens the $w$-cycle theorem,
and as a result, also strengthens the recent results of \cite{cassidy}
and \cite{MageeDeLaSalle2024}. Additionally, we show that almost
all words admit alternating bislim structures, and therefore the strengthened
results hold for almost all words.

\subsection{\label{subsec:intro:The-w-cycle-theorem}The $w$-cycle theorem}

We first explain the statement of the $w$-cycle theorem. Let $\Sigma$
be a finite alphabet. A \emph{core graph} is a directed graph without
any leaves or isolated vertices, with edges labeled by an alphabet
$\Sigma$, such that at each vertex of $\Gamma$ there is at most
one outgoing edge of each label, and at most one incoming edge of
each label (see Figure \ref{fig:core_graph}).\begin{wrapfigure}[19]{O}{0.25\columnwidth}%
\centering{}\includegraphics[width=0.2\columnwidth]{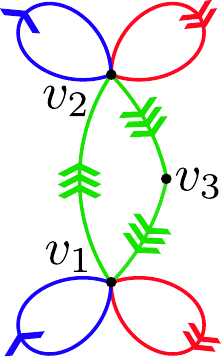}\caption{\label{fig:core_graph}A core graph. The labels $a$, $b$ and $c$
are drawn using one, two and three arrows in blue, red and green,
respectively.}
\end{wrapfigure}%

A word over an alphabet $\Sigma$ is a word $w=x_{1}\dots x_{n}$
where $x_{i}\in\Sigma\cup\Sigma^{-1}$ and $\Sigma^{-1}$ is the set
of inverses of the elements of $\Sigma$. In the following, $w$ is
always a nontrivial cyclically reduced word.\footnote{A word $w$ is \emph{cyclically reduced} if $x_{i}\neq x_{i+1}^{-1}$
for all $i$ and $x_{1}\neq x_{n}^{-1}$, e.g., $w=aba^{-1}b^{-1}$.}

\begin{defn}
Let $w$ be a nontrivial cyclically reduced word. A \emph{$\mathemph w$-cycle}
$\ell$ in a core graph $\Gamma$ is a cycle in $\Gamma$ spelling
$w^{k}$ which is not a repetition of a smaller $w$-cycle. A $w$-cycle
need not be a simple cycle. We say that $\ell$ is of degree $k$,
and we denote by $\mathemph{\num{\Gamma}}$ the number of $w$-cycles
in $\Gamma$ counted by their degrees.
\end{defn}

For example, if we take the graph $\Gamma$ shown in Figure \ref{fig:core_graph}
and $w=a^{2}b^{2}c^{3}$, then $\Gamma$ has two $w$-cycles, starting
on $v_{1}$ and $v_{2}$, and $\num{\Gamma}=2$. If we take the same
graph and $w=c$, then $\Gamma$ has one degree-3 $w$-cycle, so $\num{\Gamma}=3$.
Wise \cite[Conjecture 1.1]{Wise2005} conjectured that if $w$ is
not a proper power, the number of $w$-cycles in $\Gamma$ can be
bounded in terms of its Euler characteristic $\chi\smallparens{\Gamma}=\left|V\smallparens{\Gamma}\right|-\left|E\smallparens{\Gamma}\right|$.
This conjecture was proven independently in \cite{Louder_Wilton_2017}
and in \cite{helferwise}. One formulation of the theorem is as follows: 
\begin{thm}[{The $w$-cycle Theorem, \cite{helferwise} and \cite[Theorem~2]{Louder_Wilton_2017}}]
\label{thm:the_w_cycle_theorem}Let $\Gamma$ be a $\Sigma$-labeled
core graph, and let $w$ be a nontrivial cyclically reduced word over
$\Sigma$ which is not a proper power. Assume that each edge of $\Gamma$
is covered by $w$-cycles at least twice. Then 
\[
\num{\Gamma}\ \le\ -\chi\smallparens{\Gamma}.
\]
\end{thm}

Edges of $\Gamma$ are considered to be covered twice if they are
either covered by two distinct $w$-cycles or covered twice by the
same $w$-cycle. 

\subsection{\label{subsec:intro:Stackings-and-bislim-structures}Stackings and
bislim structures}

There are two kinds of structures which can be used to prove the $w$-cycle
theorem, namely stackings and bislim structures, which serve parallel
roles in \cite{Louder_Wilton_2017} and \cite{helferwise} respectively.
In fact, in an upcoming work we show that bislim structures are equivalent
to a variant of stackings \cite{partial_stackings_and_bislim_structures}.\footnote{The paper \cite{Bamberger2024} claims that stackings are equivalent
to bislim structures, but it contains an error, as we discuss in \cite{partial_stackings_and_bislim_structures}.} We find that bislim structures are better suited for proving the
theorems of Section \ref{sec:Existence-of-bislim-structures}, so
we chose to work with bislim structure in the main body of the paper.
However, we find that stackings are more intuitive, and so we chose
to focus on stackings in the introduction.

Let $w$ be a cyclically reduced word over an alphabet $\Sigma$,
let $S_{w}$ be a cycle graph spelling $w$ and let $\Omega$ be a
bouquet of circles consisting of one vertex and $\left|\Sigma\right|$
edges --- one edge per element of $\Sigma$, so there is a canonical
map $\eta:S_{w}\to\Omega$ given by the labeling.
\begin{defn}[Stackings, geometric definition, \cite{Louder_Wilton_2017}]
 A \emph{stacking} of $w$ is a continuous height map $h:S_{w}\to\mathbb{R}$
such that the combined function $\widehat{\eta}:S_{w}\hookrightarrow\Omega\times\mathbb{R}$
defined by $\widehat{\eta}\smallparens p=\smallparens{\eta\smallparens p,h\smallparens p}$
for every point $p\in S_{w}$ is injective.
\end{defn}

Visually, the cycle $S_{w}$ is embedded in the space $\Omega\times\mathbb{R}$,
such that for each edge $e$ in $S_{w}$ labeled by a letter $\sigma\in\Sigma$,
when $e$ is projected into $\Omega$ it is projected to the edge
of $\Omega$ corresponding to $\sigma$, as in Figure \ref{fig:stacking}.

\begin{wrapfigure}[26]{O}{0.2\columnwidth}%
\begin{centering}
\includegraphics[width=0.2\columnwidth]{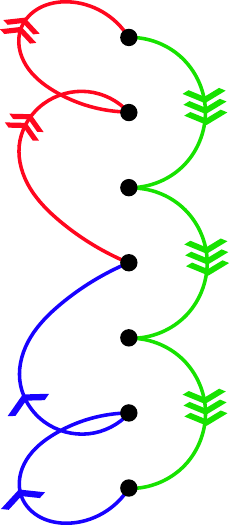}\caption{\label{fig:stacking}A stacking of $w=a^{2}b^{2}c^{3}$. Drawn as
the embedding of $S_{w}$ in $\Omega\times\mathbb{R}$, projected
to the page where the vertical axis represents height.}
\par\end{centering}
\end{wrapfigure}%

For a given directed edge $e$, denote by $\mathemph{\iota\smallparens e}$
and $\mathemph{\tau\smallparens e}$ the source and target vertices
of $e$ respectively. Let $e_{1}$ and $e_{2}$ be edges in $S_{w}$
where $\eta\smallparens{e_{1}}=\eta\smallparens{e_{2}}$, i.e., $e_{1}$
and $e_{2}$ are labeled by the same letter of $\Sigma$. It follows
from the intermediate value theorem and

the injectivity of $\widehat{\eta}$ that if $h\smallparens{\iota(e_{1})}<h\smallparens{\iota(e_{2})}$
if and only if $h\smallparens{\tau\smallparens{e_{1}}}<h\smallparens{\tau\smallparens{e_{2}}}$.
Thus by forgetting the exact height map $h$, and only retaining the
ordering of the vertices by their heights, we get an equivalent combinatorial
definition of stackings:
\begin{defn}[Stackings, combinatorial definition, \cite{Louder_Wilton_2017}]
\label{def:stackings_combinatorial_definition} A \emph{stacking}
of a word $w$ is an ordering $\preceq$ of the vertices of $S_{w}$
such that for any two edges $e_{1}$ and $e_{2}$ in $S_{w}$ which
are labeled by the same label, we have that
\[
\iota\smallparens{e_{1}}\preceq\iota\smallparens{e_{2}}\iff\tau\smallparens{e_{1}}\preceq\tau\smallparens{e_{2}}.
\]
\end{defn}

The w-cycle theorem is proven in \cite{Louder_Wilton_2017} in the
following two parts:
\begin{thm}[\cite{Louder_Wilton_2017}]
\label{thm:louder_wilton_non_power_words_admit_stackings} A word
$w\neq1$ admits a stacking if and only if $w$ is not a proper power.
\end{thm}

\begin{thm}[\cite{Louder_Wilton_2017}]
 Let $\Gamma$ be a $\Sigma$-labeled core graph, and let $w\in F_{\Sigma}$
be a word in $\Sigma$. Assume that each edge of $\Gamma$ is covered
by $w$-cycles at least twice. If $w$ admits a stacking then 
\[
\num{\Gamma}\ \le\ -\chi\smallparens{\Gamma}.
\]
\end{thm}

In an upcoming joint work with Daniel Wise and Jonah Gaster \cite{partial_stackings_and_bislim_structures},
we define partial stackings, and prove that ``good'' partial stackings
are equivalent to bislim structures. Partial stackings are defined
the same way as in Definition \ref{def:stackings_combinatorial_definition}
except that $\preceq$ is allowed to be a partial order:
\begin{defn}[Partial stackings, \cite{partial_stackings_and_bislim_structures}]
\label{def:partial_stackings_combinatorial_definition} A \emph{partial
stacking} of a word $w$ is a \emph{partial} ordering $\preceq$ of
the vertices of $S_{w}$ such that for any two edges $e_{1}$ and
$e_{2}$ in $S_{w}$ which are labeled by the same label, we have
that
\[
\iota\smallparens{e_{1}}\preceq\iota\smallparens{e_{2}}\iff\tau\smallparens{e_{1}}\preceq\tau\smallparens{e_{2}}.
\]
\end{defn}

\subsection{\label{subsec:intro:Word-measures}Word measures and character expectations}

Let $w\in F_{r}$ be a word in the free group with $r$ generators,
and let $G$ be a compact group. A $\mathemph w$\emph{-random} group
element is obtained by substituting each free generator of $F_{r}$
by an independent, uniformly random group element and evaluating $w\smallparens{g_{1},\dots,g_{r}}$.
The resulting probability measure is called a word measure. To study
word measures, we consider the expectations of characters\footnote{That is, a trace of a complex representation of $G$.}
$\psi$ over a $w$-random element, i.e., 
\[
\expectation{\psi}=\expectation[g_{i}\in G]{\psi\smallparens{w\smallparens{g_{1},\dots,g_{r}}}}.
\]
We focus on the cases where $G$ is either a symmetric group $S_{n}$
or a unitary group $\U n$. We point the reader to \cite{Puder2023}
for additional results and conjectures of this type including additional
families of groups. 

\subsubsection{Stable characters}

Foundational works \cite{Puder2015,MageePuder2019}  in this field
started by studying the specific cases of $\psi_{n}=\text{\#fix}$
in $S_{n}$ and $\psi_{n}=\text{Tr}$ in $\U n$, which are the standard
representations of their respective groups. In each of these families
the characters are almost ``the same'' across all values of $n$
in many different respects, and one way in which this manifests is
the fact that $\mathbb{E}_{w}\left[\psi_{n}\right]$ is a rational
function in $n$ in both cases. This can be generalized to ``stable
characters'', which are families of characters $\left\{ \psi_{n}\right\} _{n}$
where all $\psi_{n}$ are essentially ``the same'' across all large
enough values of $n$. In the cases of $S_{n}$ and $\U n$, a \emph{stable
character} is given as a sum of \emph{irreducible stable characters},
which can be characterized as follows.

In the case of $S_{n}$, each irreducible stable character is characterized
by an integer partition: Given an integer partition $\lambda=\smallparens{\lambda_{1},\dots,\lambda_{p}}\vdash d$,
if $n\ge\lambda_{1}+d$ we denote by $\mathemph{\lambda^{+}\smallparens n}$
the partition $\smallparens{n-d,\lambda_{1},\dots,\lambda_{p}}$.
We denote by $\mathemph{\chi^{\lambda^{+}\smallparens n}}$ the unique
irreducible character of $S_{n}$ corresponding to the partition $\lambda^{+}\smallparens n$.
Taking this for all $n\ge\lambda_{1}+d$ gives an irreducible stable
character, and every irreducible stable character is of this form.
Every irreducible stable character $\psi_{n}=\chi^{\lambda^{+}\smallparens n}$
is of dimension $\dim\psi_{n}=\Theta\largeparens{n^{d}}$ by the hook
length formula. In \cite{Hanany2020} it is shown that $\expectation{\psi_{n}}$
is always a rational function in $n$ for large enough $n$.

In the case of $\U n$, each irreducible stable character is characterized
by two integer partitions $\lambda$ and $\mu$: Given two integer
partitions $\lambda=\smallparens{\lambda_{1},\dots,\lambda_{p}}$
and $\mu=\smallparens{\mu_{1},\dots,\mu_{q}}$ where $n\ge p+q$,
we denote by $\mathemph{\sn}$ the unique irreducible character of
$\U n$ with dominant weight vector
\[
\left(\lambda_{1},\dots,\lambda_{p}\underset{n-p-q\text{ zeros}}{,\underbrace{0,\dots,0},}-\mu_{q},\dots,-\mu_{1}\right)
\]
(see \cite[Chapters~19,~21]{Bump2013}) where there are $n-p-q$ zeros.
Taking this for all $n\ge p+q$ gives an irreducible stable character,
and every irreducible stable character is of this form. Every irreducible
stable character $\psi_{n}=\sn$ where $\lambda\vdash k$ and $\mu\vdash\ell$
is of dimension $\dim\psi=\Theta\left(n^{k+\ell}\right)$ (see \cite[Equation~3.10]{Puder2023}).
It follows from \cite{MageePuder2019} that $\expectation{\psi_{n}}$
is always a rational function in $n$.

The theory of stable characters can be generalized to additional stable
families of groups, see \cite{Church2013,PuderShomroni2025,Puder2023,Sam2015}.

\subsubsection{Asymptotic growth conjectures}

We focus on the asymptotic behavior of $\expectation{\psi_{n}}$ as
$n$ tends to infinity. In the case of $\U n$, the asymptotic growth
rate is conjectured to be as follows:
\begin{conjecture}[{The asymptotic growth conjecture for $\U n$ (\cite[Conjecture~5.7]{Puder2023})}]
\label{conj:asymptotic_growth_conjecture_Un}Let $\psi_{n}$ be an
irreducible stable character of $\U n$, and let $w$ be a word. Then
\[
\expectation{\psi_{n}}=O\largeparens{\smallparens{\dim\psi_{n}}^{-\ssql\smallparens w}},
\]
where $\ssql\smallparens w$ is the stable square length of $w$ (\cite[Definition~5.1]{Puder2023}).
\end{conjecture}

We have similar conjectures for $S_{n}$, though for $S_{n}$ the
history is more complicated. Originally, the following was conjectured:
\begin{conjecture}[{The original asymptotic growth conjecture for $S_{n}$ (\cite[Conjecture~1.8]{Hanany2020})}]
\label{conj:asymptotic_growth_conjecture_Sn_pi}Let $\psi_{n}$ be
an irreducible stable character of $S_{n}$ or $\U n$, and let $w$
be a word. Then
\[
\expectation{\psi_{n}}=O\largeparens{\smallparens{\dim\psi_{n}}^{1-\pi\smallparens w}},
\]
where $\pi\smallparens w$ is the primitivity rank of $w$ (\cite[Definition~1.7]{Puder2011}).
\end{conjecture}

Later, Wilton defined the stable primitivity rank $\spi\smallparens w$
of a word $w$, an invariant which is closely related to the $w$-cycle
theorem and the primitivity rank $\pi\smallparens w$ (\cite[Definition~10.6]{wilton_curvature_invariants}).
By their definitions, $\spi\smallparens w\le\pi\smallparens w-1$
for all $w$. Wilton conjectured that in fact, they are always equal:
\begin{conjecture}[{Wilton, see \cite[Conjecture~4.7]{Puder2023}}]
\label{conj:spi_equals_pi-1}For all words $w$, $\spi\smallparens w=\pi\smallparens w-1$.
\end{conjecture}

The definition of $\spi$ enabled \cite{Puder2023} to split Conjecture
\ref{conj:asymptotic_growth_conjecture_Sn_pi} into Conjecture \ref{conj:spi_equals_pi-1}
and the following conjecture:
\begin{conjecture}[{The corrected asymptotic growth conjecture for $S_{n}$ (\cite[Conjecture~1.2]{Puder2023})}]
\label{conj:asymptotic_growth_conjecture_Sn_spi}Let $\psi_{n}$
be an irreducible stable character of $S_{n}$, and let $w$ be a
word. Then
\[
\expectation{\psi_{n}}=O\largeparens{\smallparens{\dim\psi_{n}}^{-\spi\smallparens w}}.
\]
\end{conjecture}

Puder and Shomroni \cite{Puder2023} explain that Conjecture \ref{conj:asymptotic_growth_conjecture_Sn_spi}
is more natural than Conjecture \ref{conj:asymptotic_growth_conjecture_Sn_pi}.
In the current work, we prove the following:
\begin{thm}
\label{thm:generic_words}Conjectures \ref{conj:asymptotic_growth_conjecture_Sn_pi},
\ref{conj:spi_equals_pi-1} and \ref{conj:asymptotic_growth_conjecture_Sn_spi}
hold for a generic word $w$.
\end{thm}

An additional application of this work is an efficient algorithm that
may confirm the above conjectures for arbitrary words $w$ (see Section
\ref{subsec:Algorithms}). Using this algorithm we obtain:
\begin{thm}
\label{thm:conjecture_holds_for_small_words}Conjectures \ref{conj:asymptotic_growth_conjecture_Sn_pi},
\ref{conj:spi_equals_pi-1} and \ref{conj:asymptotic_growth_conjecture_Sn_spi}
hold for all words $w\in F_{4}$ of length at most $10$.
\end{thm}

However, we found that the first word for which the conjectures were
not confirmed were in fact counterexamples:
\begin{thm}
\label{thm:counterexample}Conjectures \ref{conj:asymptotic_growth_conjecture_Sn_pi}
and \ref{conj:spi_equals_pi-1} do not hold for $w=\left(abab^{-1}\right)^{2}cac^{-1}$. 
\end{thm}

See Section \ref{subsec:Computational-experiments}. Still, we believe
that Conjecture \ref{conj:asymptotic_growth_conjecture_Sn_spi} is
indeed true.

\subsubsection{Previous partial results}

From the $w$-cycle theorem it follows that Conjecture \ref{conj:spi_equals_pi-1}
holds when $\pi\smallparens w\le2$, and in particular for all $w\in F_{2}$.
Conjecture \ref{conj:asymptotic_growth_conjecture_Sn_pi} holds for
the special case of $\psi_{n}=\#\text{fix}-1$, and consequently Conjecture
\ref{conj:asymptotic_growth_conjecture_Sn_spi} also holds in this
case (\cite[Theorem~1.8]{Puder2015}). In \cite[Corollary~1.4]{Hanany2020}
it is proven that over $S_{n}$, $\expectation{\psi_{n}}=O\largeparens{n^{1-\pi\smallparens w}}$
if $\smallparens{\psi_{n}}_{n}$ is a nontrivial irreducible stable
character.

It follows from \cite[Corollary~1.8]{MageePuder2019} that $\expectation{\psi_{n}}=O\largeparens{\smallparens{\dim\psi_{n}}^{-2\scl\smallparens w}}$
if $\psi_{n}$ is a \textit{polynomial} stable character, where $\scl\smallparens w$
is the stable commutator length of $w$, see \cite[Theorem~1.1]{Puder2023}.
This result proves Conjecture \ref{conj:asymptotic_growth_conjecture_Un}
for the case of the polynomial characters of $\U n$ since $\ssql\smallparens w\le2\scl\smallparens w$
(\cite[Lemma~10.9]{wilton_curvature_invariants}).

\global\long\def\part{\text{P}}%

Recently, \cite{cassidy,MageeDeLaSalle2024} proved striking new results
on word measures in $S_{n}$ and $\U n$, namely that $\expectation{\psi_{n}}=O\largeparens{\frac{1}{\dim\psi_{n}}}$
(\cite[Theorem~1.6]{cassidy}) in the case of $S_{n}$ and $\expectation{\psi_{n}}=\allowbreak O\largeparens{\smallparens{\dim\psi_{n}}^{-\frac{1}{6}}}$
in the case of $\U n$ (\cite[Corollary~11.3]{MageeDeLaSalle2024}),
assuming that $w$ is not a proper power. Cassidy's result proves
Conjectures \ref{conj:asymptotic_growth_conjecture_Sn_pi} and \ref{conj:asymptotic_growth_conjecture_Sn_spi}
when $\pi\smallparens w\le2$, and in particular for all $w\in F_{2}$.
Note that Conjecture \ref{conj:asymptotic_growth_conjecture_Sn_pi}
is the ``logical union'' of Cassidy's result and \cite[Corollary~1.4]{Hanany2020}.

Both results (\cite[Theorem~1.6]{cassidy} and \cite[Corollary~11.3]{MageeDeLaSalle2024})
bound character expectations $\expectation[w]{\psi_{n}}$ by reducing
the problem to versions of the $w$-cycle theorem, through representation-theoretic
computations of $\expectation[w]{\psi_{n}}$. In the following, we
show how to generalize the $w$-cycle theorem, and so by using the
same representation-theoretic computations, we obtain new bounds on
character expectations. See further discussions in Sections \ref{subsec:intro:Alternating-structures-and-new-results}
and \ref{sec:Applications-to-word-measures}.

Prior to this work, lower bounds of $\spi\smallparens w$ larger than
$1$ were rare, and very little was known about actual values of $\spi\smallparens w$
besides some special cases.\footnote{Roughly, $\spi\smallparens w$ was known when $\pi\smallparens w\le2$
or when $w$ is a product $w=w_{1}w_{2}$ where $\spi\smallparens{w_{1}}$
and $\spi\smallparens{w_{2}}$ are known and $w_{1}$, $w_{2}$ do
not have any letters in common.} The paper \cite[Theorem~A]{Wilton_rationality_theorem} shows that
$\spi\smallparens w$ is in fact rational and computable, the computation
requires double-exponential running time, and is to the best of our
knowledge intractable in practice even for relatively short words.
Additionally, from \cite{Louder_Wilton_2017,helferwise} it follows
that $\spi\smallparens w\ge1$ whenever $w$ is not a proper power,
and from \cite[Theorem~1.16]{Louder2018} together with \cite{Wilton_rationality_theorem}
it follows that $\spi\smallparens w>1$ whenever $\pi\smallparens w\ge3$.

\subsection{\label{subsec:intro:Alternating-structures-and-new-results}Alternating
structures and new results}

We generalize stackings and bislim structures into $k$-alternating
stackings and $k$-alternating bislim structures. Recall the definition
of partial stackings (Definition \ref{def:partial_stackings_combinatorial_definition}).
We generalize it as follows:
\begin{defn}
\label{def:alternating}Let $w$ be a word, $S_{w}$ be a cycle graph
spelling $w$, and let $\preceq$ be a partial stacking of $w$. We
say that an edge $e$ of $S_{w}$ is \emph{high} if $\iota\smallparens e\succ\iota\largeparens{e'}$
for each other edge $e'$ of $S_{w}$ with the same label as $e$.
Respectively, $e$ is \emph{low} if $\iota\smallparens e\prec\iota\smallparens{e'}$
for each other edge $e'$ with the same label as $e$.

We say that a partial stacking $\preceq$ of $w$ is $\mathemph k$\emph{-alternating}
if there are $k$ high edges and $k$ low edges that appear in an
alternating order around $S_{w}$.
\end{defn}

For example, the stacking of the word $w=a^{2}b^{2}c^{3}$ shown in
Figure \ref{fig:stacking} is a $2$-alternating stacking. If $w$
is a word over an alphabet $\Sigma$, it follows that $k\le\left|\Sigma\right|$,
and in fact, we will show momentarily that $k\le\left|\Sigma\right|-1$.
Every (non-partial) stacking necessarily has at least one high edge
and one low edge, and so is $1$-alternating. Thus every word $w$
which is not a proper power admits a $1$-alternating stacking by
Theorem \ref{thm:louder_wilton_non_power_words_admit_stackings}.

We similarly generalize bislim structures into $k$-alternating bislim
structures in Section \ref{sec:Alternating-bislim-structures} (Definitions
\ref{def:bislim_structure_for_good_complex} and \ref{def:alternating_heightened_complex}).
We note that $k$-alternating partial stackings are equivalent to
$k$-alternating bislim structures, as we prove in an upcoming work
\cite{partial_stackings_and_bislim_structures}. We show that $k$-alternating
partial stackings and $k$-alternating bislim structures can be used
to strengthen the $w$-cycles theorem as follows:
\begin{thm}
\label{thm:w-cycle-theorem-alternating}Let $\Gamma$ be a $\Sigma$-labeled
core graph, and let $w$ be a cyclically reduced word admitting a
$k$-alternating partial stacking (or a $k$-alternating bislim structure).
Assume that each edge of $\Gamma$ is covered at least twice by $w$-cycles.
Then
\[
k\cdot\num{\Gamma}\ \le\ -\chi\smallparens{\Gamma}.
\]
\end{thm}

As an example, if we take the graph $\Gamma$ from Figure \ref{fig:core_graph},
the word $w=a^{2}b^{2}c^{3}$ and its $2$-alternating stacking from
Figure \ref{fig:stacking}, we have equality in Theorem \ref{thm:w-cycle-theorem-alternating}.
Theorem \ref{thm:w-cycle-theorem-alternating} bounds $\spi$ from
below, the same way that $\spi\smallparens w\ge1$ for non powers
$w$ follows from Theorem \ref{thm:the_w_cycle_theorem} (see the
discussion preceding \cite[Theorem~4.4]{Puder2023}):
\begin{cor}
\label{cor:alternating_bislim_structures_bound_stable_primitivity_rank}If
a word $w$ admits a $k$-alternating partial stacking, then $\spi\smallparens w\ge k$.
\end{cor}

We also apply Theorem \ref{thm:w-cycle-theorem-alternating} to the
results of \cite{cassidy} and \cite{MageeDeLaSalle2024}, and obtain
the following:
\begin{thm}
\label{thm:alternating_both_cases}Let $w$ be a word admitting a
$k$-alternating partial stacking (or a $k$-alternating bislim structure)
and let $\psi_{n}$ be an irreducible stable character of $S_{n}$
or $\U n$. Then
\[
\expectation{\psi_{n}}=O\largeparens{\smallparens{\dim\psi_{n}}^{-k}}.
\]
\end{thm}

This leaves open the question of the existence of alternating partial
stackings, and therefore, the applicability of these results. We say
that a property holds for a \emph{generic} word $w$ over a finite
alphabet $\Sigma$ if the probability that the property holds for
a uniformly random word $w$ of length $\ell$ tends to $1$ as $\ell$
tends to infinity.
\begin{thm}
\label{thm:generic_bislim_structures_one_relator_case}A generic word
$w$ over an alphabet $\Sigma$ admits a $\smallparens{\left|\Sigma\right|-1}$-alternating
partial stacking (and a $\smallparens{\left|\Sigma\right|-1}$-alternating
bislim structure).
\end{thm}

An alternation degree of $\left|\Sigma\right|-1$ is optimal: by letting
$\Gamma$ be the bouquet of circles with one edge for each element
of $\Sigma$ in Theorem \ref{thm:w-cycle-theorem-alternating} we
obtain that $k\le\left|\Sigma\right|-1$, unless some letter of $\Sigma$
occurs at most once in $w$. In that case, there cannot be $2\left|\Sigma\right|$
high and low edges: there is at most one high and one low edge for
each label in $\Sigma$, and that label has at most one corresponding
edge. 

Additionally, by \cite{Puder2011} $\spi\smallparens w\le\pi\smallparens w-1\le\left|\Sigma\right|-1$,
unless $w$ belongs to a free basis of the free group generated by
$\Sigma$, in which case $s\pi\smallparens w=\infty$. Since a generic
word does not belong to a free basis (\cite{Burillo2002}), Theorem
\ref{thm:generic_words} follows directly from Corollary \ref{cor:alternating_bislim_structures_bound_stable_primitivity_rank}
and Theorems \ref{thm:alternating_both_cases} and \ref{thm:generic_bislim_structures_one_relator_case}.

\subsection{\label{subsec:intro:Organization}Organization}

In Section \ref{sec:Alternating-bislim-structures} we introduce $k$-alternating
bislim structures (Definitions \ref{def:bislim_structure_for_good_complex}
and \ref{def:alternating_heightened_complex}), develop their theory,
and prove the strengthened versions of the $w$-cycle theorem. In
Section \ref{sec:Existence-of-bislim-structures} we show that random
words have optimally alternating bislim structures with high probability
(Theorem \ref{thm:generic_bislim_structures_one_relator_case}). In
Section \ref{sec:Applications-to-word-measures} we prove Theorem
\ref{thm:alternating_both_cases}, bounding character expectations
on $w$-measures in the symmetric and unitary groups. In Section \ref{sec:nonpositive_and_negative_immersions}
we apply $k$-alternating bislim structures back to geometric group
theory, the field where bislim structures originated from. We show
how alternating bislim structures can be used to show that a 2-complex
$X$ has negative immersions, that it is hyperbolic and coherent,
and to bound a curvature invariant of $X$ due to Wilton which generalizes
$\spi$. In Section \ref{subsec:Algorithms} we discuss how alternating
bislim structures can be efficiently computed, and therefore can be
used to certify that a 2-complex $X$ is hyperbolic and has negative
immersions. In Section \ref{subsec:Computational-experiments} we
use this to algorithmically prove Theorem \ref{thm:conjecture_holds_for_small_words}.

Our definition of bislim structures (Definition \ref{def:bislim_structure_for_good_complex})
is different than the original definition in \cite[Definition~2.1]{helferwise}.
Appendix \ref{sec:equivalence-to-bi-slim} shows that the definitions
are equivalent.

\section*{Acknowledgments}

I am deeply grateful for my advisor, Prof.\ Doron Puder, for his
invaluable guidance, deep insight, and for the considerable time and
effort devoted to improving this paper. I am sincerely grateful to
Yotam Shomroni for our ongoing collaboration and the many deep discussions
we have had. I also thank Prof. Daniel T. Wise for his insights and
advice, and Nir Lazarovich for carefully reading the paper and providing
helpful comments. Finally, I thank my father for his encouragement
and support. This work was supported by the European Research Council
(ERC) under the European Union’s Horizon 2020 research and innovation
programme (grant agreement No 850956) as well as by the Israel Science
Foundation, ISF grant 1140/23.

\section{\label{sec:Alternating-bislim-structures}Alternating bislim structures}

We start working our way to defining bislim structures. Most of the
content in this section is adapted from Helfer and Wise \cite{helferwise},
which introduced the original definition of bislim structures.

\subsection{\label{subsec:Combinatorial-complexes-and}Combinatorial complexes
and maps}

\begin{wrapfigure}[17]{O}{0.35\columnwidth}%
\begin{centering}
\includegraphics[width=0.35\textwidth]{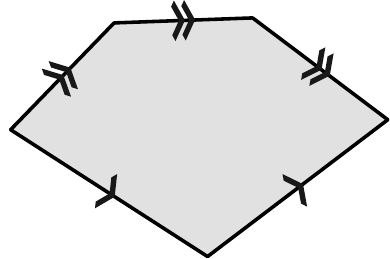}
\par\end{centering}
\caption{\label{fig:presentation_complex}Consider the presentation $\left\langle x,y\vert x^{2}y^{-3}\right\rangle $.
Its presentation complex is given by gluing edges with the same label
on the diagram. This complex is a combinatorial 2-complex.}
\end{wrapfigure}%

A 2-dimensional CW-complex (a \emph{2-complex}) $X$ is given by its
\emph{1-skeleton} $\mathemph{\skeleton X}$ which is a 1-complex (a
graph), and a collection of 2-cells $C_{i}$ which are topological
disks with an attaching map $\boundary C_{i}\to\skeleton X$ from
their boundary $\boundary C_{i}$ which is a topological circle to
the 1-skeleton. Gluing the 2-cells to $\skeleton X$ along the attaching
maps gives the 2-complex $X$.

Let $X$ be a 2-complex. We call its 0-cells \emph{vertices }and its
1-cells \emph{edges}. We denote the \emph{universal cover} of $X$
by $\mathemph{\universalCover X}$.

A map of 2-complexes $Y\to X$ is \emph{combinatorial} if every open
cell of $Y$ is homeomorphically mapped to an open cell of $X$. A
2-complex $X$ is \emph{combinatorial} if for each 2-cell $C$ of
$X$, the boundary $\boundary C$ has the structure of a 1-complex
such that the attaching map $\boundary C\to\skeleton X$ is a combinatorial
map (e.g., Figure \ref{fig:presentation_complex}). 

A map $Y\to X$ is a \emph{branched combinatorial map} if every vertex
is mapped to a vertex, every open edge is mapped homeomorphically
to an open edge, and every open 2-cell $C$ of $Y$ is mapped into
an open 2-cell $C'$ of $X$ such that the map $C\to C'$ is a branched
map, i.e., locally injective at every interior point of $C$ other
than the center, similarly to the maps $z\mapsto z^{k}$ from the
complex unit disk to itself. 

If a complex $X$ contains a 2-cell $C$ such that $\boundary C$
gives a contractible loop in $\skeleton X$, we say that $X$ is \emph{degenerate}.
All complexes in this paper are non-degenerate combinatorial 2-complexes,
and all maps of complexes in this paper are branched combinatorial
maps.

\subsubsection{Sides}

Throughout this paper we will use the notion of a \emph{side} of an
edge quite a lot. A naive but wrong definition is that a side $s$
is a tuple $\smallparens{C,e}$ of a 2-cell $C$ and an adjacent edge
$e$, see Figure \ref{fig:book_with_three_sides}. However, Figure
\ref{fig:side_edgecase} is an edgecase where two different sides
correspond to the same pair $\smallparens{C,e}$.

Formally, a \emph{side} $s$ of $X$ is an edge of $\boundary C$
for some 2-cell $C$ of $X$. We remind the reader that $\boundary C$
is \emph{not} just the set of edges which are adjacent to $C$, but
rather a topological circle that is mapped into $\skeleton X$ by
its attaching map. Since $X$ is a combinatorial complex, $\boundary C$
also has the structure of a graph. Indeed, the two different sides
drawn in Figure \ref{fig:side_edgecase} are two different edges $s_{1}$
and $s_{2}$ of $\boundary C$.

We draw sides as markings on $C$ adjacent to their edges (see Figure
\ref{fig:sides}). Each side $s$ corresponds to an edge $e$ and
a 2-cell $C$, and we say that $s$ is a side of $e$ and $e$ has
a side $s$. We also say that $s$ is a side of $C$ and $C$ has
a side $s$.

Maps of complexes also induce maps of their sides, i.e., given a map
$\phi:Y\to X$ and a side $s$ of $Y$ we have that $\phi\smallparens s$
is a side of $X$. We say that an edge of $X$ is \emph{internal}
if it has at least two sides in $X$, and a 2-cell of $X$ is \emph{internal}
if all of the edges adjacent to it are internal.

\begin{figure}
\subfloat[\label{fig:book_with_three_sides}A complex made from three rectangles
glued together on an edge $e$, which has three sides, where $s_{3}$
is hidden behind a 2-cell.]{\begin{centering}
\includegraphics[width=0.3\textwidth]{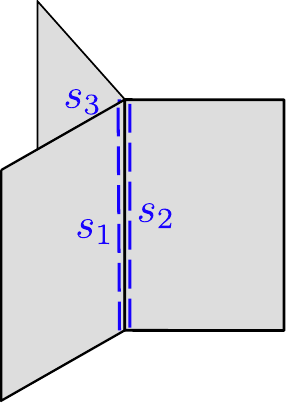}
\par\end{centering}
}\hfill{}\subfloat[\label{fig:side_edgecase}The sides $s_{1}$ and $s_{2}$ are different,
even though they are sides of the same edge and of the same 2-cell.]{\begin{centering}
\includegraphics[width=0.35\textwidth]{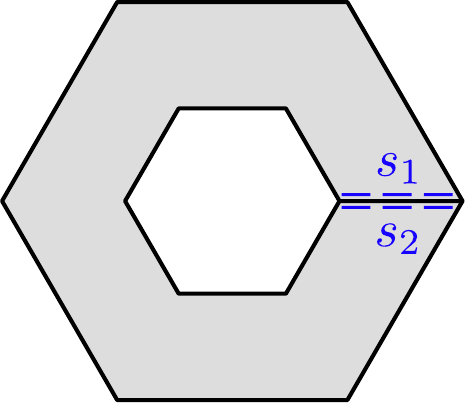}
\par\end{centering}
}\caption{\label{fig:sides}Figures demonstrating the concept of sides. Sides
are marked by blue striped lines. }
\end{figure}

A branched combinatorial map $\phi:Y\to X$ is a \emph{branched near-immersion}
if $\phi$ is locally injective at every point of an open edge (see
also \cite[Definition~5.2]{Gaster2018}). Equivalently, any two different
sides $s\neq s'$ of the same edge in $Y$ are mapped to two different
sides $\phi\smallparens s\neq\phi\smallparens{s'}$ in $X$. See Figure
\ref{fig:reduced_disk_diagram} for an example and Figure \ref{fig:non_near_immersion}
for a non-example. The map $\phi$ might not be locally injective
at vertices or at the centers of 2-cells. In this paper, almost all
maps between complexes are branched near-immersions, but this will
always be stated explicitly.

\begin{figure}
\begin{centering}
\begin{minipage}[t]{0.45\columnwidth}%
\begin{center}
\includegraphics[width=0.8\columnwidth]{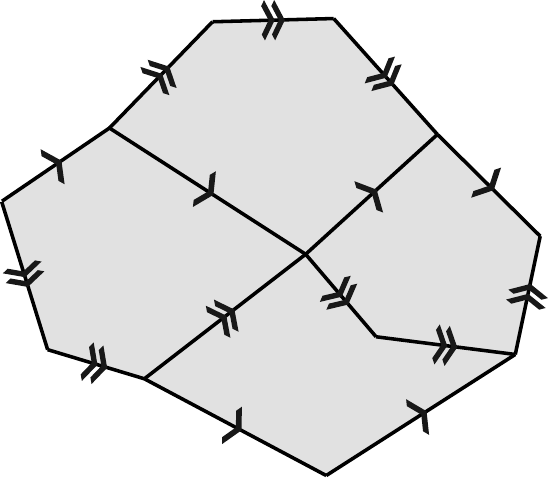}\caption{\label{fig:reduced_disk_diagram}A disk diagram $D$ with a branched
near-immersion $\phi:D\to X$ where $X$ is from Figure \ref{fig:presentation_complex},
and every edge is marked by its image in $X$. Thus $D$ is a reduced
disk diagram $D$ in $X$.}
\par\end{center}%
\end{minipage}\hfill{}%
\begin{minipage}[t]{0.45\columnwidth}%
\begin{center}
\includegraphics[width=0.8\columnwidth]{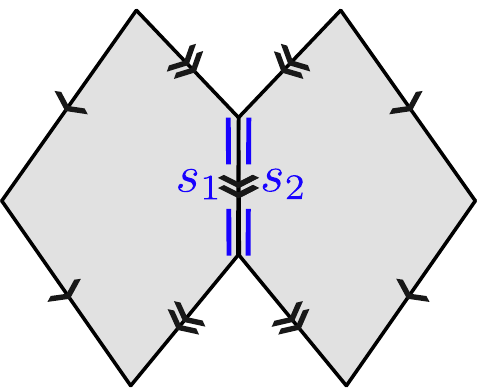}\caption{\label{fig:non_near_immersion}This complex $Y$ has a map $\phi:Y\to X$
where $X$ is from Figure \ref{fig:presentation_complex}, and every
edge is marked by its image in $X$. The map $\phi$ is not a branched
near-immersion since two sides $s_{1}$ and $s_{2}$ of the same edge
are mapped to the same side of $X$. Thus $Y$ is a non-reduced disk
diagram in $X$.}
\par\end{center}%
\end{minipage}
\par\end{centering}
\end{figure}

\subsubsection{Disk diagrams}

A \emph{disk diagram,} also known as a Van-Kampen diagram, is a compact
contractible planar 2-complex. Let $X$ be a complex. A \emph{disk
diagram} $\mathemph D$\textbf{ in }$\mathemph X$ is a disk diagram
$D$ with a combinatorial branched map $\phi:D\to X$. A disk diagram
$D$ in $X$ is \emph{reduced} if $\phi:D\to X$ is a branched near-immersion.
(See Figure \ref{fig:counterexampl_without_gamma} for an example
of a reduced disk diagram in $X$.)

This definition of a reduced disk diagram is exactly equivalent to
the standard definition of a reduced disk diagram, with the addition
of allowing branched 2-cells: given that $D$ is a standard disk diagram
in $X$ (i.e., $D$ has no branched 2-cells), two sides $s$ and $s'$
of the same edge $e$ map into the same side in $X$ if and only if
their 2-cells $C$ and $C'$ are a reduction pair (see Figure \ref{fig:non_near_immersion}).

\subsection{Bislim structures and heightened complexes}

Helfer and Wise defined bislim structures (\cite[Definition 2.1]{helferwise}).
We introduce a new definition of bislim structures (Definition \ref{def:bislim_structure_for_good_complex}
below) which gives a slightly different notion. Below, in this subsection
and in Appendix \ref{sec:equivalence-to-bi-slim} we explain the connection
between our definition and the original one. Our definition allows
us to consider alternating bislim structures (Definition \ref{def:alternating_heightened_complex}
below), which give stronger bounds in the $w$-cycle theorem (Theorem
\ref{thm:the_w_cycle_theorem_extended} below). Additionally, we find
our definition of bislim structures elegant and symmetric and thus
we believe that it is also valuable on its own.

We now begin discussing bislim structures. The data required to define
bislim structures is a heightened complex:
\begin{defn}
\label{def:heightened_complexes}We call $\heightened$ a \emph{heightened
complex }if $X$ is a 2-complex and $H$ and $L$ are disjoint sets
of sides of $X$. We call the sides in $H$ \emph{high sides}, the
sides in $L$ \emph{low sides}, and the sides in $H\cup L$ \emph{extremal
sides}. The extremal sides of a 2-cell $C$ are $H_{C}=\boundary C\cap H$
and $L_{C}=\boundary C\cap L$. Some examples are illustrated in Figure
\ref{fig:examples-of-heightened-complexes}.

We say that $\heightened$ is \emph{good}\footnote{This concept is taken from \cite{Louder_Wilton_2017}.}
if $H_{C}$ and $L_{C}$ are nonempty for every 2-cell $C$.
\end{defn}

\begin{figure}
\subfloat[\label{fig:bislim_complex}A bislim heightened complex]{\includegraphics[scale=0.66]{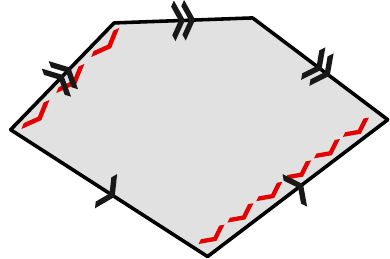}

}\hfill{}\subfloat[\label{fig:non_bislim_complex}A non-bislim heightened complex (see
Figure \ref{fig:counterexample})]{\includegraphics[scale=0.66]{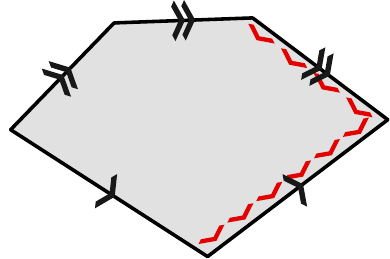}

}\hfill{}\subfloat[\label{fig:non_bislim_alternating}A 2-alternating heightened complex]{\includegraphics[scale=0.66]{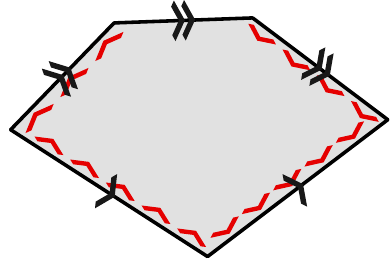}

}\caption{\label{fig:examples-of-heightened-complexes}Examples of heightened
complexes $\protect\heightened$ where $X$ is taken from Figure \ref{fig:presentation_complex}.
We mark extremal sides with red arrow like markings on the side of
the corresponding edge. The high sides $H$ point out of their cell
and the low sides $L$ point into their cell.}
\end{figure}

\global\long\def\alt{\operatorname{alt}}%

The original definition of bislim structures (\cite[Definition 2.1]{helferwise})
requires that the corresponding heightened complex is good, and almost
all heightened complexes in this paper will be good. Therefore we
focus on good heightened complexes. However, sometimes heightened
complexes in which $H_{C}$ and $L_{C}$ are empty for some 2-cells
$C$ can still give nontrivial results, e.g., in Theorem \ref{thm:w_cycles_complex_version_two_complexes}
and Section \ref{sec:nonpositive_and_negative_immersions}. 

Every branched near-immersion $\phi:Y\to X$ where $X$ is a heightened
complex $\heightened$ induces on $Y$ the structure of a heightened
complex $\heightened[Y][\phi^{-1}\smallparens H][\phi^{-1}\smallparens L]$.
\begin{example}
Let $X$ be the presentation complex of $\left\langle x,y\vert x^{2}y^{-3}\right\rangle $
as in Figure \ref{fig:presentation_complex}, which is given a heightened
structure $\heightened$ as in Figure \ref{fig:non_bislim_complex}.
In Figure \ref{fig:reduced_disk_diagram} we see a disk diagram $D$
with a branched near-immersion $\phi:D\to X$. In Figure \ref{fig:counterexampl_without_gamma}
we see the heightened complex structure on $D$ which is induced by
$\phi$.

\begin{figure}
\subfloat[\label{fig:counterexampl_without_gamma} The disk diagram $D$ from
Figure \ref{fig:reduced_disk_diagram} has a map $\phi:D\to X$ with
$X$ from Figure \ref{fig:non_bislim_complex}. The map $\phi$ induces
on $D$ a heightened complex structure, which is shown in this diagram.]{\includegraphics[scale=0.6]{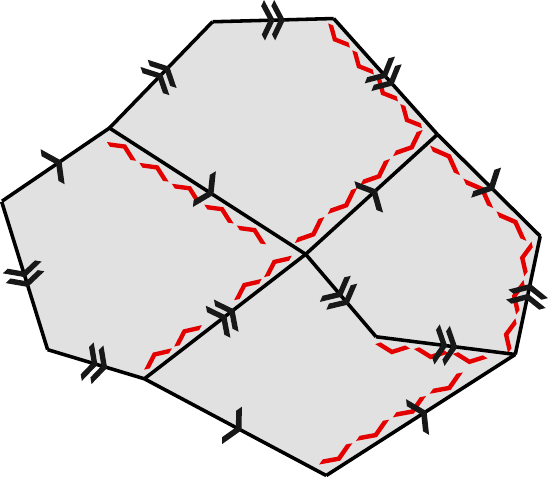}}\hfill{}\subfloat[\label{fig:counterexample} The induced graph $\protect\inducedGamma D$
(drawn in blue) for $D$ from Figure \ref{fig:counterexampl_without_gamma}.
It is induced by following the red arrows which depict the extremal
sides of $D$. Only internal extremal sides are shown since only they
induce edges in $\protect\inducedGamma D$.]{\includegraphics[scale=0.6]{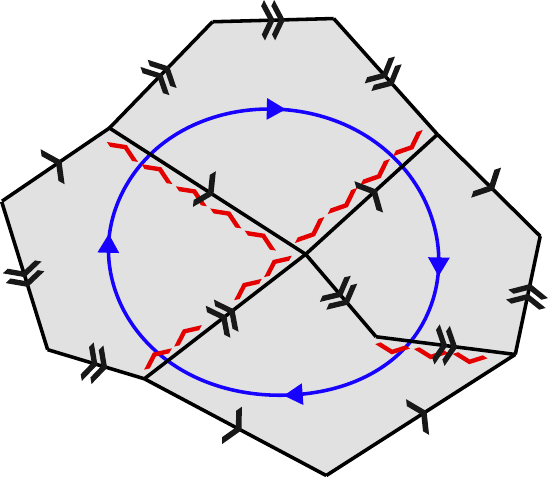}

}\caption{\label{fig:disk_diagram}A reduced disk diagram $D$ in $X$ which
shows that $\protect\heightened$ from Figure \ref{fig:non_bislim_complex}
is not bislim. Note that we mark the edges of $D$ according to their
image in $X$, and we do not glue together edges with the same marking.}
\end{figure}
\end{example}

\begin{defn}
\label{def:gamma}Let $\heightened[Y]$ be a heightened complex. Define
a directed graph\emph{ }$\mathemph{\inducedGamma Y=\inducedGamma{Y,H,L}}$
by ``following the red arrows'' as in Figure \ref{fig:counterexample}.
The vertex set of $\inducedGamma Y$ is the set of 2-cells of $Y$.
As for the edge set, roughly, high sides induce edges going out of
their 2-cells, and low sides induce edges going into their 2-cells.

More precisely, if $s_{1}\neq s_{2}$ are two sides of the same edge,
their corresponding 2-cells are $C_{1}$ and $C_{2}$,\footnote{$C_{2}$ may be equal to $C_{1}$.}
and $s_{1}$ is a high side, then $s_{1}$ induces an edge from $C_{1}$
to $C_{2}$ in $\inducedGamma Y$. If instead $s_{1}$ is low, then
$s_{1}$ induces an edge from $C_{2}$ to $C_{1}$ in $\inducedGamma Y$.
\end{defn}

This definition is inspired by \cite[Definition~4.1]{Bamberger2024}.
We are not being precise about the possibility of there being multiple
parallel edges, as this is immaterial to us --- we only care whether
$\inducedGamma Y$ is acyclic. For example, if a high side $s_{1}$
of $C_{1}$ is adjacent to a low side $s_{2}$ of $C_{2}$, they both
induce an edge from $C_{1}$ to $C_{2}$, and $\inducedGamma Y$ may
be considered to have a single edge or two parallel edges from $C_{1}$
to $C_{2}$.

A heightened complex $\heightened$ induces a graph $\inducedGamma Y$
for every complex $Y$ with a branched near-immersion $\phi:Y\to X$,
by first inducing a structure of a heightened complex $\heightened[Y][\phi^{-1}\smallparens X][\phi^{-1}\smallparens Y]$
on $Y$ and then considering the resulting graph $\inducedGamma Y$.
Additionally, $\phi$ induces a map $\phi:\inducedGamma Y\to\inducedGamma X$,
and if $\inducedGamma X$ is acyclic then $\inducedGamma Y$ is acyclic.
In particular, this applies to the universal cover $\universalCover X$
and to every reduced disk diagram $D$ in $X$ (e.g., Figure \ref{fig:counterexample}).
It is essential that $\phi$ be a branched near-immersion: otherwise,
suppose that $s\neq s'$ are sides of the same edge in $Y$ with $\phi\smallparens s=\phi\smallparens{s'}$,
as in Figure \ref{fig:non_near_immersion}, and suppose that $\phi\smallparens s$
is extremal. Then $s$ and $s'$ induce edges in $\inducedGamma Y$
which do not map to $\inducedGamma X$.

We now define bislim heightened complexes. We focus first on the specific
useful case of good heightened complexes.
\begin{defn}
\label{def:bislim_structure_for_good_complex}Let $\heightened$ be
a heightened complex. If it is a good heightened complex, we say that
$\heightened$ is \emph{bislim} if $\inducedGamma D$ is acyclic for
every reduced disk diagram $D$ in $X$.

In this case we call $\heightened$ a\emph{ bislim structure} or a
\emph{good bislim structure} (since $\heightened$ is assumed to be
good). For the case where $\heightened$ is not necessarily good,
we defer the full definition of bislim structures to Definition \ref{def:bislim-definition-general}
below. We note that it follows from Definition \ref{def:bislim_structure_for_good_complex}
that $X$ is non-degenerate, since if $\heightened$ is a good heightened
complex and $X$ is degenerate, then $\inducedGamma X$ must contain
a self loop.
\end{defn}

\begin{example}
Figure \ref{fig:counterexample} shows that the heightened complex
in Figure \ref{fig:non_bislim_complex} is not bislim.
\end{example}

We prove in Appendix \ref{sec:equivalence-to-bi-slim} that a 2-complex
$X$ admits a good bislim structure in the sense of Definition \ref{def:bislim_structure_for_good_complex}
if and only if $X$ admits a bislim structure in the sense of \cite[Definition 2.1]{helferwise}.
The main difference between the definitions is that in \cite[Definition 2.1]{helferwise}
every 2-cell $C$ of $X$ has exactly one designated ``high side''
and exactly one designated ``low side'', i.e., additional extremal
sides are not allowed. Instead, and this is the crux of our paper,
Definition \ref{def:bislim_structure_for_good_complex} enables us
to consider alternating bislim structures which have multiple high
and low sides, giving stronger results.

Additionally, the definition is symmetric to reversing height, i.e.,
switching low sides with high sides, which is hidden in \cite[Definition 2.1]{helferwise}.
Thus we believe that this definition may prove valuable even independently
of alternating bislim structures.

We note that the requirement in Definition \ref{def:heightened_complexes}
that $H$ and $L$ are disjoint is nonessential. If $H$ and $L$
have a side in common, which is internal, then they induce a cycle
in $\inducedGamma D$ for a disk diagram $D$ which consists of just
two 2-cells, so the heightened complex cannot be bislim. If $X$ has
any non-internal edge, then the corresponding 2-cell can be removed,
and most theorems (e.g., Lemma \ref{lem:no_tree_islands}) can be
proven easily.
\begin{defn}
\label{def:alternating_heightened_complex}A heightened complex $\heightened$
is \emph{alternating} if for every 2-cell $C$ of $X$, the sides
of $H_{C}\cup L_{C}$ alternate between $H_{C}$ and $L_{C}$ when
going around $\boundary C$, i.e., the cyclic ordering of $H_{C}\cup L_{C}$
in $\boundary C$ is $s_{1},\dots,s_{2d}$ where $H_{C}=\left\{ s_{1},s_{3},\dots,s_{2d-1}\right\} $
and $L_{C}=\left\{ s_{2},s_{4},\dots,s_{2d}\right\} $. In this case,
we define the \emph{alternation degree} of $C$ to be $\mathemph{\alt_{C}}=d=\left|H_{C}\right|=\left|L_{C}\right|$.
We say that $C$ is\emph{ $\mathemph k$-alternating} when $\alt_{C}\ge k$,
and we say that $\heightened$ is $\smallparens{d_{1},\dots,d_{r}}$-alternating
when the $i$'th 2-cell of $X$ is $d_{i}$-alternating. We say that
$\heightened$ is $d$-alternating if $X$ has a single 2-cell $C$
and $\alt_{C}\ge d$.

This is essentially the same concept as in Definition \ref{def:alternating},
although since here we are defining bislim structures rather than
partial stackings, the definitions are somewhat different. Figures
\ref{fig:bislim_complex}, \ref{fig:non_bislim_complex} and \ref{fig:non_bislim_alternating}
are all good, alternating and $1$-alternating heightened complexes.
Figure \ref{fig:non_bislim_alternating} is also $2$-alternating.
\end{defn}

Good bislim structures and 1-alternating bislim structures are equivalent:
any 1-alternating bislim structure is good, and any good bislim structure
$\heightened$ can be made into a 1-alternating heightened complex
$\heightened[][H'][L']$ with $H'\subseteq H$ and $L'\subseteq L$,
from which it follows that $\heightened[X][H'][L']$ is also bislim.

We give some additional equivalent definitions for bislim structures,
as follows:
\begin{thm}
\label{thm:equivalent_conditions_for_bislim_structure}Let $\heightened$
be a good heightened complex. The following conditions are equivalent:
\begin{enumerate}
\item \label{enu:bislim-definition-reduced-disk-diagram}$\heightened$
is bislim, i.e., $\inducedGamma D$ is acyclic for every reduced disk
diagram $D$ in $X$.
\item \label{enu:bislim-definition-universal-cover}$\Gamma\largeparens{\universalCover X}$
is acyclic, where $\universalCover X$ is the universal cover of $X$.
\item \label{enu:bislim-definition-most-general}$\inducedGamma Y$ is acyclic
for every branched near-immersion $\phi:Y\to X$ with $Y$ simply
connected.
\end{enumerate}
\end{thm}

\begin{proof}
Condition \ref{enu:bislim-definition-most-general} implies all other
conditions, since the map $D\to X$ for a reduced disk diagram $D$
in $X$, and the natural projection $\universalCover X\to X$, are
both branched near-immersions.

Condition \ref{enu:bislim-definition-universal-cover} implies Condition
\ref{enu:bislim-definition-most-general}: let $\phi:Y\to X$ be a
branched near-immersion where $Y$ is simply connected. Thus $\phi$
lifts to $\psi:Y\to\universalCover X$ such that $\phi=\pi\circ\psi$
where $\pi$ is the covering map $\pi:\universalCover X\to X$. The
map $\psi$ is a branched near-immersion since for any two sides $s_{1}$
and $s_{2}$, if $\psi\smallparens{s_{1}}=\psi\smallparens{s_{2}}$
then also $\phi\smallparens{s_{1}}=\phi\smallparens{s_{2}}$. Thus
$\psi$ induces a map $\inducedGamma Y\to\Gamma\left(\universalCover X\right)$.
If there was any cycle in $\inducedGamma Y$, then its image would
be a cycle in $\Gamma\left(\universalCover X\right)$, a contradiction.

Condition \ref{enu:bislim-definition-reduced-disk-diagram} implies
Condition \ref{enu:bislim-definition-most-general}: Let $\heightened$
be a bislim structure, and let $Y$ be a simply connected 2-complex
with a branched near-immersion $Y\to X$. Assume by contradiction
that there is a directed cycle $\ell$ in $\inducedGamma Y$.

The rough idea of the proof is that since $Y$ is simply connected,
$\ell$ bounds a disk, and this disk can be made into a reduced disk
diagram $D$ containing $\ell$, in contradiction with Condition \ref{enu:bislim-definition-reduced-disk-diagram}.
This idea fails, in particular since $\ell$ is not part of the skeleton
of $Y$, so it cannot bound a disk diagram. However, this idea can
be fixed, in the same way as in the proof of \cite[Theorem~4.5]{Bamberger2024}.

Instead of a disk diagram bounded by $\ell$, we get a pair $A$,
$D$ where $A$ is either an annulus or a Möbius strip of 2-cells
going along $\ell$, so $\inducedGamma A$ is a cycle, and $D$ is
a disk diagram in $Y$ which is attached to $A$ by letting the boundary
$\boundary D$ go once around $A$. Moreover, together they form a
contractible 2-complex $A\cup D$ with a branched near-immersion $A\cup D\to Y$.
Let $A$, $D$ be the smallest such pair.

Note that $\boundary D$ may cross through internal edges of $A$,
and indeed it must do so unless $A\cup D$ is a disk diagram, contradicting
Condition \ref{enu:bislim-definition-reduced-disk-diagram}. Let $e$
be an internal edge of $A$ which is traversed by $\boundary D$.
Thus $e$ is adjacent to a third 2-cell $C_{1}$, and as every internal
edge of $A$, $e$ has an adjacent extremal side, and thus there is
an edge of $\Gamma\left(A\cup D\right)$ between $C_{1}$ and a 2-cell
of $A$. Suppose without loss of generality that this edge is directed
into $C_{1}$.

Since every 2-cell is good, $\boundary C_{1}$ contains a high side,
so there is an outgoing edge from $C_{1}$ into another 2-cell, $C_{2}$.
This way we can continue this path until this path either closes into
a loop, or returns to a 2-cell of $A$. In either case, this forms
a new cycle in $\Gamma\left(A\cup D\right)$, which then forms a smaller
pair $A'$, $D'$ with the same properties, a contradiction.

Thus, since $A$, $D$ is the smallest such pair, it follows that
$A\cup D$ must be a disk diagram itself, contradicting Condition
\ref{enu:bislim-definition-reduced-disk-diagram}.
\end{proof}
We now define bislim structures for general heightened complexes $\heightened$
which are not necessarily good. Although most bislim structures of
interest are indeed good, bislim structures which are not good can
give nontrivial results, e.g., Theorem \ref{thm:generic_bound_on_maximal_irreducible_curvature}.
\begin{defn}
\label{def:bislim-definition-general}A heightened complex $\heightened$
is \emph{bislim} if $X$ is non-degenerate and $\inducedGamma D$
is acyclic for every reduced disk diagram $D$ in $X$ without internal
2-cells.
\end{defn}

\begin{claim}
If $\heightened$ is a good heightened complex, then Definition \ref{def:bislim-definition-general}
is equivalent to definition \ref{def:bislim_structure_for_good_complex}.
\end{claim}

We note that if $\heightened$ is not good, then these definitions
are not equivalent.
\begin{proof}
Let $\heightened$ be a good heightened complex which is bislim in
the sense of Definition \ref{def:bislim-definition-general} but not
in the sense of Definition \ref{def:bislim_structure_for_good_complex}.
Let $D$ be a counterexample to Definition \ref{def:bislim_structure_for_good_complex},
which is a reduced disk diagram $D$ in $X$ where $\inducedGamma D$
has a directed cycle. We will show that $D$ can be made to have no
internal 2-cells, contradicting Definition \ref{def:bislim-definition-general}.

Let $\ell$ be a simple directed cycle in $\inducedGamma D$. We can
view $\ell$ as a metric path inside $D$, so by the Jordan Curve
Theorem $\ell$ divides the space $D$ into an exterior $E$ and an
interior $I$ (including $\ell$ itself). Without loss of generality
we can remove every 2-cell $C$ which is completely outside $\ell$
(i.e., $C\subseteq E$). Without loss of generality $\ell$ is minimal
in the sense that no other simple directed cycle of $\inducedGamma D$
is contained within $I$.

Let $C$ be a 2-cell of $D$ which $\ell$ does not pass through.
Then $C$ is necessarily internal. Since $\heightened$ is good, $H_{C}$
and $L_{C}$ are nonempty, and thus we can find an edge going out
of $C$ and an edge going into $C$ in $\inducedGamma D$ (see Figure
\ref{fig:internal_cell_a}). 

Therefore, starting at $C$, we can form a directed path $p$ in $\inducedGamma D$
going from $C$ in both directions, until each end either loops back
into $p$ or stops at $\ell$ (e.g., Figure \ref{fig:internal_cell_a}).

If one of the ends loops back into $p$, that gives a simple directed
cycle contained in $I$, which does not exist. If both ends of $p$
stop at $\ell$ we can cut through $p$ and also obtain a simple directed
cycle contained in $I$ (Figure \ref{fig:internal_cell_a}), which
again does not exist. Thus the existence of $C$ gives a contradictions,
and $\ell$ passes through every internal 2-cell.

Now, let $C$ be an internal 2-cell, so $\ell$ passes through $C$,
and so some vertex $v$ of $C$ must be outside of $\ell$ (see Figure
\ref{fig:internal_cell_edgecase}). Now we use $v$ to make $C$ non-internal,
i.e., force $\boundary C$ to have a non-internal edge. There is a
path between $v$ and the boundary of $D$ that does not cross $\ell$
since $v$ is outside of $\ell$. We can force $v$ to be on the boundary
of $D$ by separating the two sides of this path (see Figure \ref{fig:internal_cell_edgecase_reduced}).
However, it is still possible that there is a vertex $v$ on $\boundary C\cap\boundary D$
while $C$ is still internal (see Figure \ref{fig:internal_cell_edgecase_reduced}).
We can make $C$ non-internal by branching $C$ (see Figure \ref{fig:internal_cell_edgecase_doubled}).

Thus we can make any internal 2-cell $C$ non-internal. The number
of internal 2-cells strictly decreases, so by doing this again and
again we get a disk diagram $D$ without any internal 2-cells. We
retain the path $\ell$ in $\inducedGamma D$, so $D$ is now a counterexample
to Definition \ref{def:bislim-definition-general}.

\begin{figure}
\subfloat[\label{fig:internal_cell_a}Given an internal 2-cell $C$, $H_{C}$
and $L_{C}$ are nonempty. Thus we can find an edge going out of $C$
and an edge going into $C$ in $\protect\inducedGamma D$. Doing this
again and again we build a path. Here the path eventually meets the
cycle $\ell$ of $\protect\inducedGamma D$, contradicting the minimality
of $\ell$.]{\centering{}\includegraphics[width=0.55\columnwidth]{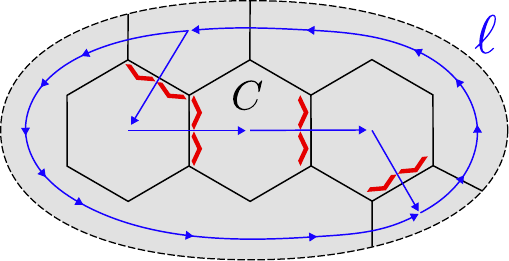}}\hfill{}\subfloat[\label{fig:internal_cell_edgecase}An example where the internal 2-cell
$C$ is on the cycle $\ell$ of $\protect\inducedGamma D$.]{\centering{}\includegraphics[width=0.35\textwidth]{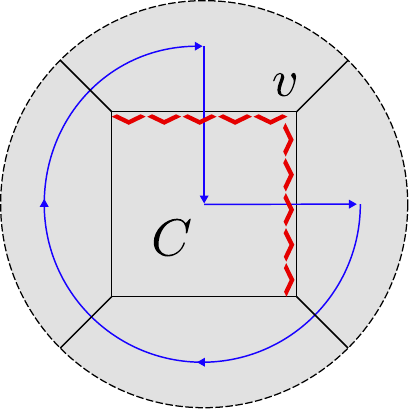}}

\subfloat[\label{fig:internal_cell_edgecase_reduced}We can unglue edges to
ensure that $C$ has an external vertex.]{\centering{}\includegraphics[width=0.35\textwidth]{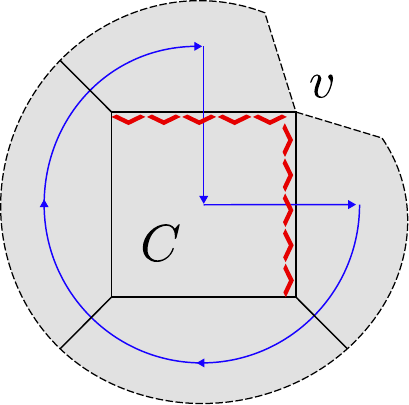}}\hfill{}\subfloat[\label{fig:internal_cell_edgecase_doubled}In this edgecase, we double
the 2-cell $C$ in Figure \ref{fig:internal_cell_edgecase_reduced}
to make it non-internal (we allow branched 2-cells in disk diagrams).
The vertex $v$ splits into two vertices $v'$ and $v''$.]{\centering{}\includegraphics[width=0.4\textwidth]{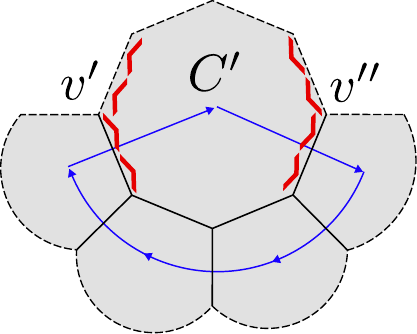}}\caption{How to get rid of internal 2-cells in counterexamples to Definition
\ref{def:bislim_structure_for_good_complex}. This implies Definition
\ref{def:bislim-definition-general}.}
\end{figure}
\end{proof}
We now show an equivalent condition for bislim structures, which is
a variant of Condition \ref{enu:bislim-definition-most-general} of
Theorem \ref{thm:equivalent_conditions_for_bislim_structure}, which
will be helpful in order to work with bislim structures which are
not necessarily good.
\begin{thm}
\label{thm:bislim_equivalent_condition_simply_connected_no_internal_2_cell}A
heightened complex $\heightened$ is bislim if and only if $\inducedGamma Y$
is acyclic for every branched near-immersion $\phi:Y\to X$ where
$Y$ is simply connected and does not have any internal 2-cells.
\end{thm}

\begin{proof}
If the condition holds for all such $Y$, then it holds for all disk
diagrams without internal 2-cells, so $\heightened$ is bislim by
definition.

Assume that $\heightened$ is bislim, let $Y$ be such a 2-complex,
and assume by contradiction that $\inducedGamma Y$ contains a cycle.
The same proof as in Theorem \ref{thm:equivalent_conditions_for_bislim_structure}
which shows Condition \ref{enu:bislim-definition-reduced-disk-diagram}
implies Condition \ref{enu:bislim-definition-most-general} applies
here --- even though Theorem \ref{thm:equivalent_conditions_for_bislim_structure}
requires $Y$ to be good, under closer inspection the proof only requires
that internal 2-cells are good. Since in this case we assume that
$Y$ has no internal 2-cells, the proof applies and we obtain a reduced
disk diagram $D$ in $Y$ where $\inducedGamma D$ also contains a
cycle. Since $Y$ does not have any internal 2-cells, it follows that
$D$ also does not have any internal 2-cells, in contradiction with
Definition \ref{def:bislim-definition-general}.
\end{proof}

\subsection{The $w$-cycle theorem}

\global\long\def\isle{T}%

Our goal now is proving the $w$-cycle theorem for alternating bislim
structures (Theorem \ref{thm:w-cycle-theorem-alternating}) and multiple
variations on it. Similar claims and proofs appear in \cite[Sections~3~and~4]{helferwise},
though their bislim structures are effectively only 1-alternating.
We start with two technical lemmas:
\begin{lem}
\label{lem:neighborhood_technical_lemma}Let $\isle\hookrightarrow Y$
be an embedding of a tree $\isle$ in the skeleton of a finite 2-complex
$Y$. Then this map factors through an embedding into a 2-complex
$M$ as $\isle\hookrightarrow M\to Y$, where
\begin{itemize}
\item The map $M\to Y$ is a branched near-immersion.
\item The map $M\to Y$ is a local isomorphism on every point of $\isle$.
\item The 2-complex $M$ is simply connected.
\item Every 2-cell $C$ in $M$ is adjacent to $\isle$ and is not internal.
\end{itemize}
\end{lem}

\begin{proof}
\global\long\def\first{\text{first}}%
\global\long\def\last{\text{last}}%
\global\long\def\textmax{\text{max}}%
\global\long\def\textmin{\text{min}}%

Consider a small neighborhood of $\isle$ in $I$. We explain how
to complete it into the new complex $M$. For example, consider Figure
\ref{fig:no_islands}.

We start by constructing the skeleton $\skeleton M$. In Figure \ref{fig:no_islands_neighborhood}
the neighborhood sees six ``half-edges'' adjacent to $\isle$, and
each is completed into a corresponding edge in $\skeleton M$ (see
Figure \ref{fig:no_islands_expanded}). The neighborhood only sees
two small parts of the edge $e_{1}$, and does not see that they come
from the same edge $e_{1}$, so they become different edges in $\skeleton M$.
In the same way, even though $e_{2}$ and $e_{3}$ are adjacent in
$Y$, this is not visible in the neighborhood, so they become disconnected
in $\skeleton M$.

Now we similarly construct the 2-cells of $M$. Every small part of
a 2-cell in the neighborhood is completed into a corresponding 2-cell
in $M$. As before, in Figure \ref{fig:no_islands} the 2-cell $C_{1}$
corresponds both to $C_{1}$ and $C_{1}'$ in \ref{fig:no_islands_expanded}
because the neighborhood only sees two separate small parts. Additionally,
$C_{2}$ becomes branched, since otherwise it would force $e_{2}$
and $e_{3}$ to be adjacent to each other.

There is a canonical branched near-immersion $M\to Y$. As we saw
before, there can be multiple edges or 2-cells in $M$ that correspond
to the same edge or 2-cell in $Y$, so this map is not an embedding.

\begin{figure}
\subfloat[\label{fig:no_islands_a}An island $\protect\isle$ (in blue) embedded
in a 2-complex $Y$. Adjacent edges to $\protect\isle$ are shown
in red.]{\begin{centering}
\includegraphics[scale=0.9]{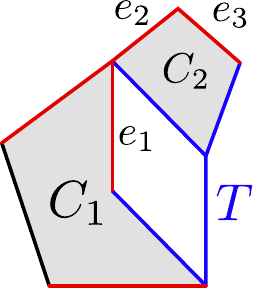}
\par\end{centering}
}\hfill{}\subfloat[\label{fig:no_islands_neighborhood}An $\epsilon$-neighborhood of
$T$]{\begin{centering}
\includegraphics[scale=0.9]{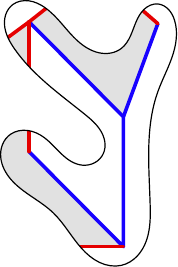}
\par\end{centering}
}\hfill{}\subfloat[\label{fig:no_islands_expanded}The complex $M$, created by completing
the $\epsilon$-neighborhood of $\protect\isle$. Note that $e_{1}$
and $C_{1}$ split into two, $e_{2}$ and $e_{3}$ disconnected, and
$C_{2}$ had to become branched.]{\begin{centering}
\includegraphics[scale=0.8]{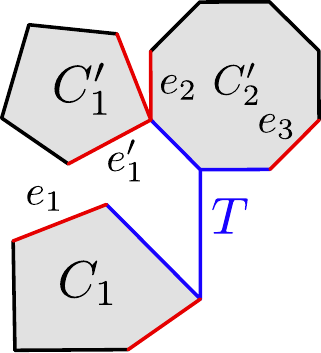}
\par\end{centering}
}\caption{\label{fig:no_islands}A demonstration of the construction of $M$
in Lemma \ref{lem:neighborhood_technical_lemma}. The tree $\protect\isle$
is shown in blue. Adjacent edges to $\protect\isle$ are shown in
red.}
\end{figure}

Let $\isle'$ be the union of $\isle$ and the edges adjacent to $\isle$
in $M$, so $\isle'$ is a tree. For every 2-cell $C$ in $M$, $\boundary C\cap\isle'$
is connected and is not the full cycle $\boundary C$ since otherwise
$Y$ would be degenerate (see Section \ref{subsec:Combinatorial-complexes-and}).
Additionally, edges in $\boundary C\backslash\isle'$ are necessarilly
non-internal. Therefore, $C$ is not internal, as required, and $C$
can be homotoped into $\isle'$. Since this is true for every 2-cell
$C$, $M$ is homotopic to $\isle'$, and thus, contractible.
\end{proof}

We now exploit Lemma \ref{lem:neighborhood_technical_lemma} to analyze
complexes which admit alternating bislim structures:
\begin{lem}
\label{lem:no_tree_islands}Let $\heightened[Y]$ be a finite alternating
bislim structure and assume that all the edges in $Y$ are internal
and $Y$ is not a single vertex. Then every connected component of
$\skeleton Y-H$ contains a cycle.\footnote{Here $H$ is the set of edges corresponding to the high sides in $H$
by abuse of notation. The vertices adjacent to edges in $H$ are not
removed in $\skeleton Y-H$.}
\end{lem}

\begin{proof}
Let $\isle$ be a component of $\skeleton Y-H$, and assume by contradiction
that $\isle$ is a tree. Let $\isle\hookrightarrow M\to Y$ be the
result of Lemma \ref{lem:neighborhood_technical_lemma}. By Theorem
\ref{thm:bislim_equivalent_condition_simply_connected_no_internal_2_cell}
$\inducedGamma M$ is acyclic. Since $Y$ is finite and not a single
vertex, $M$ is finite and contains some 2-cell. Therefore there is
a 2-cell $C_{\textmax}$ which is a vertex of $\inducedGamma M$ without
any incoming edge.

Let $\isle'$ be the union of $\isle$ and the edges which are adjacent
to $\isle$ in $M$. Note that in $\skeleton Y$, every edge which
is adjacent to $\isle$ but is not in $\isle$ is in $H$. Since $M\to Y$
is locally an isomorphism at every point of $\isle$, the same is
true for $M$: every edge in $\isle'-\isle$ has a high side.

The first edge $e_{\first}$ of the path $\boundary C_{\textmax}\cap\isle'$
is in $\isle'$ but not in $\isle$, and so $e_{\first}$ has a high
side $s_{\first}$. Similarly the last edge $e_{\last}$ has a high
side $s_{\last}$. Since $\boundary C_{\textmax}$ intersects at least
one vertex of $\isle$, its two adjacent sides in $\boundary C_{\textmax}$
are contained in $\isle'$, so $\left|\boundary C_{\textmax}\cap\isle'\right|\ge2$
and $e_{\first}\neq e_{\last}$.

The side $s_{\first}$ can be either a side of $C_{\textmax}$ or
a side of another 2-cell adjacent to $e_{\first}$. If $s_{\first}$
is not a side of $C_{\textmax}$, then $s_{\first}$ induces a directed
edge into $C_{\textmax}$, a contradiction (Figure \ref{fig:no_islands_max}).
Thus $s_{\first}$ is a high side of $C_{\textmax}$. The same applies
to $s_{\last}$, so $s_{\last}$ is also a high side of $C_{\textmax}$.
Since the bislim structure is alternating and $e_{\first}\neq e_{\last}$,
somewhere inside $\boundary C_{\textmax}\cap\isle'$ there is a low
side $s_{\text{min}}$ of $C_{\textmax}$ (Figure \ref{fig:no_islands_final}).
Every edge of $\isle$ is internal (since this is true in $Y$), so
$s_{\text{min}}$ induces an edge of $\inducedGamma M$ into $C_{\textmax}$,
a contradiction.

\begin{figure}
\subfloat[\label{fig:no_islands_max}Since $e_{\protect\first}$ is adjacent
to $\protect\isle$, it is in $H$ (hence marked in red), so it has
a high side $s_{\protect\first}$. If $s_{\protect\first}$ is not
a side of $C_{\protect\textmax}$, then $C_{\protect\textmax}$ has
an incoming edge of $\protect\inducedGamma M$.]{\begin{centering}
\includegraphics{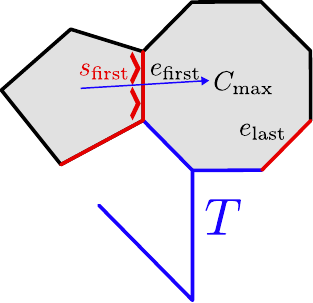}
\par\end{centering}
}\hfill{}\subfloat[\label{fig:no_islands_final}If $s_{\protect\first}$ and $s_{\protect\last}$
are sides of $C_{\protect\textmax}$, since $M$ is alternating, between
them there is a low side $s_{\protect\textmin}$, which again induces
an incoming edge of $\protect\inducedGamma M$.]{\centering{}\includegraphics{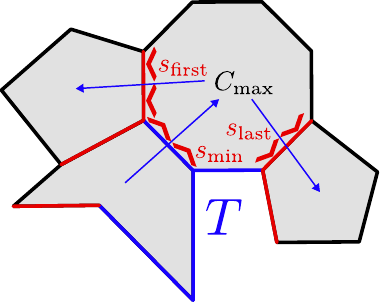}}\caption{A demonstration of the proof of Lemma \ref{lem:no_tree_islands}.
$\protect\isle$ is shown in blue. Adjacent edges to $\protect\isle$
are shown in red, and must be in $H$. For simplicity only relevant
edges and 2-cells are shown in each step.}
\end{figure}
\end{proof}
Lemma \ref{lem:no_tree_islands} provides bounds on Euler characteristics
of complexes:

\begin{thm}
\label{thm:cycle_counting_complex_version_single_complex}If $Y$
admits an alternating bislim structure, all the edges in $Y$ are
internal and $Y$ is not a single vertex, then
\[
\chi\largeparens{\skeleton Y}\ \le\ -\sum_{C\in\text{2-cells}\smallparens Y}\alt_{C}.
\]
\end{thm}

\begin{proof}
By Lemma \ref{lem:no_tree_islands}, no connected component of $\skeleton Y-H$
is a tree. Thus all components have nonpositive Euler characteristic,
so $\chi\largeparens{\skeleton Y-H}\le0$. No edge has two high sides,
since they would induce a length-2 cycle in $\inducedGamma E$ in
a disk diagram $E$ consisting of the corresponding two 2-cells. Thus
$\left|H\right|=\sum_{C\in\text{2-cells}\smallparens Y}\alt_{C}$,
since $\left|H_{C}\right|=\alt_{C}$. Therefore, 
\[
\chi\largeparens{\skeleton Y}\ =\ \chi\largeparens{\skeleton Y-H}-\sum_{C\in\text{2-cells}\smallparens Y}\alt_{C}\ \le\ -\sum_{C\in\text{2-cells}\smallparens Y}\alt_{C}.\qedhere
\]
\end{proof}
We will usually be interested in a base complex $X$, and we will
be interested in $\chi\smallparens Y$ for complexes $Y$ with a branched
near-immersion $Y\to X$. In this case an alternating bislim structure
on $X$ induces an alternating bislim structure on $Y$, and by applying
Theorem \ref{thm:cycle_counting_complex_version_single_complex},
we obtain the following:
\begin{thm}
\label{thm:w_cycles_complex_version_two_complexes}If $Y\to X$ is
a branched near-immersion, the complex $X$ admits an alternating
bislim structure, all the edges of $Y$ are internal and $Y$ is not
a single vertex, then
\[
\chi\largeparens{\skeleton Y}\ \le\ -\sum_{C\in\text{2-cells}\smallparens X}\#_{C}\smallparens Y\cdot\alt_{C},
\]
where $\#_{C}\smallparens Y$ is the number of 2-cells of $Y$ in
the preimage of $C$, counted with multiplicity by their branching
degree.
\end{thm}

\begin{proof}
The bislim structure of $X$ induces a bislim structure on $Y$ where
if a 2-cell $C'$ of $Y$ maps to a 2-cell $C$ of $X$ with branching
degree $d$, then $\alt_{C'}=d\cdot\alt_{C}$. Then the result follows
from Theorem \ref{thm:cycle_counting_complex_version_single_complex}.
\end{proof}
This concludes the core of the theory of bislim structures, and now
we turn to proving a few variants and generalizations that will be
useful in certain applications in the rest of the paper. First of
all, in Section \ref{subsec:Hyperbolicity-and-Negative-immersions}
we will relate bislim structures to hyperbolicity, and it will be
useful to give results for the case where $Y$ does have non-internal
edges. Note that the proof of Lemma \ref{lem:no_tree_islands} only
considers the edges which are adjacent to $\isle$. Therefore the
statement can be strengthened as follows:
\begin{lem}
\label{lem:no_tree_islands_with_boundary}Let $\heightened[Y]$ be
a finite alternating bislim structure such that $Y$ is not a single
vertex. If $\isle$ is a connected component of $\skeleton Y-H$ which
is a tree, then there is a non-internal edge of $Y$ in $\isle$ or
adjacent to $\isle$.
\end{lem}

And then similarly to Theorem \ref{thm:cycle_counting_complex_version_single_complex},
we obtain:
\begin{thm}
\label{thm:cycle_counting_complex_version_with_boundary}If $Y$ admits
an alternating bislim structure and $Y$ is not a single point, then
\[
\chi\largeparens{\skeleton Y}\ \le\ -\sum_{C\in\text{2-cells}\smallparens Y}\alt_{C}+2\left|\boundary Y\right|,
\]
where $\boundary Y$ is the set of non-internal edges in $Y$.
\end{thm}

\begin{proof}
The proof is essentially the same as the proof of Theorem \ref{thm:cycle_counting_complex_version_single_complex},
but now connected components of $\skeleton Y-H$ may be trees if they
are adjacent to or contain an edge of $\boundary Y$. Each edge $e$
of $\boundary Y$ can be contained in one tree or be adjacent to two
different trees (if $e\in H$), giving rise to the new term $2\left|\boundary Y\right|$.
\end{proof}
We note the similarity between Theorem \ref{thm:w_cycles_complex_version_two_complexes}
and Theorem \ref{thm:w-cycle-theorem-alternating}. Furthermore, Theorem
\ref{thm:w-cycle-theorem-alternating} follows from Theorem \ref{thm:w_cycles_complex_version_two_complexes}
by letting $X$ be the presentation complex of $\left\langle \Sigma\vert w\right\rangle $,
and in fact we obtain a version of Theorem \ref{thm:w-cycle-theorem-alternating}
which allows more than one word:

\begin{thm}[The $w$-cycle Theorem, extended]
\label{thm:the_w_cycle_theorem_extended}Let $H$ be a core graph
labeled by an alphabet $\Sigma$, let $w_{1},\dots,w_{k}$ be words
in $\Sigma$, and assume that the $w$-cycles in $H$ cover each edge
in $H$ at least twice. Let $X$ be the presentation complex of $\left\langle \Sigma\vert w_{1},\dots,w_{k}\right\rangle $.
If $X$ admits an alternating bislim structure with alternation degrees
$\alt_{w_{i}}$, then 
\[
\sum_{i=1}^{k}\alt_{w_{i}}\cdot\#_{w_{i}}\smallparens H\ \le\ -\chi\smallparens H.
\]
\end{thm}

\begin{proof}
Let $Y$ be the complex which has $H$ as its 1-skeleton, and a 2-cell
$C$ for each $w_{i}$-cycle in $H$ where the boundary $\boundary C$
is exactly the $w_{i}$-cycle. There is a natural branched map $\varphi:Y\to X$.
Additionally, $\varphi$ is a near-immersion: Assume that two sides
$s_{1}$ and $s_{2}$ of the same edge $e\in H$ map to the same side
in $X$. This corresponds to two $w_{i}$-cycles in $H$ that meet
at an edge in $H$ in the same position within $w_{i}$. Since $H$
is a core graph, there is a unique way to complete the $w_{i}$-cycle,
so both $w_{i}$-cycles are the same. If still $s_{1}\neq s_{2}$,
this means that the cycle is in fact repeating the same $w_{i}$-cycle
which starts at $e$ more than once, so it is not in fact a $w_{i}$-cycle.

Thus the result follows from Theorem \ref{thm:w_cycles_complex_version_two_complexes}.
\end{proof}
Theorem \ref{thm:w-cycle-theorem-alternating} implies Theorem \ref{thm:the_w_cycle_theorem}:
Let $w$ be a word which is not a proper power. By \cite[Proposition 2.4]{helferwise},
and Theorem \ref{thm:equivalence_between_new_and_old_definitions_of_bislim_structures},
$\left\langle \Sigma\vert w\right\rangle $ admits a bislim structure.
This can also be shown to follow from \cite[Lemma 16]{Louder_Wilton_2017}.

Additionally, we present another variant of the $w$-cycle theorem,
which does away with the requirement that $H$ is a core graph, and
will be useful in Section \ref{subsec:Proof-of-Theorem}. This variant
is stated in terms of commutative diagrams of graphs, rather than
by counting $w$-cycles, similarly to \cite[Theorem 2]{Louder_Wilton_2017}.

One may view the data of a combinatorial 2-complex $X$ as the map
of graphs $\cup_{C\in\text{2-cells}\smallparens X}\boundary C\to\skeleton X$.
Conversely, any map of graphs $W\to\Omega$ where $W$ is a disjoint
union of cycles corresponds to a combinatorial 2-complex $X$, by
letting $\Omega$ be the 1-skeleton of $X$, and attaching a 2-cell
on $w$ for every cycle $w$ in $W$. Under this correspondence, sides
of $X$ correspond to edges of $W$, and the complex $X$ admits an
alternating bislim structure if and only if $W\to\Omega$ has a corresponding
alternating partial stacking (\cite{partial_stackings_and_bislim_structures}).\footnote{Alternatively, to make the following self-contained and not dependent
on \cite{partial_stackings_and_bislim_structures}, it is possible
to \textit{define} that $W\to\Omega$ admits an alternating partial
stacking if $X$ admits an alternating bislim structure.} For each component $w$ of $W$ we denote by $\alt_{w}$ the alternation
degree of the corresponding 2-cell of $X$.

If we have another 2-complex $Y$ with a branched map $\varphi:Y\to X$,
then this corresponds to the following commutative diagram:
\[\begin{tikzcd}
	{\mathbb{S}} & {Y^{\left( 1 \right)}} \\
	W & \Omega
	\arrow[from=1-1, to=1-2]
	\arrow["{\text{cover}}"', from=1-1, to=2-1]
	\arrow[from=1-2, to=2-2]
	\arrow[from=2-1, to=2-2]
\end{tikzcd}\] where the map $\mathbb{S}\to W$ is a covering map on each component
of $\mathbb{S}$ (and the covering degree corresponds to the branching
degree). If $\varphi$ is also a branched near-immersion, then this
corresponds to the following property: for every pair of edges $e_{1}$
and $e_{2}$ in $\mathbb{S}$, if their images in $W$ are the same
and their images in $\skeleton Y$ are the same, then $e_{1}=e_{2}$.

Thus, by renaming $\skeleton Y$ to $H$ and $\skeleton X$ to $\Omega$,
we obtain the following equivalent statement of Theorem \ref{thm:w_cycles_complex_version_two_complexes}:
\begin{thm}
\label{thm:w_cycles_version_with_S}Let $\Omega$, $W$, $H$ and
$\mathbb{S}$ be graphs fitting into the following commutative diagram:
\[\begin{tikzcd}
	{\mathbb{S}} & H \\
	W & \Omega
	\arrow[from=1-1, to=1-2]
	\arrow["{\text{cover}}"', from=1-1, to=2-1]
	\arrow[from=1-2, to=2-2]
	\arrow[from=2-1, to=2-2]
\end{tikzcd}\]where $W$ and $\mathbb{S}$ are nonempty disjoint unions of cycles,
and the map $\mathbb{S}\to H$ covers each edge of $H$ at least twice.
For each component $w$ of $W$ assume the restriction of the map
$\psi:\mathbb{S}\to W$ to $\psi^{-1}\smallparens w$ is a degree
$d_{w}$ covering map. Finally, assume that there is no pair of edges
in $\mathbb{S}$ which both have the same image both in $W$ and in
$H$. If $W\to\Omega$ admits an alternating partial stacking, then

\[
\chi\smallparens H\ \le\ -\sum_{w\text{ component of }W}d_{w}\cdot\alt_{w}.
\]
\end{thm}

\begin{proof}
As explained before, the diagram corresponds to a map of 2-complexes
$Y\to X$ which is a branched near-immersion, where $H=\skeleton Y$,
the graph $\mathbb{S}$ corresponds to 2-cells of $Y$, and also $\Omega=\skeleton X$
and $W$ corresponds to the 2-cells of $X$. Now this statement is
equivalent to Theorem \ref{thm:w_cycles_complex_version_two_complexes}.
\end{proof}

\section{\label{sec:Existence-of-bislim-structures}Existence of bislim structures
for generic presentations}

In this section, we prove that generic presentations with $r$ generators
and $k$ relators have alternating bislim structures with optimal
alternation degrees. Given a presentation $\left\langle x_{1},\dots,x_{r}\vert w_{1},\dots,w_{k}\right\rangle $,
we construct its \emph{presentation complex} $X$ by starting with
a bouquet of $r$ circles with one circle for each generator, and
for each relator $w_{i}$ we add a disk $C_{i}$ with $\boundary C_{i}$
glued along $w_{i}$. We say that the presentation admits an $\smallparens{a_{1},\dots,a_{k}}$-alternating
bislim structure if $X$ admits an $\smallparens{a_{1},\dots,a_{k}}$-alternating
bislim structure. Finally, we say that $w_{i}$ is $a_{i}$-alternating
if $C_{i}$ is $a_{i}$-alternating. We note that bislim structures
are not necessarily preserved between different presentations of the
same group.

We work in the few relator model of random presentations, where we
fix positive integers $r$ and $k$, and we generate a random presentation
$\left\langle x_{1},\dots,x_{r}\vert w_{1},\dots,w_{k}\right\rangle $
by picking $k$ relators $w_{1},\dots,w_{k}$ uniformly out of all
cyclically-reduced words of length $\ell$. We say that a property
holds for a generic presentation if the probability that the property
holds tends to $1$ as $\ell$ tends to infinity. 

Theorem \ref{thm:generic_bislim_structures_one_relator_case} claims
that a generic one-relator presentation (i.e., $k=1$) admits an $\smallparens{r-1}$-alternating
bislim structure. Moreover, in the generic case,
\begin{thm}
\label{thm:generic_bislim_structures_generic_case}For every choice
of nonnegative integers $r$, $k$ and $a_{1},\dots,a_{k}$ such that
$\sum_{i}a_{i}<r$, a generic $r$-generator $k$-relator presentation
$\left\langle x_{1},\dots,x_{r}\vert w_{1},\dots,w_{k}\right\rangle $
admits a bislim structure where $w_{i}$ is $a_{i}$-alternating.
\end{thm}

The assumption that $\sum_{i}a_{i}<r$ is tight: by Theorem \ref{thm:cycle_counting_complex_version_single_complex}
we have that $1-r=\chi\largeparens{\skeleton X}\le-\sum_{i}a_{i}$
unless $X$ has a non-internal edge. In that case, if $\sum_{i}a_{i}\ge r$,
there are at least $2r$ extremal sides in $X$. However, each of
the $r$ edges of $X$ can have at most one high side and one low
side, and the non-internal edge of $X$ only has one side, a contradiction.

In order to develop the ideas of this section, we start by showing
a (relatively) simple sufficient criterion for a heightened complex
to be bislim in Section \ref{subsec:simple-criteria}. The simple
criterion gives infinitely many examples of complexes which admit
alternating bislim structures, yet does not capture any alternating
bislim structures with $\sum_{i}a_{i}=r-1$, and thus is insufficient
to prove Theorem \ref{thm:generic_bislim_structures_generic_case}.
However, the simple criterion will serve as a foundation on which
we will build a relaxed criterion in Section \ref{subsec:The-final-criteria},
and use the relaxed criterion to prove Theorem \ref{thm:generic_bislim_structures_generic_case}
in Section \ref{subsec:Generic-words-analysis}.

\subsection{\label{subsec:simple-criteria}A Simple criterion for bislim structures}

In this section we present a simple criterion, which guarantees that
a complex admits a bislim structure (Theorem \ref{thm:simple-criterion}).
Recall the notion of a piece from small cancellation theory:
\begin{defn}
Let $D$ be a disk diagram. A \emph{piece} in $D$ is a maximal path
of internal edges in $\skeleton D$ in which all of the vertices are
of degree exactly 2 except the first and last vertices. A \emph{piece}
in a 2-complex $X$ is a piece in some reduced disk diagram $D$ in
$X$.

It is enough to consider reduced disk diagrams $D$ in $X$ with only
two faces to exhibit all pieces of $X$. This notion of a piece of
$X$ is the same as the standard notion of a piece of a representation
from small cancellation theory when $X$ is a presentation complex.
\end{defn}

A piece and its inverse are considered to be the same piece. This
matters since we will count the number of pieces in disk diagrams
in the following. 
\begin{defn}
A heightened complex $\heightened$ has a \emph{trivial cycle} if
$X$ has two high sides of the same edge $e$, or two low sides of
the same edge $e$.
\end{defn}

See Figure \ref{fig:trivial_cycle_heightened_complex} for an example
of a heightened complex with a trivial cycle. A heightened complex
$\heightened$ has a trivial cycle if and only if there is a reduced
disk diagram $D$ in $X$ with just two 2-cells glued together at
exactly one edge, where $\inducedGamma D$ has a cycle of length 2,
as in Figure \ref{fig:trivial_cycle_diagram}. If $X$ is a presentation
complex of $\left\langle x_{1},\dots,x_{r}\vert w_{1},\dots,w_{k}\right\rangle $,
then $\heightened$ has a trivial cycle if and only if two appearances
of the same generator $a\in\left\{ x_{1},\dots,x_{r}\right\} $ in
$w_{1},\dots,w_{k}$ are both high or both low.

\begin{figure}
\subfloat[\label{fig:trivial_cycle_heightened_complex}]{
\begin{centering}
\includegraphics[totalheight=0.2\textwidth]{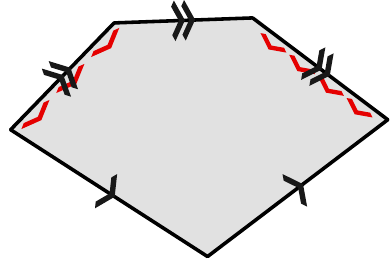}
\par\end{centering}
}\hfill{}\subfloat[\label{fig:trivial_cycle_diagram}]{
\begin{centering}
\includegraphics[totalheight=0.2\textwidth]{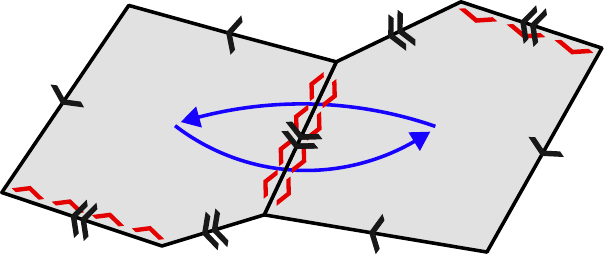}
\par\end{centering}
}\caption{\label{fig:trivial_cycles}Figure \ref{fig:trivial_cycle_heightened_complex}
depicts a heightened complex $\protect\heightened$ with a trivial
cycle where $X$ is given by Figure \ref{fig:presentation_complex}.
The trivial cycle is exemplified by the diagram $D$ in Figure \ref{fig:trivial_cycle_diagram},
where $\Gamma\protect\smallparens D$ contains a cycle, drawn in blue.}

\end{figure}

\begin{defn}
Let $X$ be a 2-complex. We say that an edge $e$ of $X$ is \emph{neutral}
if $e$ has no extremal sides. We say that a side $s$ of $X$ is
\emph{neutral} if it corresponds to a neutral edge, i.e., no side
$s'$ of the same edge as $s$ is extremal.
\end{defn}

\begin{thm}[The Simple criterion]
\label{thm:simple-criterion}Let $\heightened$ be a heightened complex.
Given an integer $R\ge4$, if the following criterion holds, $\heightened$
is bislim.
\global\long\def\simpltiny{\frac{R}{4}}%
\begin{enumerate}
\item \label{enu:no_trivial_cycle_condition}$\heightened$ has no trivial
cycle.
\item \label{enu:small_cancellation_condition}The pieces of $X$ are of
length at most $\simpltiny$.
\item \label{enu:radius_condition}For every extremal side $s$ of a 2-cell
$C$, all sides within distance $R-2$ of $s$ in $\boundary C$ are
neutral.
\item \label{enu:extremal-edge-distance-condition}For every two extremal
sides $s_{1}$ and $s_{2}$ of the same 2-cell $C$, the distance
between $s_{1}$ and $s_{2}$ in $\boundary C$ is at least $2R-1$.
\end{enumerate}
\end{thm}

While generic presentation complexes can be shown to satisfy Conditions
1, 2, and \ref{enu:extremal-edge-distance-condition}, it can be shown
that they do not satisfy Condition \ref{enu:radius_condition}.
\begin{example}
Consider the presentation complex $X$ of $\left\langle \alpha,\beta,x,y\,\vert\,w\right\rangle $
where
\[
w=xy\boldsymbol{\alpha}x^{2}\alpha yx\beta^{-2}xy^{-1}\boldsymbol{\beta}y^{2}\beta\alpha^{-1}x\beta.
\]
Let $H$ be the singleton set of just the side of $X$ corresponding
to the letter $\boldsymbol{\alpha}$ marked in bold in $w$, and let
$L$ be the singleton set of just the side of $X$ corresponding to
the letter $\boldsymbol{\beta}$ marked in bold in $w$. The simple
criterion applies to the heightened complex $\heightened$ with $R=4$.
Indeed,
\begin{enumerate}
\item There are no trivial cycles, since the two extremal sides are sides
of different edges (the edge corresponding to $\alpha$ and the edge
corresponding to $\beta$).
\item All pieces of $X$ are of length $1$.
\item The sides within distance $2$ of the extremal sides are adjacent
to the edges corresponding to $x$ and $y$, which are neutral edges.
\item The distance between the extremal sides is $9$.
\end{enumerate}
Thus $\heightened$ is bislim. Since the sole 2-cell is 1-alternating,
$\heightened$ is a 1-alternating bislim structure.
\end{example}

This example is somewhat redundant, since by \cite[Proposition~2.4]{helferwise}
every torsion-free one-relator presentation admits a 1-alternating
bislim structure. However, this example is clearly extensible to cases
where $w$ admits a k-alternating bislim structure with $k\ge2$,
or where the complex has multiple 2-cells. In these cases relatively
little was known before the current work.

Conditions \ref{enu:radius_condition} and \ref{enu:small_cancellation_condition}
of the simple criterion imply that there must be at least two neutral
edges: if there is only one neutral edge, then by Condition \ref{enu:radius_condition}
the neighborhoods of extremal sides must go through this single edge,
which thus must be the same as each other. Thus they are pieces of
length $R-2$, in contradiction with Condition \ref{enu:small_cancellation_condition}.

Since each edge can have at most one high side and one low side, and
none if it is neutral, the overall alternation degree $\left|H\right|=\left|L\right|$
is at most the number of non-neutral edges, so the simple criterion
cannot ever reach the optimal alternation degree $\left|H\right|=\left|L\right|=r-1$
required in Theorem \ref{thm:generic_bislim_structures_generic_case}.

However, it is easy to construct infinitely many presentation complexes
for which the criterion holds. First, reserve at least two neutral
letters, pick the positions of the extremal letters in the relators
to be evenly spaced, such that Condition \ref{enu:extremal-edge-distance-condition}
holds, and alternating according to the desired alternation degrees.
Pick all the high letters and all the low letters such that Condition
\ref{enu:no_trivial_cycle_condition} is satisfied. Then randomly
choose the neighborhood of radius $R$ around the extremal letters
to be a random word in the neutral letters so that Condition \ref{enu:radius_condition}
holds. Since there are at least two neutral letters, the chances that
these neighborhoods contain large pieces which contradict Condition
\ref{enu:small_cancellation_condition} tends to $0$ as $R$ tends
to infinity.

Finally, fill in the rest of the relators randomly as well, now using
the whole alphabet including the extremal letters. If $R$ grows even
moderately quickly as the length of the relators tends to infinity,
the probability that Condition \ref{enu:small_cancellation_condition}
holds tends to $1$.
\begin{rem}
In fact, some cases of Theorem \ref{thm:generic_bislim_structures_generic_case}
where $\sum_{i}a_{i}$ is small enough in relation to $r$, can be
proven using the simple criterion by itself through careful analysis.
Since the relaxed criterion is required to prove Theorem \ref{thm:generic_bislim_structures_generic_case}
fully, we do not dwell on that here.
\end{rem}

\global\long\def\internalBoundary{\overline{\boundary}}%

\global\long\def\first{\text{first}}%
\global\long\def\last{\text{last}}%

\subsubsection{Proof of Theorem \ref{thm:simple-criterion}}

We now walk through a proof of Theorem \ref{thm:simple-criterion}.
Assume that $\heightened$ satisfies the conditions of Theorem \ref{thm:simple-criterion}
yet $\heightened$ is not bislim. We say that $D$ is a \emph{counterexample}
if $D$ is a counterexample to Definition \ref{def:bislim-definition-general},
i.e., $D$ is a reduced disk diagram in $X$ without internal 2-cells
such that $\inducedGamma D$ has a cycle. Additionally $D$ is a \emph{minimal
}counterexample if the number of internal edges plus the number of
2-cells is minimal. For example the disk diagram in Figure \ref{fig:disk_diagram}
is a minimal counterexample. The \emph{internal tree} $\mathemph{T_{D}}$
is the subgraph of $\skeleton D$ generated by its internal edges.
It is a connected graph since $D$ is minimal, and it is a tree since
$D$ has no internal 2-cells. The \emph{internal boundary} $\mathemph{\internalBoundary C}$
of a 2-cell $C$ is the set of sides of $C$ which are internal, i.e.,
$\boundary C\cap T_{D}$.

Our goal is to count the number of pieces in $T_{D}$ to arrive at
a contradiction. The number of pieces is bounded from above by a simple
argument (see Lemma \ref{lem:counting_pieces} below):
\begin{equation}
\#\text{pieces}\largeparens{T_{D}}\ <\ 2\cdot\#\text{leaves}\largeparens{T_{D}}.\label{eq:bounding_pieces_from_above}
\end{equation}
E.g., in Figure \ref{fig:disk_diagram} we have $4$ pieces and $4$
leaves in $T_{D}$. Now our goal is to use the conditions and the
heightened complex to bound the number of pieces from below, which
we do in Lemma \ref{lem:piece_lower_bound} below, as follows:
\begin{equation}
\#\text{pieces}\largeparens{T_{D}}\ \ge\ 2\cdot\#\text{leaves}\largeparens{T_{D}},\label{eq:bounding_pieces_from_below}
\end{equation}

arriving at a contradiction. Hence $\heightened$ is bislim as no
counterexample $D$ is possible.

We start by proving simple lemmas about $D$, and conclude that (\ref{eq:bounding_pieces_from_above})
holds.

\begin{wrapfigure}[16]{O}{0.4\columnwidth}%
\centering{}\includegraphics[width=0.4\columnwidth]{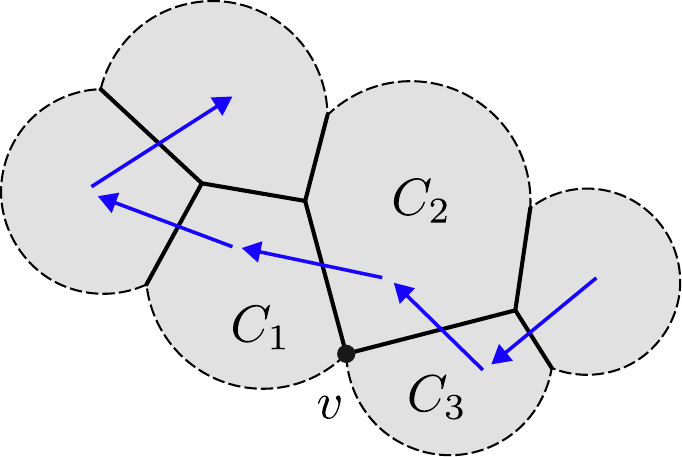}\caption{\label{fig:boundary_leaf_lemma} A diagram for Lemma \ref{lem:boundary_leaf_lemma}.
Since $C_{2}$ divides the disk diagram into two, $\protect\inducedGamma D$
cannot traverse a simple directed cycle through all 2-cells in $D$.
}
\end{wrapfigure}%

\begin{lem}
\label{lem:boundary_leaf_lemma}Let $v$ be a vertex on the boundary
of $D$ which is also in $T_{D}$. Then $v$ is a leaf of $T_{D}$.
\end{lem}

\begin{proof}
If $v$ is not a leaf of $T_{D}$, then $v$ is adjacent to at least
three 2-cells, and let $C_{1},C_{2},C_{3}$ be the first three in
clockwise order (see Figure \ref{fig:boundary_leaf_lemma}). Since
$C_{2}$ is not internal, if we remove $C_{2}$ from $D$ then $v$
becomes a cut vertex. Thus $\inducedGamma D$ cannot traverse a directed
cycle through all 2-cells in $D$, without going through $C_{2}$
twice, contradicting the minimality of $D$.
\end{proof}
\begin{cor}
\label{cor:boundary_leaf}The first and last vertices of $\internalBoundary C$
are leaves of $T_{D}$ for every 2-cell $C$ of $D$.
\end{cor}

\begin{proof}
The first vertex of $\internalBoundary C$ is both on the boundary
$\boundary D$ (since it is adjacent to a non-internal edge) and on
$T_{D}$ (since it is adjacent to an internal edge). The same holds
for the last vertex.
\end{proof}
\begin{lem}
\label{lem:counting_pieces}Equation (\ref{eq:bounding_pieces_from_above})
holds, i.e.,
\[
\#\text{pieces}\largeparens{T_{D}}\ <\ 2\cdot\#\text{leaves}\largeparens{T_{D}}.
\]
\end{lem}

\begin{proof}
Contract vertices of degree 2 of $T_{D}$ to obtain a new tree $S$.
Pieces of $D$ correspond to edges of $S$ (cases such as Figure \ref{fig:boundary_leaf_lemma}
where edges of $S$ correspond to multiple pieces of $D$ are ruled
out by Lemma \ref{lem:boundary_leaf_lemma}). Now $S$ is a tree without
vertices of degree 2, and so
\[
2\cdot\#\text{leaves}\smallparens S\ge\sum_{v\in S}3-\deg v=3\cdot\#\text{vertices}\smallparens S-2\cdot\#\text{edges}\smallparens S=\#\text{edges}\smallparens S+3.\qedhere
\]
\end{proof}

Now we turn to proving a technical lemma:
\begin{lem}
\label{lem:length_two} Let $\heightened$ be a heightened complex
which is not bislim, and let $D$ be a minimal counterexample for
$X$. If $\heightened$ does not have trivial cycles, the internal
boundary $\internalBoundary C$ is of length at least two for each
2-cell $C$ of $D$.
\end{lem}

\begin{proof}
If $\internalBoundary C$ is empty, then $C$ can be removed from
$D$ and thus $D$ is not minimal. If $\internalBoundary C=\left\{ e\right\} $,
let $C'$ be the only neighbor of $C$ on the other side of $e$.
If $C$ has both an incoming edge and an outgoing edge to $C'$ in
$\inducedGamma D$, then $C$ and $C'$ make a trivial cycle. Otherwise,
$C$ can be removed from $D$ without removing any cycles from $\inducedGamma D$,
so $D$ is not minimal.
\end{proof}
Now we proceed to proving (\ref{eq:bounding_pieces_from_below}).
Consider a leaf $v$ of $T_{D}$, and we wish to find at least $2$
pieces for every leaf $v$. Let $e$ be the edge in $T_{D}$ adjacent
to $v$ (see Figure \ref{fig:leaf_edge_lemma_1}). We can construct
a new disk diagram $D'$ by separating the two 2-cells adjacent to
$e$ (see Figure \ref{fig:leaf_edge_lemma}). If $e$ is neutral we
have $\inducedGamma D=\inducedGamma{D'}$ in contradiction with the
minimality of $D$. Thus $e$ must have an extremal side $s$. Overall
we proved the following:

\begin{figure}
\subfloat[\label{fig:leaf_edge_lemma_1}A disk diagram which does not satisfy
Lemma \ref{lem:leaf_edge_lemma}]{\begin{centering}
\includegraphics[width=0.3\columnwidth]{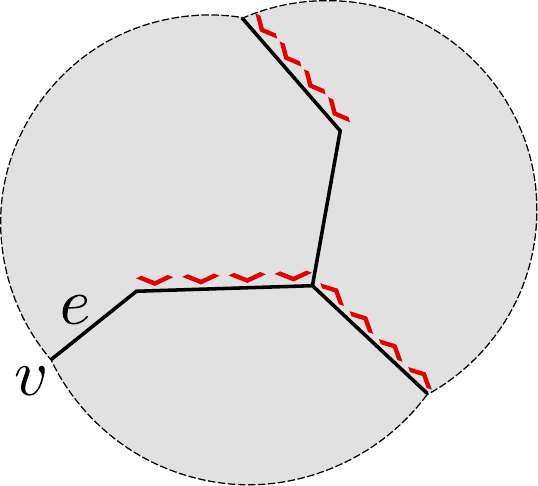}
\par\end{centering}
}\hfill{}\subfloat[\label{fig:leaf_edge_lemma_2}An ungluing step of \ref{fig:leaf_edge_lemma_1},
showing that the diagram was not minimal. The vertex $v$ is split
into $v_{1}$ and $v_{2}$, and $e$ is split into $e_{1}$ and $e_{2}$.]{\begin{centering}
\includegraphics[width=0.3\columnwidth]{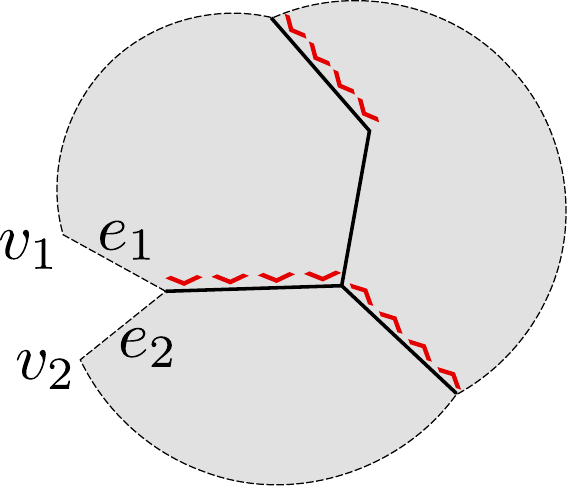}
\par\end{centering}
}

\caption{\label{fig:leaf_edge_lemma}A demonstration of Lemma \ref{lem:leaf_edge_lemma}}
\end{figure}

\begin{lem}
\label{lem:leaf_edge_lemma}If $e\in T_{D}$ be an edge adjacent to
a leaf of $T_{D}$, then $e$ has an extremal side (see Figure \ref{fig:leaf_edge_lemma}).
\end{lem}

Moreover, consider a simple cycle $\ell$ in $\inducedGamma D$. By
the above argument, $\ell$ goes through every edge of $\inducedGamma D$
which induced by an extremal side $s$ which is adjacent to a leaf
of $T_{D}$, since otherwise, $D$ would not be minimal. Going around
the boundary of $D$, the edges of $\inducedGamma D$ of this type
in fact close a cycle themselves, so they are the entirety of $\ell$.
Now consider an internal extremal side $s'$ which is not of this
type, i.e., not adjacent to a leaf of $T_{D}$. It induces another
edge in $\inducedGamma D$, which together with half of $\ell$ closes
a cycle. This cycle does not go through every edge that $\ell$ does,
and so by the above argument, $D$ is not minimal. Therefore we conclude
the following:
\begin{lem}
\label{lem:No_internal_extremal_sides}Every internal extremal side
in $D$ is adjacent to a leaf of $T_{D}$.
\end{lem}

Now we continue to examine $v$, $e$ and its extremal side $s$.
Let $C$ be the 2-cell containing $s$. By Condition \ref{enu:no_trivial_cycle_condition}
of Theorem \ref{thm:simple-criterion}, $\heightened$ has no trivial
cycle, and so by Lemma \ref{lem:length_two}, $\left|\internalBoundary C\right|\ge2$,
so $s=s_{1}$ is adjacent to an internal side $s_{2}\in\internalBoundary C$.
Continuing from $s_{2}$ consider the $R-2$ next sides $s_{2},\dots,s_{R-1}$
in $\boundary C$ (see Figure \ref{fig:simple_construction}). By
Condition \ref{enu:radius_condition} the sides $s_{2},\dots,s_{R-1}$
are all neutral, and they cannot have extremal sides, and so by Lemma
\ref{lem:leaf_edge_lemma} they cannot be adjacent to leaves of $T_{D}$.
It follows that $\left|\internalBoundary C\right|\ge R$ (see Figure
\ref{fig:simple_construction_all}). Overall we proved the following:

\begin{figure}
\centering{}\subfloat[\label{fig:simple_construction}An example demonstrating Lemma \ref{lem:large_cells}.
Given a vertex $v$ of $T_{D}$, by Lemma \ref{lem:leaf_edge_lemma}
we find an adjacent extremal side $s$. Since $R=5$, by Condition
\ref{enu:radius_condition} the sides $s_{2},s_{3},s_{4}$ must be
neutral. However, by Lemma \ref{lem:leaf_edge_lemma} on $v_{5}$
the side $s_{4}$ cannot be neutral --- a contradiction.]{\begin{centering}
\includegraphics[width=0.3\columnwidth]{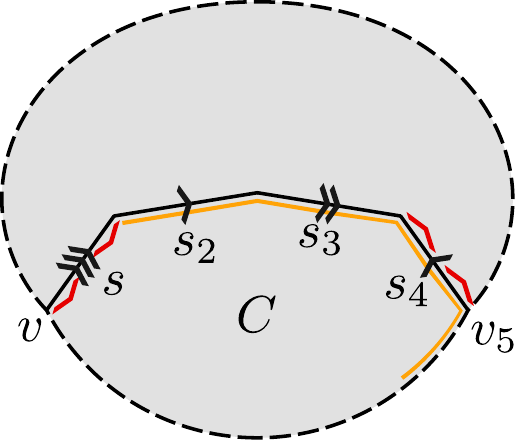}
\par\end{centering}
}\hfill{}\subfloat[\label{fig:simple_construction_correct}An example of Lemma \ref{lem:large_cells},
where $R=5$. Here $s_{1},\dots,s_{4}$ are neutral but $s_{5}$ is
not, and $R=5$. The path $p_{v}\subseteq\protect\internalBoundary C$
is marked in orange, though in this case $p_{v}=\protect\internalBoundary C$.]{\begin{centering}
\includegraphics[width=0.3\columnwidth]{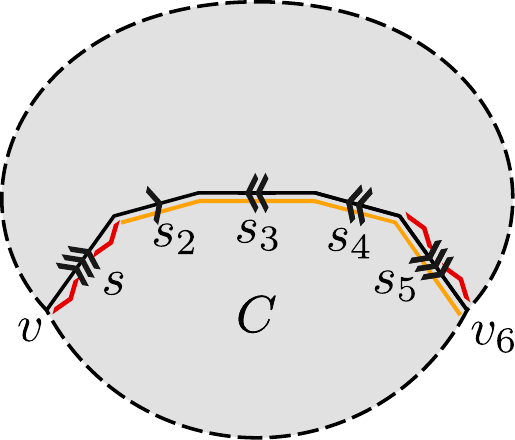}
\par\end{centering}
}\hfill{}\subfloat[\label{fig:simple_construction_counting}The piece $q$ can be contained
in paths of the form $p_{v}$ at most twice, once from each adjacent
2-cell, since if a 2-cell contains two paths of the form $p_{v}$,
they are disjoint, as shown in this diagram.]{\begin{centering}
\includegraphics[width=0.3\columnwidth]{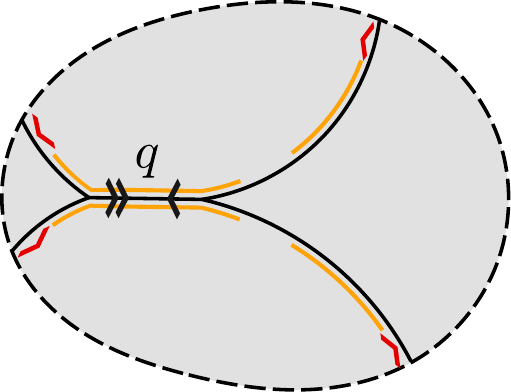}
\par\end{centering}
}\caption{\label{fig:simple_construction_all}Diagrams for Lemma \ref{lem:large_cells}.
The paths of the form $p_{v}$ are marked in orange, except for their
first side which is marked as extremal.}
\end{figure}

\begin{lem}
\label{lem:large_cells} Let $e\in T_{D}$ be adjacent to a leaf of
$T_{D}$ where $e$ has an extremal side whose 2-cell is $C$. Then
$\left|\internalBoundary C\right|\ge R$.
\end{lem}

Let $\mathemph{p_{v}}$ be the prefix of $\internalBoundary C$ of
length $R$, starting at $s$ (see Figure \ref{fig:simple_construction_correct}).
Since $\left|\internalBoundary C\right|\ge R$, the path $p_{v}$
is contained in $\internalBoundary C$. By Condition \ref{enu:small_cancellation_condition}
all pieces are of length at most $\simpltiny$, so the path $p_{v}$
contains at least $4$ pieces. Now it remains to show that each piece
is counted at most twice:
\begin{lem}
\label{lem:double_counting}Each piece $q$ of $D$ is contained in
at most two paths of the form $p_{v}$ for leaves $v$ of $T_{D}$.
\end{lem}

\begin{proof}
Let $C$ be a 2-cell and let $v_{\first}$ and $v_{\last}$ be the
first and last vertices of $\internalBoundary C$, which are leaves
of $T_{D}$ by Corollary \ref{cor:boundary_leaf}, and let $s_{\first}$
and $s_{\last}$ be the first and last sides in $\internalBoundary C$.
If $s_{\last}$ and $s_{\first}$ are both extremal, then $C$ is
associated with two paths of the form $p_{v}$, namely $p_{v_{\first}}$
and $p_{v_{\last}}$. In this case, by Condition \ref{enu:extremal-edge-distance-condition}
the distance between $s_{\first}$ and $s_{\last}$ is at least $2R-1$,
so the paths $p_{v_{\first}}$ and $p_{v_{\last}}$ are disjoint.

It follows that the piece $q$ can only be contained in a path of
the form $p_{v}$ at most once for each of its two neighboring 2-cells,
so $q$ is contained in at most two paths $p_{v}$ (see Figure \ref{fig:simple_construction_counting}).
\end{proof}
Now we have the ingredients to prove (\ref{eq:bounding_pieces_from_below}):
\begin{lem}
\label{lem:piece_lower_bound}Equation (\ref{eq:bounding_pieces_from_below})
holds, i.e.,
\[
\#\text{pieces}\largeparens{T_{D}}\ \ge\ 2\cdot\#\text{leaves}\largeparens{T_{D}}.
\]
\end{lem}

\begin{proof}
Recall that every piece is of length at most $\simpltiny$, and every
path $p_{v}$ is internal and of length $R$, and so every path $p_{v}$
contains at least $4$ pieces. By Lemma \ref{lem:double_counting},
\[
2\cdot\#\text{pieces}\largeparens{T_{D}}\ \ge\ \sum_{v\in\text{leaves}\smallparens{T_{D}}}\#\text{pieces}\largeparens{p_{v}}\ \ge\ 4\cdot\#\text{leaves}\largeparens{T_{D}}\qedhere
\]
\end{proof}
This concludes the simple criterion, since (\ref{eq:bounding_pieces_from_above})
and (\ref{eq:bounding_pieces_from_below}) are contradictory, hence
no minimal counterexample exists, and $\heightened$ is bislim.

\subsection{\label{subsec:The-final-criteria}The relaxed criterion for bislim
structures}

\global\long\def\b#1{b_{#1}}%
\global\long\def\r#1{r_{#1}}%

\global\long\def\bl{\ell_{b}}%
\global\long\def\rl{\ell_{r}}%

\global\long\def\a{e_{*}}%

\global\long\def\sidedpiece{\underline{p}}%

We now generalize the simple criterion (Theorem \ref{thm:simple-criterion})
to the relaxed criterion (Theorem \ref{thm:full-criterion}). We start
with an informal overview of the approach, based on the simple criterion.
A rough sketch of the proof of the simple criterion is as follows:
\begin{enumerate}
\item Let $\heightened$ be a heightened complex, and assume that it is
not bislim.
\item Consider the minimal counterexample $D$ to the definition of bislim
structure (Definition \ref{def:bislim-definition-general}). That
is, a minimal reduced disk diagram $D$ in $X$ without internal 2-cells
such that $\inducedGamma D$ has a directed cycle.
\item For every leaf $v$ of $T_{D}$, consider the path $p_{v}$, which
is a path starting in $v$ of length $R$. We require that $R$ is
large enough so that if $p_{v}$ is internal, $p_{v}$ must contain
at least $4$ pieces.
\item \label{step:containment} We showed in Lemma \ref{lem:large_cells}
that all the edges of $p_{v}$ are internal, so $p_{v}\subseteq T_{D}$.
Thus $p_{v}$ contains at least $4$ pieces in $T_{D}$.
\item \label{step:counting_pieces}Therefore $\#\text{pieces}\largeparens{T_{D}}\ge2\cdot\#\text{leaves}\largeparens{T_{D}}$,
in contradiction with Lemma \ref{lem:counting_pieces}.
\end{enumerate}

As noted in the discussion in Section \ref{subsec:simple-criteria},
the main reason that the simple criterion fails to prove Theorem \ref{thm:generic_bislim_structures_generic_case}
is Condition \ref{enu:radius_condition}. However, Condition \ref{enu:radius_condition}
is required for Step \ref{step:containment} in the program above
(see Lemma \ref{lem:large_cells}).

\begin{wrapfigure}[23]{O}{0.35\columnwidth}%
\begin{centering}
\includegraphics[width=0.35\columnwidth]{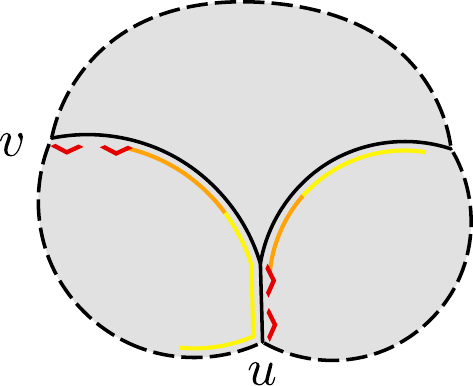}
\par\end{centering}
\caption{\label{fig:overflowing_remainder}Blockers are shown in orange except
their first side which is shown as an extremal side. Remainders are
shown in yellow (see Definition \ref{def:blockers_and_remainders}
below). The remainder $\protect\r v$ escapes out of $T_{D}$ into
the boundary $\protect\boundary D$, and $\protect\r v$ may have
too few internal pieces. To remedy this, the pieces of the next blocker
$\protect\b u$ are attributed to $v$ in addition to the pieces of
$\protect\r v$.}
\end{wrapfigure}%

To overcome this, we relax our requirements, and split up each path
$p_{v}$ into two parts: the first $\bl$ edges are the \emph{blocker}
$\b v$, and the rest of $p_{v}$ is the \emph{remainder} $\r v$.
We take the length of the blockers $\bl$ to be a global parameter
of our method which will be determined later. We still require that
$\b v$ consists only of neutral edges, ensuring that $\b v$ is contained
in $\internalBoundary C$ (Lemma \ref{lem:no_escape}). But $\b v$
is too short, so $\b v$ might contain few pieces. And while the remainder
$\r v$ is still longer than $4$ pieces, it is not required to consist
of neutral edges, so it can escape out of $\internalBoundary C$ (see
Figure \ref{fig:overflowing_remainder} and Definition \ref{def:blockers_and_remainders})
and its pieces might be non-internal. In both cases, we cannot find
4 pieces, and cannot perform the equivalent of steps \ref{step:containment}
and \ref{step:counting_pieces}.

To resolve this issue, in the proof of the relaxed criterion, we count
towards each leaf $v$ more pieces in addition to the pieces contained
in $\r v$. Specifically, if $\r v$ escapes out of $\internalBoundary C$,
it intersects the last vertex $u$ of $\internalBoundary C$ (see
Figure \ref{fig:overflowing_remainder}). In this case we also count
the pieces of $\b u$ (the blocker starting in $u$) towards $\boldsymbol{v}$.
Often this is enough, but we still cannot guarantee that at least
4 pieces would be counted towards $v$ (see also Remark \ref{rem:full_construction_obstruction}).
In Section \ref{subsec:gamma'_and_D'} we use a trick to ``force''
$\b u$ to contain at least 4 pieces, finally guaranteeing that there
would be at least 4 pieces attributed to $v$. In Section \ref{subsec:Proof-of-the-full-criteria}
we finish the proof of Theorem \ref{thm:full-criterion}.

\subsubsection{Blockers}

We now turn to stating Theorem \ref{thm:full-criterion} and some
definitions which are required for the statement. In the following
$\heightened$ is a heightened complex with no trivial cycles.

In the above sketch of the proof of Theorem \ref{thm:simple-criterion},
$p_{v}$ is a path in $\skeleton D$, but in fact it is always contained
in the boundary $\boundary C$ of some 2-cell $C$. Moreover, when
we count the pieces in $p_{v}$, we conclude that each piece is counted
at most twice, once from each neighboring 2-cell. This leads us to
define:
\begin{defn}
Roughly, a \emph{sided path} in a 2-complex $X$ is a path in $X$
that goes along the boundary of a specific 2-cell. Formally, a \emph{sided
path} in a 2-complex $X$ is a pair $\smallparens{p,C}$ of a 2-cell
$C$ of $X$ and a path $p\subseteq\boundary C$.

A \emph{sided piece $\sidedpiece$} in a disk diagram $D$ is a sided
path $\smallparens{p,C}$ in $D$ where $p$ is a piece. We say that
two sided paths $\smallparens{p,C}$ and $\smallparens{p',C'}$ \emph{intersect}
at an \textit{unsided} path $x$ if $x$ is a subpath of both $p$
and $p'$.

We note that we consider paths and sided paths to be directed, though
we consider pieces and sided pieces to be undirected.
\end{defn}

We fix two positive integer parameters, $\mathemph{\bl}$ and $\mathemph{\rl}$,
which are the length of all blockers, and the length of all remainders,
respectively. To prevent clutter, we do not explicitly index blockers
and remainders by these parameters, as they remain fixed throughout
Section \ref{subsec:The-final-criteria}.

\begin{defn}
\label{def:blockers_and_remainders}Let $X$ be a 2-complex. A \emph{blocker}
in $X$ is a sided path $\smallparens{b,C}$ of length $\bl$ where
the first side of $b$ is extremal. A \emph{remainder} is a sided
path $\smallparens{r,C}$ of length $\rl$ continuing from a blocker.

\end{defn}

We now state the relaxed criterion.
\begin{thm}[The relaxed criterion]
\label{thm:full-criterion}Let $\heightened$ be a heightened complex.
Fix some choice of $\bl,\rl\in\mathbb{N}$ and a specific neutral
edge $\a$ with both of its endpoints in $X$ are the same, and fix
a direction of $\a$. If the following conditions hold, then $\heightened$
is bislim.
\begin{enumerate}
\item \label{enu:no_trivial_cycles-full}$\heightened$ has no trivial cycle.
\item \label{enu:small-condition-full}The maximal nontrivial\footnote{An intersection is trivial if it is the intersection between the remainder
and itself.} intersection length between a remainder of $X$ and any boundary
$\boundary C$ of a 2-cell $C$ of $X$ is at most $\frac{\rl}{5}$.
\item \label{enu:structure_of_blockers_condition-full}Every blocker in
$X$ is of the form $s\a^{\bl-1}$ if $s$ is high, or $s\a^{-\bl+1}$
if $s$ is low. I.e., the rest of the path traverses the edge $\a$
repeatedly, forwards if $s$ is high, and backwards if $s$ is low.
\item \label{enu:tiny_condition-full}Remainders do not contain $\a^{\frac{\bl}{5}}$
or $\a^{-\frac{\bl}{5}}$ as subpaths. In particular, the maximal
intersection length of a blocker and a remainder is at most $\frac{\bl}{5}$.
\item \label{enu:extremal-edge-distance-condition-full}Extremal sides of
the same 2-cell $C$ are separated by a distance of at least $2\rl+2\bl$
within $\boundary C$. I.e., boundaries and remainders of the same
2-cell $C$ are disjoint, other than the trivial case of 2 blockers
of the same extremal side going in opposite directions.
\end{enumerate}
\end{thm}

Since all blockers consist just of the single neutral edge $\a$,
blockers can generally have large intersections, which is inconvenient.
However, we show that blockers do not intersect \emph{in minimal counterexamples,}
as follows.
\begin{lem}
\label{lem:no_blocker_intersection}Let $X$ be a 2-complex satisfying
Conditions \ref{enu:structure_of_blockers_condition-full}, \ref{enu:tiny_condition-full}
and \ref{enu:extremal-edge-distance-condition-full}, and let $D$
be a minimal counterexample for $X$. Blockers in $D$ do not intersect
unless they start on the same edge.
\end{lem}

\begin{proof}
Consider two blockers $b_{1}$ and $b_{2}$ in $T_{D}$, which intersect
on an edge $e$ that they do not start on, such as in Figure \ref{fig:crossing_blockers}.
The edge $e$ must be internal since otherwise $b_{1},b_{2}\subseteq\boundary C$
for the same 2-cell $C$, which is ruled out by Condition \ref{enu:extremal-edge-distance-condition-full}.
By Condition \ref{enu:structure_of_blockers_condition-full}, $e$
is mapped into the special neutral edge $\a$ when mapped into $X$.
We give $e$ a fixed direction induced by the fixed direction of $\a$
when mapped into $D$.

Recall that each blocker $b_{i}$ starts at an extremal side $s_{i}$.
Assume that $s_{1}$ or $s_{2}$ is non-internal, say $s_{1}$. Denote
by $e'$ the first edge in $\b 1\cap T_{D}$. By Lemma \ref{lem:leaf_edge_lemma},
the edge $e'$ has an extremal side. But $e'$ must map to $\a$,
and thus $e'$ is neutral, a contradiction. Therefore $\b 1$ does
not intersect $T_{D}$ at all, in contradiction of $e$ being internal.

By Lemma \ref{lem:leaf_edge_lemma}, either the first or the second
vertex of $\b 1$ is a leaf of $T_{D}$. Assume that it is the second
vertex of $\b 1$, so the second side of $\b 1$ is non-internal.
Then as before, denote by $e'$ the first edge in $\b 1\cap T_{D}$
other than $s_{1}$ itself, and a contradiction follows the same way.

The same applies to $\b 2$, and so, both $b_{1}$ and $b_{2}$ start
at leaves $v_{1}$ and $v_{2}$ of $T_{D}$. The edge $e$ where they
intersect splits $T_{D}$ into two halves. Two possible cases are
shown in Figure \ref{fig:crossing_blockers_x} and Figure \ref{fig:crossing_blockers_y}
depending on whether $v_{1}$ and $v_{2}$ are in the same half of
$T_{D}$. Note that $\inducedGamma D$ is a cycle going around $D$,
and so has an orientation, and without loss of generality it is oriented
clockwise. The orientation of $\inducedGamma D$ determines whether
$s_{1}$ and $s_{2}$ are high or low, and this further determines
the direction of the edges of the blockers by Condition \ref{enu:structure_of_blockers_condition-full}.
E.g., in Figure \ref{fig:crossing_blockers_x} $s_{1}$ must be low
and thus the edges of its blocker $\b 1$ are labeled by $\a$ and
directed towards $s_{1}$, and $s_{2}$ must be high and thus the
edges of $\b 2$ are labeled by $\a$ and directed away from $s_{2}$.
Since $e$ is contained in both $\b 1$ and $\b 2$, this is a contradiction,
since the edge $e$ cannot be directed both towards $s_{1}$ and towards
$s_{2}$. 

In both cases (depending on whether $s_{1}$ and $s_{2}$ are in the
same half of $T_{D}$), $e$ cannot agree with both $\b 1$ and $\b 2$,
reaching a contradiction.
\end{proof}
\begin{figure}
\subfloat[\label{fig:crossing_blockers_x}]{\begin{centering}
\vspace{0bp}
\includegraphics[width=0.45\columnwidth]{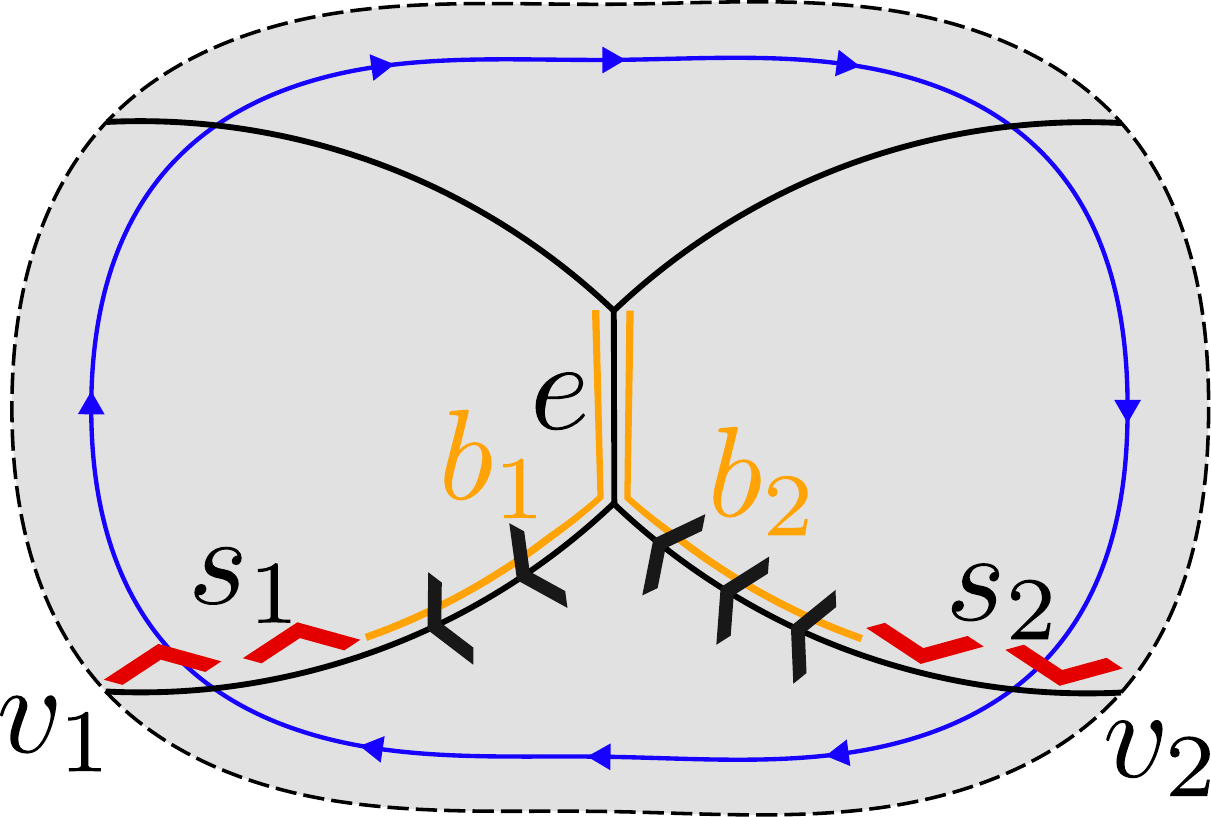}
\par\end{centering}
}\hfill{}\subfloat[\label{fig:crossing_blockers_y}]{\begin{centering}
\vspace{0bp}
\includegraphics[width=0.45\columnwidth]{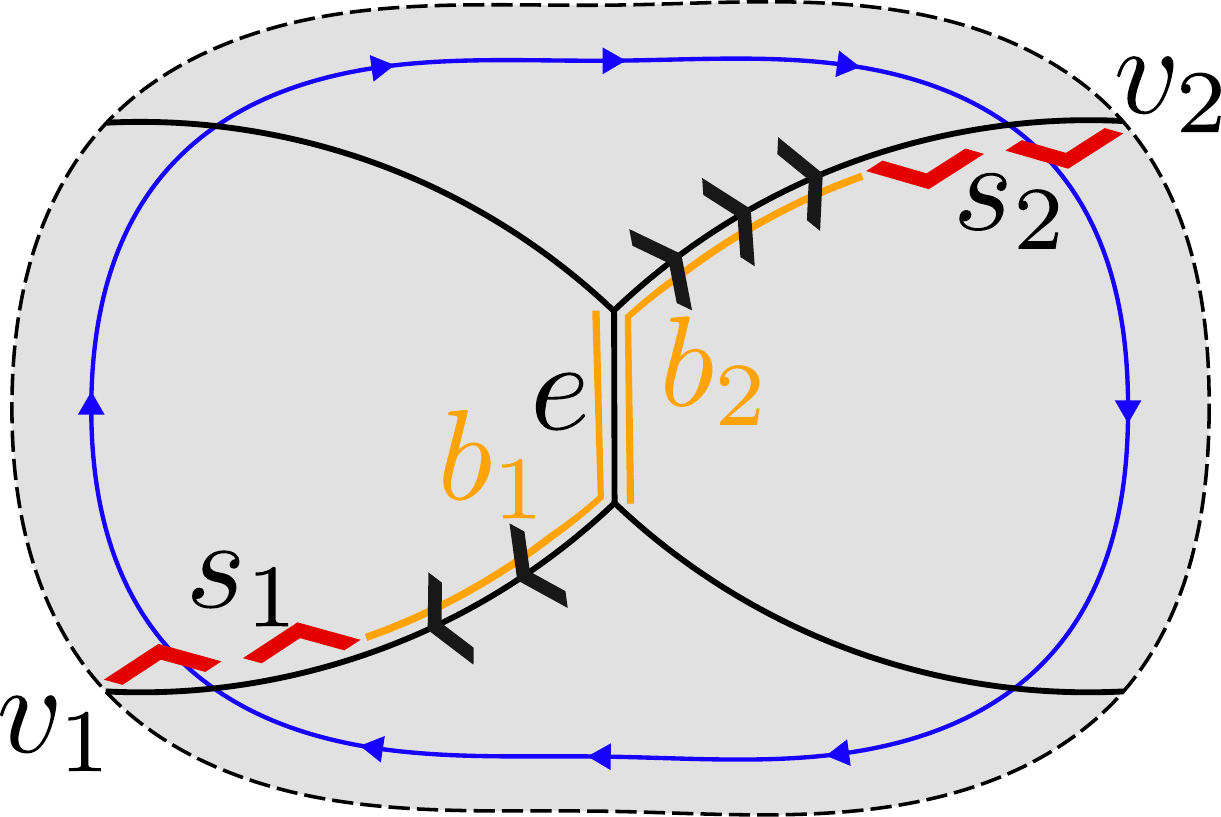}
\par\end{centering}
}

\caption{\label{fig:crossing_blockers}Figures for Lemma \ref{lem:no_blocker_intersection}.
The induced graph $\protect\inducedGamma D$ is marked in blue, extremal
sides are marked by red arrows, blockers are marked in orange, and
edges which map to the neutral edge $\protect\a$ are marked with
a single arrow, in the direction in which they map to $\protect\a$.
In both cases, the edge $e$ cannot have a consistent direction, giving
a contradiction.}
\end{figure}

\begin{rem}[An aside on the necessity of $\Dprime$]
\label{rem:full_construction_obstruction}We remark that it's impossible
to prove Theorem \ref{thm:full-criterion} in a way which is too similar
to the proof of Theorem \ref{thm:simple-criterion}, using a local
counting argument on a minimal counterexample $D$. This obstruction
is avoided in Section \ref{subsec:gamma'_and_D'} by introducing a
new disk diagram $\Dprime$, and applying a local counting argument
on it instead of $D$. This remark is included to motivate the definition
of $\Dprime$ in Section \ref{subsec:gamma'_and_D'}. However, this
remark is not needed for a formal proof of Theorem \ref{thm:full-criterion},
and can be safely skipped. 

Assume that there is an argument counting the number of sided pieces
in a reduced disk diagram $D$ where $\inducedGamma D$ has a cycle.
Also assume that similarly to the counting argument in the proof of
Theorem \ref{thm:simple-criterion}, our hypothetical counting argument
satisfies the following:
\begin{itemize}
\item It attributes at least four sided pieces per leaf of $T_{D}$.
\item Moreover, it can be strengthened to $5$ sided pieces per leaf, given
that the Conditions \ref{enu:small-condition-full} and \ref{enu:tiny_condition-full}
of Theorem \ref{thm:full-criterion} are strengthened accordingly.
Indeed an analogue of this holds for the proof of Theorem \ref{thm:simple-criterion}:
strengthening Condition \ref{enu:small_cancellation_condition} of
Theorem \ref{thm:simple-criterion} can be used to attribute $5$
(or arbitrarily many) sided pieces per leaf instead of $4$.
\item It is local: the argument only considers a bounded ``neighborhood''
of $v$ to choose which sided pieces are attributed to $v$. Distance
is measured by giving every piece of $D$ a length of $1$.
\end{itemize}
We will now use this counting argument to derive a contradiction.
In Figure \ref{fig:full_construction_obstruction} we show how to
construct a disk diagram $D$ where $\inducedGamma D$ has a large
almost-cycle, in the sense that with the addition of one edge $e$,
$\inducedGamma D\cup\left\{ e\right\} $ has an arbitrarily long cycle.
We cannot apply our counting argument to $D$ directly, since $\inducedGamma D$
does not contain a cycle. However, since the counting argument is
local, its logic holds on all 2-cells which are far enough from $e$,
so it should work correctly for all but a constant number $C$ of
2-cells. Therefore our hypothetical counting argument obtains the
following:
\[
2\text{\#pieces}\largeparens{T_{D}}=\text{\#sided-pieces}\largeparens{T_{D}}\ge5\largeparens{\text{\#leaves}\largeparens{T_{D}}-C}.
\]
By Lemma \ref{lem:counting_pieces},
\[
\text{\#pieces}\largeparens{T_{D}}<2\text{\#leaves}\largeparens{T_{D}}
\]
and together we get
\[
\text{\#leaves}\largeparens{T_{D}}<5C,
\]
i.e., the number of leaves in $T_{D}$ is bounded. However, looking
at our construction of $D$ in Figure \ref{fig:full_construction_obstruction},
the number of leaves in $T_{D}$ is clearly unbounded, a contradiction.

\begin{figure}
\centering{}\subfloat[\label{fig:full_construction_obstruction}A disk diagram $D$ where
$\protect\inducedGamma D$ has an almost-cycle. I.e., $\protect\inducedGamma D$
with the addition of the green edge has a cycle. The number of \textquotedblleft rectangles\textquotedblright{}
and the size of the almost-cycle can be increased arbitrarily.]{\centering{}\includegraphics[width=0.45\columnwidth]{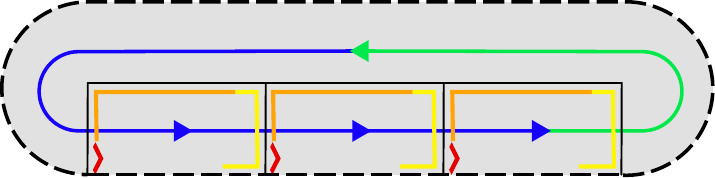}}\hfill{}\subfloat[\label{fig:full_construction_obstruction_with_gamma_prime}The same
with the new edges of $\protect\inducedprimeGamma D$ in teal. The
corresponding diagram $\protect\Dprime$ consists of only the top
2-cell and the rightmost 2-cell.]{\centering{}\includegraphics[width=0.45\columnwidth]{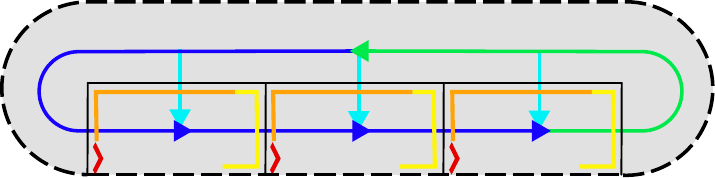}}\caption{The obstruction that explains why the use of $\protect\Gammaprime$
and $\protect\Dprime$ is necessary. Recall that blockers are shown
in orange except their first side which is shown as an extremal side
in red, and remainders are shown in yellow. }
\end{figure}
\end{rem}

\subsubsection{\label{subsec:gamma'_and_D'}$\protect\Gammaprime$ and $\protect\Dprime$}

As explained in Remark \ref{rem:full_construction_obstruction}, working
with a minimal counterexample $D$ is not sufficient to prove Theorem
\ref{thm:full-criterion}. In this section we overcome this obstruction.
We add new edges to $\inducedGamma D$ to obtain a new graph $\mathemph{\inducedprimeGamma D}$.
We then define $\mathemph{\Dprime}$ to roughly be a minimal disk
diagram where $\inducedprimeGamma{\Dprime}$ contains a cycle.
\begin{defn}
\label{def:gamma_prime}Let $D$ be a disk diagram in $X$. The graph
$\mathemph{\inducedprimeGamma D}$ is a directed graph where its vertices
are 2-cells of $D$ and its edges are the edges of $\inducedGamma D$
with additional edges as follows. For every sided piece $\sidedpiece=\smallparens{p,C_{1}}$
in $D$, let $C_{2}$ the 2-cell on the other side of $p$. If $\boundary C_{1}$
contains a blocker $b$ such that $\left|b\cap\sidedpiece\right|>\frac{\bl}{5}$,
we add an edge between $C_{1}$ and $C_{2}$. To determine the edge's
direction, consider the extremal side $s$ in $b$ ($s$ may be outside
of $\sidedpiece$). If $s$ is high, orient the edge from $C_{1}$
to $C_{2}$. If $s$ is low, orient the edge from $C_{2}$ to $C_{1}$.
See Figure \ref{fig:Gamma_prime}.
\end{defn}

Consider a counterexample $D$. Since $\inducedGamma D\subseteq\inducedprimeGamma D$
both of them have a cycle. Consider an edge $e$ in $\inducedprimeGamma D-\inducedGamma D.$
The edge $e$ closes a new cycle in $\inducedprimeGamma D$, and so
it is possible to remove part of $D$ while retaining a cycle in $\inducedprimeGamma D$.
As an example, in Figure \ref{fig:Gamma_prime}, the rightmost 2-cell
can be removed, and the new disk diagram $D'$ would retain a cycle
in $\Gammaprime\smallparens{D'}$. This way we cut off parts of the
original disk diagram $D$ to obtain a new diagram $D'$ where $\inducedprimeGamma{\Dprime}$
has a cycle. Since we started with $D$, we conserve Lemma \ref{lem:no_blocker_intersection}:
No two blockers intersect unless they start on the same edge. Thus
we formally define $\Dprime$ as follows:
\begin{defn}
\label{def:D_prime}Let $\mathemph{\Dprime}$ be a reduced disk diagram
without internal 2-cells where:
\begin{itemize}
\item $\inducedprimeGamma{\Dprime}$ has a cycle,
\item $\Dprime$ satisfies Lemma \ref{lem:no_blocker_intersection}, i.e.,
no two blockers intersect unless they start on the same edge,
\end{itemize}
and finally, $\mathemph{\Dprime}$ is minimal with regards to the
number of internal edges plus the number of 2-cells.
\end{defn}

Formally, $\Dprime$ exists since our minimal counterexample $D$
satisfies both points, even though $D$ itself is not necessarily
minimal with respect to these points.\footnote{As a reminder, $D$ exists under the assumption that our heightened
complex $\heightened$ is not bislim.} The requirement that Lemma \ref{lem:no_blocker_intersection} holds
for $\Dprime$ is necessary, since otherwise $\Dprime$ could be a
degenerate diagram such as in Figure \ref{fig:degenerate_gamma__cycles}.
Such degenerate diagrams always exist, even when $\heightened$ is
indeed bislim.

\begin{figure}
\begin{minipage}[t]{0.55\columnwidth}%
\begin{center}
\includegraphics[width=0.75\textwidth]{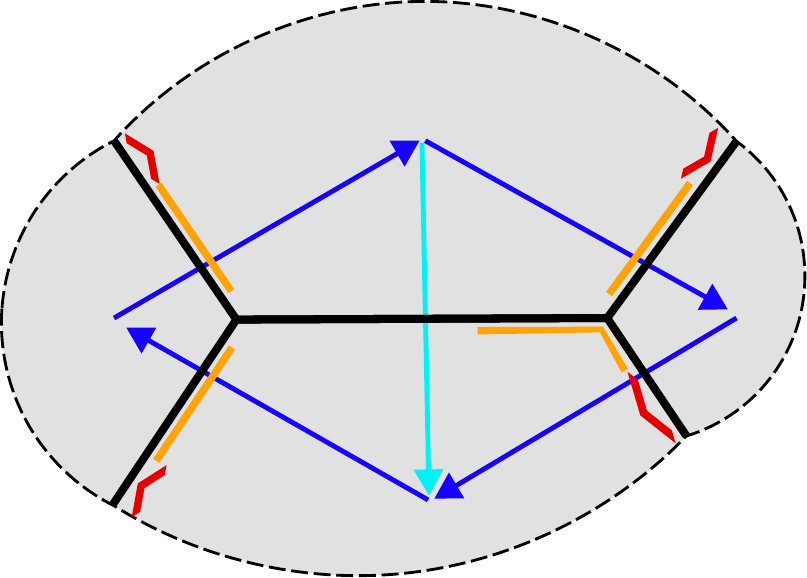}
\par\end{center}
\caption{\label{fig:Gamma_prime}A demonstration of Definition \ref{def:gamma_prime}.
The single edge of $\protect\Gammaprime$ which is not already present
in $\Gamma$ is drawn in teal. That edge is present in $\protect\Gammaprime$
since a blocker of the bottom 2-cell has a large intersection with
the top 2-cell.}
\end{minipage}\hfill{}%
\begin{minipage}[t]{0.35\columnwidth}%
\begin{center}
\includegraphics[width=0.857\textwidth]{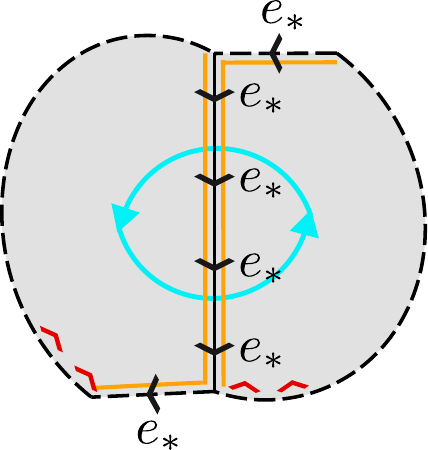}
\par\end{center}
\caption{\label{fig:degenerate_gamma__cycles}Consider this disk diagram $D'$.
Both blockers (in orange) generate edges of $\protect\inducedprimeGamma{\protect\Dprime}$
(in teal) which form a cycle. Degenerate diagrams like this exist
even when $\protect\heightened$ is bislim. We mark the label $\protect\a$
by a single arrow.}
\end{minipage}
\end{figure}

Now we generalize Theorem \ref{lem:leaf_edge_lemma} to $\Dprime$:
\begin{lem}
\label{lem:leaf_edge_lemma_full}Let $v$ be a leaf of $T_{\Dprime}$,
and let $e$ be the edge in $T_{\Dprime}$ adjacent to $v$. Then
$e$ is contained in a blocker $\mathemph{\b v}$ which is directed
from $v$ to $e$.
\end{lem}

Note that $v$ is not necessarily the first vertex of $\b v$. We
also denote the remainder following $\b v$ by $\mathemph{\r v}$.
\begin{proof}
Unglue the 2-cells adjacent to $e$ to obtain a smaller disk diagram
$\Dprime'$. By minimality of $\Dprime$, the graph $\inducedprimeGamma{\Dprime'}$
is acyclic. Thus there is an edge in $\inducedprimeGamma{\Dprime}$
missing from $\inducedprimeGamma{\Dprime'}$. There are two ways that
this edge could be induced in $\inducedprimeGamma{\Dprime}$, but
not in $\inducedprimeGamma{\Dprime'}$:
\begin{casenv}
\item \label{case:full_blocker}$\inducedGamma{\Dprime}$ also has this
edge. Thus $e$ has an extremal side which induces this edge. Thus
$\b v$ is the blocker starting at $v$ and $e$.
\item \label{case:cut_blocker}This edge is induced by a sided piece $\sidedpiece=\smallparens{p,C}$
and a blocker $b$ where $\left|p\cap b\right|>\frac{\bl}{5}$, but
$\left|\smallparens{p\backslash e}\cap b\right|\le\frac{\bl}{5}$.
In particular $e$ is contained in $b$. We claim that $b$ satisfies
the required conditions: it remains to see that the direction of $b$
is from $v$ to $e$.

Assume that $b$ is oriented the other way, from $e$ to $v$, so
$e$ is the last internal edge of $b$. Consider the first side $s$
of $b$, which must be extremal.
\begin{itemize}
\item Assume $s\in\internalBoundary C$. Then $s$ also induces an edge
in $\inducedprimeGamma{\Dprime}$, which is in the same direction
as the edge induced by $\sidedpiece$ (either both are directed into
$C$ or both are directed out of $C$). The edge induced by $s$ cuts
through the cycle of $\inducedprimeGamma{\Dprime}$, so $\inducedprimeGamma{\Dprime}$
contains a cycle which goes through the edge induced by $s$. Thus
it does not contain the edge induced by $\sidedpiece$. Therefore
this cycle is included in $\inducedprimeGamma{\Dprime'}$, in contradiction
to the minimality of $\Dprime$.
\item If $s\not\in\internalBoundary C$, then $\internalBoundary C\subseteq b$.
Since blockers do not intersect (Since $\Dprime$ satisfies Lemma
\ref{lem:no_blocker_intersection}) the only edge from or into $C$
in $\inducedprimeGamma{\Dprime}$ is the edge induced by $\sidedpiece$.
Therefore the cycle of $\inducedprimeGamma{\Dprime}$ does not go
through $C$ and $C$ can be removed from $\Dprime$, in contradiction
with the minimality of $\Dprime$.
\end{itemize}
\end{casenv}
\end{proof}
See Figure \ref{fig:full_vs_cut_blockers} for demonstration of the
two cases. We now divide the blockers $\b v$ into two kinds, corresponding
to the two cases of the proof above. For now, let $v$ be a leaf of
$T_{\Dprime}$.
\begin{defn}
\label{def:blocker-types}We say that $\b v$ is a \emph{full blocker}
if it corresponds to Case \ref{case:full_blocker} in the proof above,
and we say that $\b v$ is a \emph{cut blocker} if it corresponds
to Case \ref{case:cut_blocker} in the proof above.

\end{defn}

\begin{figure}
\subfloat[\label{fig:full_blocker}]{\begin{centering}
\includegraphics[width=0.4\columnwidth]{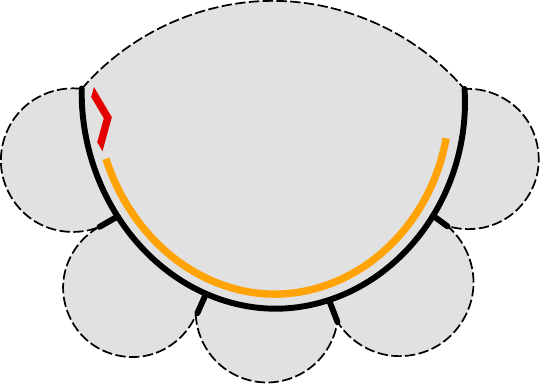}
\par\end{centering}
}\hfill{}\subfloat[\label{fig:cut_blocker}]{\begin{centering}
\includegraphics[width=0.4\columnwidth]{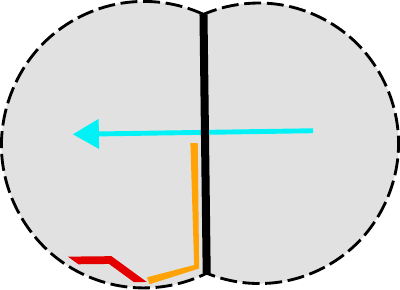}
\par\end{centering}
}\caption{\label{fig:full_vs_cut_blockers}Figure \ref{fig:full_blocker} shows
Case \ref{case:full_blocker} of Lemma \ref{lem:leaf_edge_lemma_full},
a full blocker. A full blocker must have at least five pieces by Lemma
\ref{lem:full_blocker_has_short_pieces}. Figure \ref{fig:cut_blocker}
shows Case \ref{case:cut_blocker}, a cut blocker.}
\end{figure}

Informally, we call cut blockers ``cut'' since they have been ``cut''
by the new edges of $\Gammaprime$. We now continue the logic of the
proof above and obtain that full blockers cannot have any long internal
pieces:

\begin{lem}
\label{lem:full_blocker_has_short_pieces}If $\b v$ is a full blocker,
then $\b v$ does not contain internal pieces longer than $\frac{\bl}{5}$.
\end{lem}

\begin{proof}
Let $e$ be the edge of $T_{\Dprime}$ adjacent to $v$ as before.
Denote by $s$ the side of $e$ which is the first side of $\b v$,
which must be extremal. Let $\sidedpiece$ be an internal sided piece
which is longer than $\frac{\bl}{5}$ and contained in $\b v$. Using
the same logic as in the proof of Lemma \ref{lem:leaf_edge_lemma_full},
both $s$ and $\sidedpiece$ induce edges in $\inducedprimeGamma{\Dprime}$
in the same direction. Thus there is a cycle in $\inducedprimeGamma{\Dprime}$
which goes through the edge induced by $\sidedpiece$ and not the
edge induced by $s$, and by ungluing at $e$ we get a smaller disk
diagram where $\Gammaprime$ still has a cycle, in contradiction with
the minimality of $\Dprime$.
\end{proof}
We now prove that blockers $\b v$ do not escape $T_{\Dprime}$, similarly
to Lemma \ref{lem:large_cells}. In fact, we prove that they do not
escape in a strict sense: $\b v$ does not go through a leaf of $T_{\Dprime}$
except for $v$.
\begin{lem}
\label{lem:no_escape} Blockers of the form $\b v$ do not go through
a leaf of $T_{\Dprime}$ other than $v$.
\end{lem}

\begin{proof}
Assume that $\b v$ escapes out of $T_{\Dprime}$, and let $v'$ be
the last vertex in $\b v\cap T_{\Dprime}$. By the same proof as in
Lemma \ref{lem:boundary_leaf_lemma} and Corollary \ref{cor:boundary_leaf},
it follows that $v'$ is a leaf of $T_{\Dprime}$. Consider $\b{v'}$.
At $v'$, the blocker $\b v$ is exiting $T_{\Dprime}$, while $\b{v'}$
is entering $T_{\Dprime}$ so they have opposite directions. By Lemma
\ref{lem:no_blocker_intersection} this is only possible if both blockers
start at the same edge $e$. But then $v$ and $v'$ are both adjacent
to $e$, so two leaves of $T_{\Dprime}$ are neighbors, and $T_{\Dprime}$
consists of just $v$, $v'$ and $e$. This is impossible by the same
proof as Lemma \ref{lem:length_two}.
\end{proof}
\begin{cor}
\label{cor:blockers_are_contained}If $\b v$ is a full blocker, $\b v$
is contained in $T_{\Dprime}$.
\end{cor}

\subsubsection{\label{subsec:Proof-of-the-full-criteria}Proof of the relaxed criterion
for bislim structures}

We now prove that we can indeed attribute at least 4 sided pieces
of $T_{\Dprime}$ to each leaf $v$ of $T_{\Dprime}$. First, we recall
the strategy which we use to attribute sided pieces of $T_{\Dprime}$
to leaves $v$ of $T_{\Dprime}$. For each leaf $v$ of $T_{\Dprime}$
we attribute all of the internal sided pieces of the remainder $\r v$
to $v$. We stress that we do not attribute the internal sided pieces
of $\b v$ to $v$.

If the remainder $\r v$ escapes $T_{\Dprime}$, we need to attribute
additional pieces to $v$. Let $u$ be the first vertex in $\r v\cap\boundary D$.
It follows by the same logic as in Lemma \ref{lem:boundary_leaf_lemma}
and Corollary \ref{cor:boundary_leaf} that $u$ is also a leaf of
$T_{\Dprime}$. We attribute the internal pieces of $\b u$ to $\mathemph v$
(and not $u$).

From now on we assume that all of the conditions of Theorem \ref{thm:full-criterion}
hold.
\begin{lem}
\label{lem:count_at_least_4_full}Every leaf $v$ of $T_{\Dprime}$
has at least $4$ sided pieces of $T_{\Dprime}$ attributed to $v$.
\end{lem}

\begin{proof}
If the remainder $\r v$ of $v$ does not escape $T_{\Dprime}$, then
since $\r v$ is fully internal and $\left|\r v\right|=\rl$, by Condition
\ref{enu:small-condition-full} $\r v$ contains at least $4$ sided
pieces. If $\r v$ does escape $T_{\Dprime}$, let $u$ be the first
vertex of $\r v\cap\boundary D$. As before, by the same logic as
in Lemma \ref{lem:boundary_leaf_lemma} and Corollary \ref{cor:boundary_leaf},
$u$ is a leaf of $T_{\Dprime}$. The sided pieces of $\b u$ are
attributed to $v$, so we need to show that there are at least $4$.
Since $\r v$ intersects $\b u$, by Condition \ref{enu:tiny_condition-full}
this intersection is of length at most $\frac{\bl}{5}$. Since this
is the first piece of $\b u$, it follows that $\b u$ is a full blocker
(the first piece of cut blockers must be large, see Case \ref{case:full_blocker}
of proof of Lemma \ref{lem:leaf_edge_lemma_full}). Therefore all
the sided pieces of $\b u$ are of length at most $\frac{\bl}{5}$,
and so, $\b u$ must contain at least $4$ sided pieces. Finally,
by Corollary \ref{cor:blockers_are_contained}, these sided pieces
are all internal.
\end{proof}
\begin{lem}
\label{lem:no_double_count_full}No sided piece $\sidedpiece$ of
$T_{\Dprime}$ is attributed to two different leaves of $T_{\Dprime}$.
\end{lem}

\begin{proof}
Let $\sidedpiece$ be a sided piece. It cannot be contained in more
than wan blocker or remainder of the form $\b v$ or $\r v$. Indeed,
if $\sidedpiece$ is a sided piece of a 2-cell $C$, then By Condition
\ref{enu:extremal-edge-distance-condition-full} blockers and remainders
in $\boundary C$ do not intersect other than the single edgecase
of two blockers starting on the same extremal side, in which case
only one of them is of the form $\b v$ by Lemma \ref{lem:no_escape}.

Therefore if $\sidedpiece\subseteq\r v$ then it is only attributed
to $v$, and we are done. Denote by $e$ the unique internal edge
which is adjacent to $v$. If $\sidedpiece\subseteq\b v$, then $\sidedpiece$
can be attributed to any leaf $u$ of $T_{\Dprime}$ where $\r u$
intersects $\b v$ at $e$. Since blockers and remainders of the same
2-cell do not intersect, $e$ can be contained in at most one remainder
in addition to $\b v$, and therefore $\sidedpiece$ can be attributed
to at most one leaf of $T_{\Dprime}$.
\end{proof}
We now prove Theorem \ref{thm:full-criterion}:
\begin{proof}[Proof of Theorem \ref{thm:full-criterion} ]
Assume that the conditions of Theorem \ref{thm:full-criterion} hold,
and that $\left(X,H,L\right)$ is not bislim, and we will derive a
contradiction. Let $\Dprime$ be as in Definition \ref{def:D_prime},
i.e., $\Dprime$ is a minimal reduced disk diagram without internal
2-cells such that $\inducedprimeGamma{\Dprime}$ has a cycle and $\Dprime$
satisfies Lemma \ref{lem:no_blocker_intersection}. As we have shown,
the lemmas of Section \ref{subsec:gamma'_and_D'} apply.

By Lemmas \ref{lem:count_at_least_4_full} and \ref{lem:no_double_count_full},
we have attributed at least $4$ sided pieces of $T_{\Dprime}$ to
every leaf $v$ of $T_{\Dprime}$. Thus
\[
\#\text{sided pieces}\largeparens{T_{\Dprime}}\ \ge\ 4\cdot\#\text{leaves}\largeparens{T_{\Dprime}},
\]
in contradiction with \ref{lem:counting_pieces}.
\end{proof}

\subsection{\label{subsec:Generic-words-analysis}Generic words have alternating
bislim structures}

\global\long\def\prob#1{\mathbb{P}\left[\,#1\,\right]}%

\global\long\def\pfcycle#1{\overline{E_{2}\smallparens{#1}}}%
\global\long\def\pscycle#1{E_{2}\smallparens{#1}}%

\global\long\def\pfsmall{\overline{E_{1}}}%
\global\long\def\pssmall{E_{1}}%

\global\long\def\pfblockers#1{\overline{E_{3}\smallparens{#1}}}%
\global\long\def\psblockers#1{E_{3}\smallparens{#1}}%

\global\long\def\pftiny#1{\overline{E_{4}\smallparens{#1}}}%
\global\long\def\pstiny#1{E_{4}\smallparens{#1}}%

\global\long\def\pfall#1{\overline{E_{2,3,4}\smallparens{#1}}}%
\global\long\def\psall#1{E_{2,3,4}\smallparens{#1}}%

\global\long\def\fam{\mathcal{A}}%

\global\long\def\pfany{\overline{E_{2,3,4}\smallparens{\fam}}}%
\global\long\def\psany{E_{2,3,4}\smallparens{\fam}}%

\global\long\def\lrtiny{\frac{\bl}{5}}%
\global\long\def\lrsmall{\frac{\rl}{5}}%
\global\long\def\len{\ell}%
\global\long\def\l{\ell}%

In this section we prove Theorems \ref{thm:generic_bislim_structures_one_relator_case}
and \ref{thm:generic_bislim_structures_generic_case} using Theorem
\ref{thm:generic_words_with_explicit_probability}. In fact, including
explicit parameters, we prove the following theorem:
\begin{thm}
\label{thm:generic_words_with_explicit_probability}For every choice
of nonnegative integers $r$, $k$ and $a_{1},\dots,a_{k}$ such that
$\sum_{i}a_{i}<r$, a random $r$-generator $k$-relator presentation
$\left\langle x_{1},\dots,x_{r}\vert w_{1},\dots,w_{k}\right\rangle $
with $\left|w_{i}\right|=\len$ admits an $\left(a_{1},\dots,a_{k}\right)$-alternating
bislim structure with probability at least $1-e^{-\Omega\smallparens{\len^{c}}}$
where $c=\frac{1}{2+20\cdot\sum_{i}a_{i}}$.
\end{thm}

As a reminder, here a random $r$-generator $k$-relator presentation
is given by picking the relators uniformly out of all cyclically reduced
words of length $\len$ in the letters $x_{1},\dots,x_{r}$. The multiplicative
constant in $\Omega\smallparens{\len^{c}}$, and all multiplicative
constant in Big-$O$ and Big-$\Omega$ notations in the following,
are dependent only on $r$ and $k$.

We prove this by showing that for a random presentation complex $X$,
with high probability, the relaxed criterion (Theorem \ref{thm:full-criterion})
applies to at least one $\left(a_{1},\dots,a_{k}\right)$-alternating
heightened structure $\heightened$ on $X$.

Let $X$ be a presentation complex of a presentation $\left\langle x_{1},\dots,x_{r}\vert w_{1},\dots,w_{k}\right\rangle $.
Recall that the skeleton $\skeleton X$ is a bouquet of $r$ circles,
each of them corresponding to the $r$ generators $x_{1},\dots,x_{r}$.
Then $X$ is constructed as $\skeleton X$ with the addition of a
2-cell $C_{i}$ for every relator $w_{i}$, where $\boundary C_{i}$
is glued onto the closed path in $X$ that reads $w_{i}$. Therefore
every side $\s\in\boundary C_{i}$ of $X$ can be described as a pair
$\left(i,j\right)$ where $\s$ is the $j$'th side in $\boundary C_{i}$,
and its corresponding edge corresponds to $w_{i,j}$, the $j$'th
letter in $w_{i}$. We formally define a \emph{place} to be a pair
$\left(i,j\right)$ with $1\le i\le k$ and $1\le j\le\len$. Given
the relators $w_{1},\dots,w_{k}$, the letter at a given place $\left(i,j\right)$
is $w_{i,j}\in\left\{ x_{1},x_{1}^{-1},x_{2},x_{2}^{-1},\dots,x_{r}^{-1}\right\} $.

Fix two sets of places $H$ and $L$. Then, given random relators
$w_{1},\dots,w_{k}$ of length $\len$, we can form the presentation
complex $X$ of $\left\langle x_{1},\dots,x_{r}\vert w_{1},\dots,w_{k}\right\rangle $
and the heightened structure $\heightened$ on $X$. Effectively,
this allows us to fix the heightened complex independently of the
complex $X$, and then choose $X$ randomly. We thus ask, given fixed
sets of positions $H$ and $L$, what is the probability that for
a random complex $X$, $\heightened$ satisfies the relaxed criterion?
We now answer this question.

Note that whether $\heightened$ is $\left(a_{1},\dots,a_{k}\right)$-alternating
depends only on $H$ and $L$ and not on $w_{1},\dots,w_{k}$ or $X$.
Thus in the following we assume that $H$ and $L$ are $\left(a_{1},\dots,a_{k}\right)$-alternating,
and in particular $\left|H\right|=\left|L\right|=\sum_{i}a_{i}<r$.

Recall the definition of blockers and remainders. In this setting,
blockers and remainders are specific cyclic subwords of the relators.\footnote{That is, subwords of a cyclic rotation of $w_{i}$ or $w_{i}^{-1}$
for some $i$.} Specifically, given a position $\left(i,j\right)\in H\cup L$, its
blockers are the two subwords starting at that position of length
$\bl$, i.e., the subwords $w_{i,j}w_{i,j+1}\dots w_{i,j+\bl-1}$
and $\left(w_{i,j-\bl+1}\dots w_{i,j}\right)^{-1}$, and the remainders
are the two subwords of length $\rl$ following them, i.e., $w_{i,j+\bl}\dots w_{i,j+\bl+\rl-1}$
and $\left(w_{i,j-\bl-\rl+1}\dots w_{i,j-\bl}\right)^{-1}$. 

Now we translate the conditions of the relaxed criterion (Theorem
\ref{thm:full-criterion}) into the language of this setting. First
of all, the criterion requires a choice of parameters $\bl$ and $\rl$.
For now, we will leave them as global parameters, and fix them as
specific functions of $\len$ by the end of the proof. Additionally,
there is required to be a specific chosen neutral letter $\a$. We
will just choose it to be the letter $x_{1}$, though we will still
name it $\a$ to express its unique status. In our terms, it means
that $w_{i,j}\neq\a^{\pm}$ for all $\left(i,j\right)\in H\cup L$
where $\a=x_{1}$. The rest of the conditions of the relaxed criterion
are as follows:

\begin{enumerate}
\item \label{enu:no_trivial_cycles-probability}There is no pair of distinct
positions $\left(i,j\right),\left(i',h'\right)\in H$ with $w_{i,j}=w_{i,j'}^{\pm}$.
Similarly there is no such pair with $\left(i,j\right),\left(i',j'\right)\in L$.
\item \label{enu:small-condition-probability}Every nontrivial common subword
of a remainder and any subword of any relator is of length at most
$\lrsmall$.
\item \label{enu:structure_of_blockers_condition-probability}Every blocker
in $X$ is of the form $x\a^{\bl-1}$ if it starts in a position in
$H$, or $x\a^{-\bl+1}$ if it starts in a position in $L$.
\item \label{enu:tiny_condition-probability}Remainders do not contain $\a^{\pm\frac{\bl}{5}}$
as a subword.
\item \label{enu:extremal-edge-distance-condition-probability}Places in
$H\cup L$ in the same relator are separated by a distance of at least
$2\rl+2\bl$. I.e., boundaries and remainders in the same relator
are disjoint.
\end{enumerate}
First of all, Condition \ref{enu:extremal-edge-distance-condition-probability}
depends only on the choice of $H$ and $L$ themselves, independent
of the relators $w_{1},\dots,w_{k}$. In the following we will only
choose pairs $\smallparens{H,L}$ which satisfy Condition \ref{enu:extremal-edge-distance-condition-probability}.

Condition \ref{enu:small-condition-probability} can be strengthened
to require that there is no nontrivial common subword of relators
longer than $\lrsmall$ anywhere in the presentation. This stronger
condition is independent of the choice of $H$ and $L$. Denote the
event that this strengthened condition holds by $\mathemph{\pssmall}$.
It holds with high probability, as shown in small cancellation theory.
Denote the complement of an event $E$ by $\mathemph{\overline{E}}$.
\begin{lem}
$\prob{\pfsmall}\le O\largeparens{\len^{2}\smallparens{2r-1}^{-\lrsmall}}$
\end{lem}

\begin{proof}
There are at most $2{k\len \choose 2}$ possibilities for two locations
of two appearances of a common subword. For each specific pair of
locations, the probability that the two subwords of length $\frac{\rl}{5}$
in these locations will be the same is at most $\smallparens{2r-1}^{-\lrsmall}$.
Thus by the union bound,
\[
\prob{\pfsmall}\le O\largeparens{\len^{2}\smallparens{2r-1}^{-\lrsmall}}.\qedhere
\]
\end{proof}
Denote by $\mathemph{\pscycle{H,L}}$ the event that the letter $\a=x_{1}$,
is indeed neutral, i.e., $w_{i,j}\neq\a^{\pm}$ for all $\left(i,j\right)\in H\cup L$,
and additionally, that Condition \ref{enu:no_trivial_cycles-probability}
holds. This event means that the letters $w_{i,j}$ for $\left(i,j\right)\in H$
are distinct (ignoring signs), and distinct from $\a$, and similarly
the same for $L$. Since the number of possible letters is $r$ and
$\left|H\right|=\left|L\right|=\sum_{i}a_{i}<r$, this is possible,
so $\mathemph{\pscycle{H,L}}$ happens with positive probability.
Moreover, the event depends only on less than $2r$ letters $w_{i,j}$
for places $\left(i,j\right)\in H\cup L$, and so $\prob{\pscycle{H,L}}$
depends only on $r$, i.e., holds with a positive constant probability.

In the following, it will be convenient to assume that for any ``far''
events are approximately independent. Suppose that an event $E$ depends
only on the subword $w_{a}w_{a+1}\dots w_{b}$ and an event $E'$
depends only on the subword $w_{a'}w_{a'+1}\dots w_{b'}$, for a random
reduced word $w$. Let $\mu_{\left[x,y\right]}$ be the distribution
of $w_{x}\dots w_{y}$. By bounding the total variation distance between
the distribution of $\left(w_{a}\dots w_{b},w_{a'}\dots w_{b'}\right)$
and the product distribution $\mu_{\left[a,b\right]}\times\mu_{\left[a',b'\right]}$
by $\varepsilon$, we can conclude that $\left|\prob{E\cap E'}-\prob E\prob{E'}\right|\le\varepsilon$.
By a simple computation, this total variation distance is indeed bounded
by $e^{-\Omega\smallparens d}$ where $d$ is the distance between
$\left[a,b\right]$ and $\left[a',b'\right]$. Since in the following
these error terms end up being much smaller than our actual bounds,
we choose to treat far events as just approximately independent events.

It remains to analyze the probabilities of Conditions \ref{enu:structure_of_blockers_condition-probability}
and \ref{enu:tiny_condition-probability}. Denote by $\mathemph{\psblockers{H,L}}$
and $\mathemph{\pstiny{H,L}}$ the events that the respective conditions
hold.

\begin{lem}
\label{lem:event_3}Given a pair $\left(H,L\right)$, 
\[
\prob{\psblockers{H,L}\vert\pscycle{H,L}}\ge\Omega\left(\smallparens{2r-1}^{-4\bl\cdot\sum_{i}a_{i}}\right).
\]
\end{lem}

\begin{proof}
Since $\left(H,L\right)$ satisfy Condition \ref{enu:extremal-edge-distance-condition-probability},
the places in $H\cup L$ are far from each other, which makes the
behavior in each of their neighborhoods independent up to an error
on the order of $e^{-\Omega\smallparens{\rl}}$, which we can ignore,
as discussed above. Around each place $\left(i,j\right)$ in $H\cup L$,
the two blockers of this extremal place continue for a distance of
$\bl$ in either direction, and Condition \ref{enu:structure_of_blockers_condition-probability}
requires each of these letters to be $\a$ in a specific direction.
Note that if $w_{i,j}$ itself was either $\a$ or $\a^{-1}$, this
would be impossible, as $w_{i}$ would not be a reduced word. However,
since Condition \ref{enu:no_trivial_cycles-probability} holds, $w_{i,j}\neq\a,\a^{-1}$.
Thus the probability that Condition \ref{enu:structure_of_blockers_condition-probability}
holds around $\left(i,j\right)$ is $\Omega\largeparens{\smallparens{2r-1}^{-2\bl}}$.
Since there are $2\sum_{i}a_{i}$ places in $H\cup L$, the result
follows.
\end{proof}
\begin{lem}
Given a pair $\left(H,L\right)$, 
\[
\prob{\pstiny{H,L}\vert\psblockers{H,L}\cap\pscycle{H,L}}\ge1-O\largeparens{\rl\cdot\smallparens{2r-1}^{-\lrtiny}}.
\]
\end{lem}

\begin{proof}
Condition \ref{enu:extremal-edge-distance-condition-probability}
requires that the remainders do not contain $\a^{\pm\lrtiny}$ as
a subword. We note that the events $\psblockers{H,L}$ and $\pscycle{H,L}$
depend only on the values of the letters on the blockers, which are
disjoint from the remainders. Thus conditioning on $\psblockers{H,L}$
and $\pscycle{H,L}$ still leaves the distribution of the adjacent
remainders largely uniform.

As before, there are only $4\sum_{i}a_{i}\cdot\rl=O\smallparens{\rl}$
possible places for $\a^{\pm\lrtiny}$ to appear as a subword of a
remainder, and in each specific place, the probability of this subword
to appear is at most $\smallparens{2r-1}^{-{\tiny \lrtiny}}$. By
the union bound, we obtain that $\prob{\pftiny{H,L}\vert\psblockers{H,L}\cap\pscycle{H,L}}\le O\largeparens{\rl\cdot\smallparens{2r-1}^{-\lrtiny}}$.
\end{proof}
In light of this lemma, we will set $\bl$ such that $\smallparens{2r-1}^{-\lrtiny}=\frac{C}{\rl}$
with a small enough multiplicative constant $C$, so that $\prob{\pstiny{H,L}\vert\psblockers{H,L}\cap\pscycle{H,L}}\ge\frac{1}{2}$
for large enough $\rl$. Thus $\bl=\Theta\smallparens{\log\rl}$.

Denote by $\psall{H,L}$ the event that conditions \ref{enu:no_trivial_cycles-probability},
\ref{enu:structure_of_blockers_condition-probability} and \ref{enu:tiny_condition-probability}
all hold with respect to $H$ and $L$. Taking the last two lemmas,
together with the fact that $\prob{\pscycle{H,L}}$ is a positive
constant, we obtain the following:
\begin{lem}
Given a pair $\left(H,L\right)$,
\[
\prob{\psall{H,L}}\ge\Omega\largeparens{\smallparens{2r-1}^{-4\bl\cdot\sum_{i}a_{i}}}=\Omega\largeparens{\rl^{-20\cdot\sum_{i}a_{i}}}.
\]
\end{lem}

Applying this bound to a specific pair of $H$ and $L$ is definitely
not enough to prove Theorem \ref{thm:generic_words_with_explicit_probability}.
Instead, we consider a family $\fam$ of pairs $\left\{ \largeparens{H_{j},L_{j}}\right\} _{j}$,
and consider the probability that $\psall{H,L}$ will hold for at
least some $\left(H,L\right)\in\fam$. In order to analyze this, we
require that $\left\{ \psall{H,L}\right\} _{\left(H,L\right)\in\fam}$
forms a family of approximately independent events. We will ensure
this by requiring all places within $\bigcup_{\smallparens{H,L}\in\fam}H\cup L$
to be far apart from each other, specifically with distance of $2\bl+3\rl$.
Overall, we require $\fam$ to satisfy the following:
\begin{itemize}
\item $H$ and $L$ must be $\smallparens{a_{1},\dots,a_{k}}$-alternating
for all $\smallparens{H,L}\in\fam$. In particular, for all $i$,
$\left|H\cap w_{i}\right|=\left|L\cap w_{i}\right|=a_{i}$ and $H\cap w_{i}$
and $L\cap w_{i}$ alternate.
\item For all $\left(H,L\right)\in\fam$, Condition \ref{enu:extremal-edge-distance-condition-probability}
must be satisfied, so the places in $H$ and $L$ must be far enough
apart from each other.
\item All places in $\bigcup_{\smallparens{H,L}\in\fam}H\cup L$ are of
distance at least $2\bl+3\rl$ from each other.
\end{itemize}
To construct such a family $\fam$, first pick $\mathcal{B}$ to be
a set containing $\frac{\len}{2\bl+3\rl}$ places in each relator
$w_{i}$, spaced evenly around $w_{i}$ with distances of at least
$2\bl+3\rl$. Then divide $\mathcal{B}$ into sets $\mathcal{B}_{1},\mathcal{B}_{2},\dots$
of size $2\sum_{i}a_{i}$ each, each containing $2a_{i}$ places in
each relator $w_{i}$. This way, we can get at least $n$ sets $\mathcal{B}_{1},\dots,\mathcal{B}_{n}$
with $n=\frac{\len}{2r\left(2\bl+4\rl\right)}$.

Now divide each set $\mathcal{B}_{j}$ of $2\sum_{i}a_{i}$ places
in an alternating manner into two sets $H_{j}$ and $L_{j}$, so $H_{j}$
and $L_{j}$ are $\left(a_{1},\dots,a_{k}\right)$-alternating. Finally
$\fam=\left\{ \left(H_{j},L_{j}\right)\right\} _{j=1}^{n}$. This
way we get a family $\fam$ of $n=\Omega\largeparens{\frac{\len}{\rl}}$
pairs satisfying all requirements above.

Denote by $\psany$ the event that some pair $\smallparens{H,L}\in\fam$
satisfies Conditions \ref{enu:no_trivial_cycles-probability}, \ref{enu:structure_of_blockers_condition-probability}
and \ref{enu:tiny_condition-probability}. We have that
\[
\prob{\pfany}\le\prod_{\smallparens{H,L}\in\fam}\prob{\pfall{H,L}}+\ell e^{-\Omega(\rl)}
\]

by applying the total variation distance bound in the discussion before
Lemma \ref{lem:event_3} $O(\ell)$ times. Now,
\begin{eqnarray*}
\prod_{\smallparens{H,L}\in\fam}\prob{\pfall{H,L}} & = & \largeparens{1-\prob{\psall{H,L}}}^{\Omega\largeparens{\frac{\len}{\rl}}}\\
 & \le & e^{-\prob{\psall{H,L}}\cdot\Omega\largeparens{\frac{\len}{\rl}}}\\
 & = & e^{-\Omega\largeparens{\len\cdot\rl^{-20\cdot\Sigma_{i}a_{i}-1}}.}
\end{eqnarray*}
This allows us to conclude:
\begin{proof}[Proof of Theorem \ref{thm:generic_words_with_explicit_probability}]
If $X$ does not admit an $\smallparens{a_{1},\dots,a_{k}}$-alternating
bislim structure, then either $\pfsmall$ or $\pfany$ holds. This
happens with probability at most
\[
\len^{2}e^{-\Omega\smallparens{\rl}}+e^{-\Omega\largeparens{\len\cdot\rl^{-20\cdot\Sigma_{i}a_{i}-1}}}+\ell e^{-\Omega(\rl)}.
\]
Equating the two exponents, we set $\rl=\len^{c}$ where $c=\frac{1}{2+20\cdot\sum_{i}a_{i}}$,
obtaining a probability of $e^{-\Omega\largeparens{\len^{c}}}$.
\end{proof}

\section{\label{sec:Applications-to-word-measures}Applications to word measures}

Let $w\in F_{r}$ be a word in the free group with $r$ generators,
and let $\psi_{n}$ be a stable character of $S_{n}$ or $\U n$,
as defined in Section \ref{subsec:intro:Word-measures}. The dimension
$\dim\psi_{n}$ is always a polynomial in $n$,\footnote{For $S_{n}$ this holds by the hook length formula, and for $\U n$
this holds by \cite[Equation~3.10]{Puder2023}.} and we denote by $\mathemph{\deg\psi}$ the degree of this polynomial.

\subsection{\label{subsec:Background}Background}

\global\long\def\summand{\text{P}}%

Recently \cite{cassidy,MageeDeLaSalle2024} have proven striking new
results on word measures in $S_{n}$ and $\U n$, namely that $\mathbb{E}_{w}\left[\psi_{n}\right]=O\largeparens{\frac{1}{\dim\psi_{n}}}$
in the case of $S_{n}$ (\cite[Theorem~1.6]{cassidy}) and $\mathbb{E}_{w}\left[\psi_{n}\right]=O\largeparens{\smallparens{\dim\psi_{n}}^{-\frac{1}{6}}}$
in the case of $\U n$ (\cite[Corollary~11.3]{MageeDeLaSalle2024}),
assuming that $w$ is not a proper power. Both papers \cite{cassidy,MageeDeLaSalle2024}
are based on the same approach. In each paper, the word measure $\expectation{\psi_{n}}$
is first analyzed by clever representation theoretic arguments, and
is decomposed into a finite sum 
\begin{equation}
\expectation{\psi_{n}}=\sum_{H\in\mathcal{A}}\summand_{H}\label{eq:decomposition_into_parts}
\end{equation}
 of many different parts $\summand_{H}$, where each part $\summand_{H}$
is a rational function in $n$. Each part $\summand_{H}$ is indexed
by structured data $H$, which is a certain kind of labeled graph
in the case of $S_{n}$, and certain combinatorial surfaces in the
case of $\U n$. The order of magnitude of $\summand_{H}$ can be
bounded by combinatorial properties of $H$. In the case of $S_{n}$
\cite{cassidy} shows that $\summand_{H}$ can be bounded by applying
a new slight generalization of the $w$-cycle theorem on $H$ to obtain
that $\summand_{H}=O\largeparens{\frac{1}{\dim\psi_{n}}}$, proving
Theorem \cite[Theorem~1.6]{cassidy}. In the case of $\U n$, a crude
version of the $w$-cycle theorem is applied, to obtain that $\summand_{H}=O\largeparens{\smallparens{\dim\psi_{n}}^{-\frac{1}{6}}}$,
proving \cite[Corollary~11.3]{MageeDeLaSalle2024}.

In fact, in the discussion of \cite{MageeDeLaSalle2024} the authors
suspect that for $\U n$ the stronger bound of $O\largeparens{\frac{1}{\dim\psi_{n}}}$
also holds. We show in Theorem \ref{thm:case_of_Un_non_power} that
the methods of \cite{MageeDeLaSalle2024} can be extended using existing
techniques (see Remark \ref{rem:existing_methods_remark}) to prove
the analogous result:
\begin{thm}
\label{thm:case_of_Un_non_power}Let $w$ be a non-power word and
let $\psi_{n}$ be a stable character of $\U n$. Then
\[
\expectation{\psi_{n}}=O\largeparens{\frac{1}{\dim\psi_{n}}}.
\]
\end{thm}

This proves Conjecture \ref{conj:asymptotic_growth_conjecture_Un}
when $\ssql\smallparens w=1$.

Now assume that $w$ admits a $k$-alternating bislim structure. As
we have shown, this implies stronger bounds in the $w$-cycle theorem.
In Sections \ref{subsec:Character-expectations-on-Sn} and \ref{subsec:Character-expectations-on-Un}
we will show that these stronger bounds are applicable to the methods
of \cite{cassidy,MageeDeLaSalle2024}, obtaining that $\summand_{H}=O\largeparens{\smallparens{\dim\psi_{n}}^{-k}}$
in both cases, and therefore
\[
\expectation{\psi_{n}}=O\largeparens{\smallparens{\dim\psi_{n}}^{-k}},
\]
proving Theorem \ref{thm:alternating_both_cases}. Moreover, by Theorem
\ref{thm:generic_bislim_structures_one_relator_case} generic words
in $F_{r}$ admit $\smallparens{r-1}$-alternating bislim structures,
so for generic words $w\in F_{r}$ we have that
\[
\expectation{\psi_{n}}=O\largeparens{\smallparens{\dim\psi_{n}}^{1-r}}.
\]
This is definitive progress towards Conjecture \ref{conj:asymptotic_growth_conjecture_Un},
although since $\ssql\smallparens w$ is unbounded for a generic word
$w$ (see \cite[Section~7]{Puder2023}), this cannot prove Conjecture
\ref{conj:asymptotic_growth_conjecture_Un} for generic words.
\begin{rem}
The main goal of \cite{cassidy} and \cite{MageeDeLaSalle2024} is
to show strong convergence, and so we consider the effects of this
paper on strong convergence. For one, as explained in \cite[Section~7.2]{MageeDeLaSalle2024},
the stronger bound of $\expectation{\psi_{n}}=O\largeparens{\smallparens{\dim\psi_{n}}^{-1}}$
was already known for \textit{polynomial} irreducible stable characters,
and thus these specific cases obtain a stronger result. In the same
way, Theorem \ref{thm:case_of_Un_non_power} gives the same result
for the non-polynomial case.\footnote{We note that this improvement to the methods of \cite{MageeDeLaSalle2024}
is possible with existing methods (see Remark \ref{rem:existing_methods_remark}),
and a stronger result was also proved by \cite{Chen2024} using different
methods.} Similarly, it may be possible that bounds such as Theorem \ref{thm:alternating_both_cases}
in the case where $\alt_{w}\ge2$ could strengthen the results of
\cite{cassidy} or \cite{MageeDeLaSalle2024} further. However, the
failure probability of Theorem \ref{thm:generic_words_with_explicit_probability}
is too large for this to follow directly from this work.
\end{rem}

\subsection{\label{subsec:Character-expectations-on-Sn}Stable Fourier coefficients
of word measures on $S_{n}$}

\global\long\def\del{\operatorname{del}}%
\global\long\def\partition{\rho}%

\global\long\def\pivec{\left\langle \pi\right\rangle }%
\global\long\def\overpivec{\left\langle \overline{\pi}\right\rangle }%

\global\long\def\gma#1#2{\Gamma\largeparens{#1,#2}}%

We elaborate on the proof of \cite[Theorem 1.6]{cassidy}, which shows
that for every non-power word $w\neq1$ and every stable irreducible
character $\psi_{n}$,
\begin{equation}
\expectation{\psi_{n}}=O\largeparens{\smallparens{\dim\psi_{n}}^{-1}}.\label{eq:cassidy_character_bounds}
\end{equation}
This is proven by decomposing $\expectation{\psi_{n}}$ into parts,
as explained in Section\ref{subsec:Background} (see (\ref{eq:decomposition_into_parts})).
In more detail,
\begin{equation}
\expectation{\psi_{n}}=\sum_{\partition}^{\star}\sum_{\pi_{1},\dots,\pi_{\ell\smallparens w}}^{\le S_{\deg\psi}}V\largeparens{\partition,\pi_{1},\dots,\pi_{\ell\smallparens w}},\label{eq:cassidy_decomposition}
\end{equation}
where the summation variables $\partition$ and $\pivec=\left(\pi_{1},\dots,\pi_{\ell\smallparens w}\right)$
are independent of $n$ and explained properly in Section \ref{subsec:The-graph-Gamma}
below, and $V$ is a complicated expression not detailed here, which
is a rational function in $n$ for large enough $n$. In \cite[Sections~4.4~and~5.4.2]{cassidy}
Cassidy shows how each part corresponds to a labeled graph $\Gamma=\gma{\partition}{\pivec}$
such that
\[
V\largeparens{\partition,\pivec}=O\largeparens{n^{\chi\smallparens{\Gamma}-\sum_{i}\del\pi_{i}}},
\]
where $\chi$ denotes the Euler characteristic. Finally, Cassidy proves
a variant of the $w$-cycle theorem based on the proof of \cite{Louder_Wilton_2017}
using stackings (implicit in \cite[Section~4.5]{cassidy}), concluding
that $\chi\smallparens{\Gamma}-\sum_{i}\del\pi_{i}\le-\deg\psi$ and
thus proving (\ref{eq:cassidy_character_bounds}).

We strengthen Cassidy's use of the $w$-cycle theorem as follows:
\begin{thm}
\label{thm:case_of_Sn}Assume that $w$ admits an alternating bislim
structure with alternation degree $\alt_{w}$, that $\psi_{n}$ is
a stable character of $S_{n}$, and that $\partition,\pivec$ appear
in the sum in (\ref{eq:cassidy_decomposition}). Then 
\[
\chi\largeparens{\gma{\partition}{\pivec}}-\sum_{i}\del\pi_{i}\ \le\ -\deg\psi\cdot\alt_{w}.
\]
\end{thm}

As explained above, this proves Theorem \ref{thm:alternating_both_cases}
in the case of $S_{n}$. When $w$ is not a proper power, it admits
a $1$-alternating bislim structure by \cite[Proposition~2.4]{helferwise},
and we recover the original \cite[Theorem 1.6]{cassidy}. By applying
Theorem \ref{thm:generic_bislim_structures_one_relator_case} it follows
that Conjectures \ref{conj:asymptotic_growth_conjecture_Sn_pi} and
\ref{conj:asymptotic_growth_conjecture_Sn_spi} hold for generic words.

\subsubsection{\label{subsec:The-graph-Gamma}The graph $\Gamma$}

We now elaborate on the definitions of $\partition$, $\pivec$ and
$\Gamma$, and prove Theorem \ref{thm:case_of_Sn} by reducing it
to a version of the $w$-cycles theorem, namely, to Theorem \ref{thm:w_cycles_version_with_S}.

\global\long\def\w{W}%
\global\long\def\ww{\dot{W}}%
\global\long\def\kw{d\ww}%
\global\long\def\dw{\kw}%

\global\long\def\broken#1{\widehat{#1}}%

\global\long\def\graphfinal{H}%

Let $d$ be $\deg\psi$, so $\dim\psi=\Theta\largeparens{n^{d}}$,
and let $\ell$ be the length of $w$. Let $\mathemph{\w}$ be a cycle
spelling $w$, i.e., the graph with $\ell$ vertices $u_{1},\dots,u_{\ell}$,
with an edge which is labeled and directed by the $i$'th letter of
$w$ between $u_{i}$ and $u_{i+1}$ for all $i\in\left[\ell\right]$
(see Figure \ref{fig:W}).\footnote{It is understood that $v_{\ell+1}=v_{1}$.}
Let $\mathemph{\ww}$ be a graph consisting of $\ell$ disjoint edges
and $2\ell$ vertices $v_{1}^{-},\dots,v_{\ell}^{-}$ and $v_{1}^{+},\dots,v_{\ell}^{+}$
with an edge which is labeled and directed by the $i$'th letter of
$w$ between $v_{i}^{-}$ and $v_{i+1}^{+}$ for all $i\in\left[\ell\right]$
(see Figure \ref{fig:W_split}). Thus there is a natural map $\ww\to\w$
mapping $v_{i}^{-}$ and $v_{i}^{+}$ to $u_{i}$, and the edges of
$\ww$ correspond exactly to the edges of $\w$.

Duplicate $\ww$ $d$ times to form $\mathemph{\kw}$ as in Figure
\ref{fig:dW_split}. We name its vertices $v_{i,j}^{\pm}$ where $j\in\left[d\right]$
is the copy of $\ww$ that $v_{i,j}^{\pm}$ is in. Then $\kw$ has
$2d\ell$ vertices and admits a natural projection $p:\kw\to\ww$
given by $p\smallparens{v_{i,j}^{\pm}}=v_{i}^{\pm}$. We characterize
each vertex $v$ of $\kw$ by its \emph{position,} which is its projection
$p\smallparens v$ into $\ww$. 

Now, the summation variable $\mathemph{\pi_{i}}$ in \ref{eq:cassidy_decomposition}
is a bijection between a subset $A$ of $p^{-1}\smallparens{v_{i}^{+}}=\left\{ v_{i,j}^{+}\vert j\in\left[d\right]\right\} $
and a subset $B$ of $p^{-1}\smallparens{v_{i}^{-}}=\left\{ v_{i,j}^{-}\vert j\in\left[d\right]\right\} $.
In (\ref{eq:cassidy_decomposition}), the sum is only over such $\pi_{i}$,
and this is denoted by $\le S_{d}$ above the summation symbol. We
denote by $\mathemph{\del\pi_{i}}$ the amount of vertices which are
not bijected, i.e., $d-\left|A\right|$. In addition to position,
each vertex $v$ of $\kw$ can also be characterized by its \emph{type,
}which is the label of its adjacent edge $e$ and whether $e$ is
directed towards $v$ or away from $v$.

\begin{figure}
\subfloat[\label{fig:W}The graph $\protect\w$ with $w=abab^{-1}$]{\begin{centering}
\includegraphics[width=0.29\textwidth]{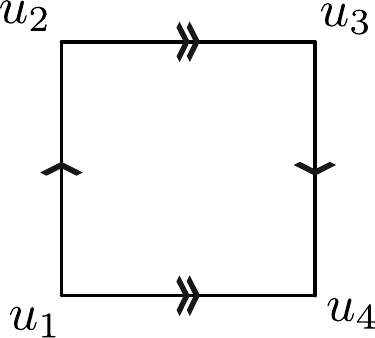}
\par\end{centering}
\centering{}}\hfill{}\subfloat[\label{fig:W_split}The graph $\protect\ww$]{\begin{centering}
\includegraphics[width=0.3\textwidth]{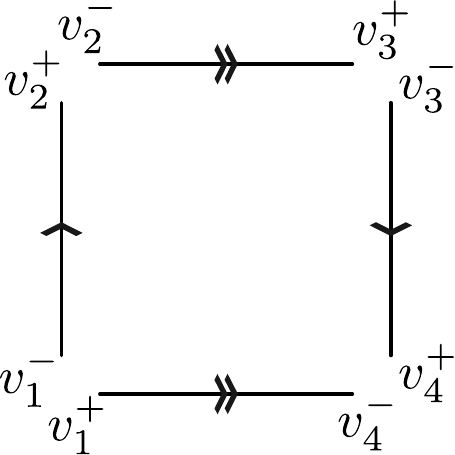}
\par\end{centering}
}\hfill{}\subfloat[\label{fig:dW_split}The graph $\protect\kw$ with $d=2$. The blue
dashed lines depict $\left(\pi_{1},\dots,\pi_{4}\right)$ and are
not part of $\protect\dw$.]{\begin{centering}
\includegraphics[width=0.33\textwidth]{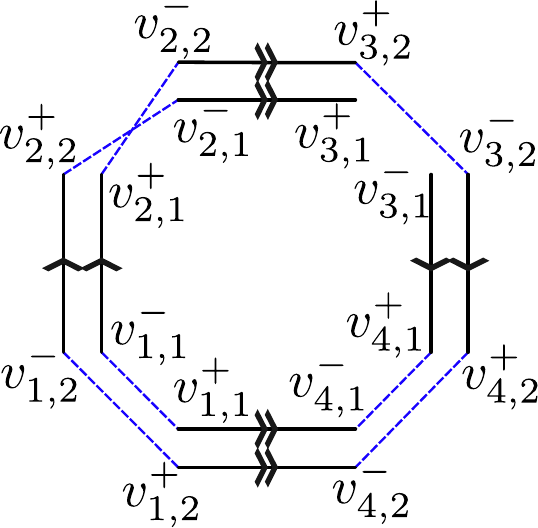}
\par\end{centering}
}

\caption{Examples of $\protect\w$, $\protect\ww$ and $\protect\kw$ where
$w=abab^{-1}$ and $d=2$, as defined in Section \ref{subsec:The-graph-Gamma}.
The label $a$ is drawn as a single arrow while $b$ is drawn as a
double arrow.}
\end{figure}

The summation variable $\partition$ in (\ref{eq:cassidy_decomposition})\footnote{In \cite{cassidy} $\partition$ is split into $2r$ partitions: $\sigma_{f}$
is the subpartition of all of the vertices $v$ adjacent to label
$f$ oriented away from $v$, and $\tau_{f}$ is the subpartition
of all of the vertices $v$ adjacent to label $f$ oriented towards
$v$.} is a partition of the vertices of $\dw$ that satisfies the following:
\begin{defn}[{Following \cite[Page 22]{cassidy}}]
\label{def:star_conditions}We say that a partition $\partition$
of the vertices of $\kw$ is admissible if the following conditions
hold:
\end{defn}

\begin{enumerate}
\item \label{enu:cond_type}Any two vertices in the same block of $\partition$
have the same type.
\item \label{enu:cond_2_cover}There cannot be singleton blocks in $\partition$,
i.e., every block is of size at least 2.
\item \label{enu:cond_efficiency}There cannot be two vertices of the same
position in the same block of $\partition$.
\end{enumerate}
These conditions are taken from the definition of $\overset{\star}{\text{Part}}$
in \cite[Page 22]{cassidy}. In Equation (\ref{eq:cassidy_decomposition})
the sum is only over admissible partitions $\partition$, and this
is denoted by $\star$ above the summation symbol. An example of an
admissible partition $\partition$, where $\kw$ is as in Figure \ref{fig:dW_split},
is given by $\left\{ v_{1,1}^{-},v_{3,2}^{-}\right\} ,\left\{ v_{1,2}^{-},v_{3,1}^{-}\right\} ,\left\{ v_{2,1}^{+},v_{4,1}^{+}\right\} ,\left\{ v_{2,2}^{+},v_{4,2}^{+}\right\} ,\left\{ v_{2,1}^{-},v_{1,1}^{+}\right\} ,\left\{ v_{2,2}^{-},v_{1,2}^{+}\right\} ,\left\{ v_{3,1}^{+},v_{4,1}^{-}\right\} ,\left\{ v_{3,2}^{+},v_{4,2}^{-}\right\} $.

Now let $\partition$ be a partition of the vertices of $\dw$, which
is not necessarily admissible, and let $\pivec=\left(\pi_{1},\dots,\pi_{\ell}\right)$
be partial bijections as above. We construct the graph $\gma{\partition}{\pivec}$
as follows: start from $\kw$, and glue together any two vertices
which are bijected by any $\pi_{i}$, as in Figure \ref{fig:W_cover}.
Then, also glue together any two $\partition$-equivalent vertices.
Finally, glue together any two edges $e_{1}$ and $e_{2}$ of the
same label in $\kw$ where their two initial vertices are $\partition$-equivalent
and their two final vertices are also $\partition$-equivalent. 

See an example of $\gma{\partition}{\pivec}$ in Figure (\ref{fig:W_Gamma})
below. We clarify that there may be parallel edges of the same label
and direction in $\gma{\partition}{\pivec}$ as long as either their
initial or final vertices were glued together indirectly by a combination
of the gluings induced by the partial bijections $\pivec$ and the
partition $\partition$, rather than being directly $\partition$-equivalent,
as in Figure (\ref{fig:W_Gamma}) below.

\subsubsection{\label{subsec:Proof-of-Theorem}Proof of Theorem \ref{thm:case_of_Sn}}

Let $w$ be a word that admits an alternating bislim structure with
alternation degree $\alt_{w}$, let $\psi$ be a stable character
of $S_{n}$ of dimension $\Theta\largeparens{n^{d}}$, let $\pivec=\largeparens{\pi_{1},\dots,\pi_{\ell}}$
be partial bijections as described above, and let $\partition$ be
an admissible partition (Definition \ref{def:star_conditions}). We
show that
\[
\chi\left(\gma{\partition}{\pivec}\right)-\sum_{i}\del\pi_{i}\ \le\ -d\cdot\alt_{w}.
\]

\begin{lem}
We may assume without loss of generality that $\pi_{i}$ are full
bijections.
\end{lem}

\begin{proof}
Denote by $\overline{\pi_{i}}$ a completion of $\pi_{i}$ to a full
bijection, and denote $\overpivec=\left(\overline{\pi_{1}},\dots,\overline{\pi_{\ell\smallparens w}}\right)$.
The graph $\gma{\partition}{\overpivec}$ can be obtained from $\gma{\partition}{\pivec}$
by gluing at most $\sum_{i}\del\pi_{i}$ pairs of vertices. Gluing
one pair of vertices reduces the Euler characteristic by exactly $1$,
so
\[
\chi\largeparens{\gma{\partition}{\pivec}}-\sum_{i}\del\pi_{i}\ \le\ \chi\largeparens{\gma{\partition}{\overpivec}},
\]
and it is enough to prove that $\chi\largeparens{\gma{\partition}{\overpivec}}\le-d\cdot\alt_{w}$.
\end{proof}
Cassidy remarks that it is possible to assume that $\pi_{i}$ are
full bijections in \cite[Remark 4.13]{cassidy}. However, Cassidy
cannot use this assumption since using it forfeits a property that
is assumed in Cassidy's proof: that no two vertices of $\dw$ of the
same position can be glued together in $\gma{\partition}{\pivec}$.
We do not need this property, so we may assume that $\pi_{1},\dots,\pi_{\ell\smallparens w}$
are full bijections.

Let $\emptyset$ be the discrete partition (note that $\emptyset$
is not admissible) and consider the graph $\gma{\emptyset}{\overpivec}$,
obtained from $\kw$ by only doing the $\overline{\pi_{i}}$-gluings.
Then $\gma{\emptyset}{\overpivec}$ is a covering graph of $\w$ of
degree $d$ (see Figure \ref{fig:W_cover}). The graphs fit into the
following commutative diagram:
\[\begin{tikzcd}
	{\Gamma \left(\emptyset, \left< \overline{\pi} \right> \right)} & {\Gamma \left(\rho, \left< \overline{\pi} \right> \right)} \\
	W & {\Omega_r}
	\arrow[from=1-1, to=1-2]
	\arrow["{\text{d-cover}}"', from=1-1, to=2-1]
	\arrow[from=1-2, to=2-2]
	\arrow[from=2-1, to=2-2]
\end{tikzcd}\] Where $\Omega_{r}$ is the bouquet of $r$ circles, with one edge
for each generator in $F_{r}$. Note that if the $\pi_{i}$ are not
full bijections, then $\gma{\emptyset}{\pivec}$ is ``broken'' and
is not a covering graph of $W$.

If this commutative diagram were to satisfy all of the requirements
of Theorem \ref{thm:w_cycles_version_with_S}, we would conclude that
$\chi\left(\gma{\partition}{\overpivec}\right)\le-d\cdot\alt_{w}$,
which is our goal. It does satisfy all of the requirements except
for one: $\gma{\emptyset}{\overpivec}$ does not necessarily cover
every edge of $\gma{\partition}{\overpivec}$ twice. Figure \ref{fig:W_Gamma}
depicts an example where some edges are only covered once by $\gma{\emptyset}{\overpivec}$.

We define a new graph $\mathemph{H=H\largeparens{\partition,\overpivec}}$
as follows: Similarly to the construction of $\gma{\partition}{\overpivec}$,
we start with $\kw$, glue together all $\overline{\pi_{i}}$-equivalent
vertices and all $\partition$-equivalent vertices. However, the last
step is modified: we glue together any two edges $e_{1}$ and $e_{2}$
where their initial vertices are $\partition$-equivalent, even if
their final vertices are not $\partition$-equivalent. Note that there
is a natural map $\varphi:\gma{\partition}{\overpivec}\to\graphfinal$,
since any edges glued in the last step of the construction of $\gma{\partition}{\overpivec}$
are also glued in $\graphfinal$.

\begin{figure}
\subfloat[\label{fig:W_cover}The graph $\protect\gma{\emptyset}{\protect\overpivec}$,
which is a covering graph of $W$, and is given by starting with $\protect\kw$
and gluing any two vertices which are bijected by $\protect\overpivec$
(see Figure \ref{fig:dW_split}).]{\centering{}\includegraphics[width=0.28\textwidth]{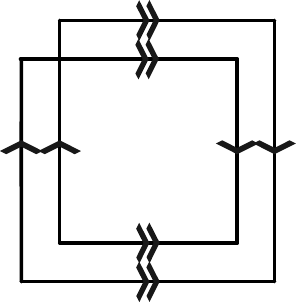}}\hfill{}\subfloat[\label{fig:W_Gamma}The graph $\protect\gma{\protect\partition}{\protect\overpivec}$.
In the natural map $\protect\gma{\emptyset}{\protect\overpivec}\to\protect\gma{\protect\partition}{\protect\overpivec}$,
the edges marked by one arrow are covered once, and the edges labeled
by double arrows are covered twice.]{\centering{}\includegraphics[width=0.33\textwidth]{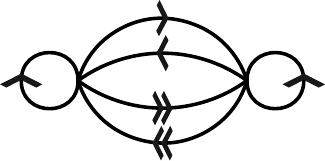}}\hfill{}\subfloat[\label{fig:W_H}The graph $H=H\left(\protect\partition,\protect\overpivec\right)$.
This graph satisfies all of the conditions of Theorem \ref{thm:w_cycles_version_with_S}. ]{\centering{}\includegraphics[width=0.33\textwidth]{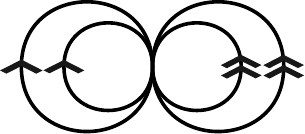}}\caption{The graphs $\protect\gma{\emptyset}{\protect\overpivec}$, $\protect\gma{\protect\partition}{\protect\overpivec}$
and $H$, where $\protect\pivec$ is given as in Figure \ref{fig:dW_split},
and $\protect\partition$ is $\left\{ v_{1,1}^{-},v_{3,2}^{-}\right\} ,\left\{ v_{1,2}^{-},v_{3,1}^{-}\right\} ,\left\{ v_{2,1}^{+},v_{4,1}^{+}\right\} ,\left\{ v_{2,2}^{+},v_{4,2}^{+}\right\} ,\left\{ v_{2,1}^{-},v_{1,1}^{+}\right\} ,\left\{ v_{2,2}^{-},v_{1,2}^{+}\right\} ,\left\{ v_{3,1}^{+},v_{4,1}^{-}\right\} ,\left\{ v_{3,2}^{+},v_{4,2}^{-}\right\} $.}

\end{figure}

We can realize the map $\varphi$ by gluing pairs of edges one at
a time. Every pair of edges we glue has $\partition$-equivalent source
vertices, which are thus already glued in $\gma{\partition}{\overpivec}$.
Gluing two edges with the same source vertex cannot decrease the Euler
characteristic of the graph, and thus $\chi\largeparens{\gma{\partition}{\overpivec}}\ \le\ \chi\smallparens H$.

By Condition \ref{enu:cond_type} of the admissibility conditions,
we only glue together edges of the same label and direction, so there
is a natural map $\graphfinal\to\Omega_{r}$. Thus we have the following
commutative diagram: 
\[\begin{tikzcd}
	{\Gamma \left(\emptyset,\left< \pi \right> \right)} & {H} \\
	{W} & {\Omega_r}
	\arrow[from=1-1, to=1-2]
	\arrow["{\text{d-cover}}"', from=1-1, to=2-1]
	\arrow[from=1-2, to=2-2]
	\arrow[from=2-1, to=2-2]
\end{tikzcd}\]We show that the conditions of Theorem \ref{thm:w_cycles_version_with_S}
apply. Recall the admissibility conditions (Definition \ref{def:star_conditions}).
\begin{itemize}
\item By Condition \ref{enu:cond_2_cover} of the admissibility conditions,
the map $\gma{\emptyset}{\overpivec}\to\graphfinal$ covers each edge
of $H$ at least twice.
\item Assume that two edges $e_{1}$ and $e_{2}$ of $\gma{\emptyset}{\overpivec}$
map to the same edge in $W$ and also map to the same edge in $\graphfinal$.
Then their initial vertices $v_{1}$ and $v_{2}$ in $\kw$ are of
the same position and are also $\partition$-equivalent, in contradiction
with Condition \ref{enu:cond_efficiency} of the admissibility conditions.
\end{itemize}
Thus, all conditions of Theorem \ref{thm:w_cycles_version_with_S}
are met, and so:
\[
\chi\largeparens{\gma{\partition}{\overpivec}}\ \le\ \chi\smallparens{\graphfinal}\ \le\ -d\cdot\alt_{w},
\]
concluding the proof of Theorem \ref{thm:case_of_Sn}.

\subsection{\label{subsec:Character-expectations-on-Un}Stable Fourier coefficients
of word measures on $\protect\U n$}

In \cite[Corollary~11.3]{MageeDeLaSalle2024} it is shown that for
every non-power word $w\neq1$, and every stable irreducible character
$\psi_{n}$,
\begin{equation}
\expectation{\psi_{n}}=O\largeparens{\smallparens{\dim\psi_{n}}^{-\frac{1}{6}}}.\label{eq:magee_de_la_salle_theorem_3.1(2)}
\end{equation}
This is proven by decomposing $\expectation{\psi_{n}}$ into a sum
of parts, as explained above in (\ref{eq:decomposition_into_parts}).
In \cite[Pages~195-197]{MageeDeLaSalle2024} it is shown how each
part corresponds to a compact surface $\Sigma$, such that the part
is bounded by $O\largeparens{n^{\chi\smallparens{\Sigma}}}$ where
$\chi$ denotes the Euler characteristic. The surface $\Sigma$ has
additional structure, given by \cite[Definition 9.2]{MageeDeLaSalle2024}
and some technical conditions relating to $\psi_{n}$. All of this
culminates in \cite[Proposition~11.1]{MageeDeLaSalle2024}, which
shows that
\begin{equation}
\expectation{\psi_{n}}=O\largeparens{n^{\max\chi\smallparens{\Sigma}}},\label{eq:magee_de_la_salle_decomposition}
\end{equation}
where the maximum is taken over all such surfaces $\Sigma$. Finally,
\cite[Proposition 11.2]{MageeDeLaSalle2024} shows that each surface
$\Sigma$ satisfies $\chi\smallparens{\Sigma}\ \le\ -\frac{\deg\psi}{6}$,
proving (\ref{eq:magee_de_la_salle_theorem_3.1(2)}). This is essentially
a crude version of the $w$-cycle theorem, and here we explain how
to derive the following stronger bound.
\begin{thm}
\label{thm:case_of_Un}Assume that $w$ admits an alternating bislim
structure with alternation degree $\alt_{w}$, that $\psi_{n}$ is
an irreducible stable character of $\U n$, and that $\Sigma$ is
a surface satisfying the conditions of \cite[Proposition 11.1]{MageeDeLaSalle2024}.
Then $\chi\smallparens{\Sigma}\ \le\ -\deg\psi\cdot\alt_{w}$.
\end{thm}

Together with (\ref{eq:magee_de_la_salle_decomposition}) this proves
Theorem \ref{thm:alternating_both_cases} for the case of $\U n$.
The case of $\alt_{w}=1$ proves Theorem \ref{thm:case_of_Un_non_power}.
\begin{proof}
Unfortunately, the surfaces $\Sigma$ as defined in \cite[Definition 9.2]{MageeDeLaSalle2024}
do not have the 2-complex structure which our methods use. However,
as explained in \cite[Section~11.1]{MageeDeLaSalle2024}, a surface
$\Sigma$ gives a ribbon graph\footnote{A ribbon graph is a graph where each vertex has an associated cyclic
ordering on its neighboring half edges. However, we do not use the
fact that $R$ is a ribbon graph.} $R$, which is directed and labeled over the same alphabet as $w$.
Additionally, the conditions of \cite[Proposition 11.1]{MageeDeLaSalle2024}
translate to $R$ as follows:
\begin{itemize}
\item The boundary components of $\Sigma$ give a family of $w$-cycles
of $R$, which cover each edge of $R$ exactly twice.
\item The number of $w$-cycles in this family, counted by their degrees,
is $\deg\psi$.
\item No two edges of these $w$-cycles which are in the same position within
$w$ traverse the same edge of $R$ (this corresponds to the ``Forbidden
matchings'' condition of \cite[Proposition 11.1]{MageeDeLaSalle2024}).
\end{itemize}
These conditions are precisely the conditions that are required to
apply Theorem \ref{thm:w_cycles_version_with_S}, proving that $\chi\smallparens R\ \le\ -\deg\psi\cdot\alt_{w}$.
However, $\Sigma$ deformation retracts onto $R$ by construction,
so $\chi\smallparens{\Sigma}=\chi\smallparens R$, proving Theorem
\ref{thm:case_of_Un}.
\end{proof}
\begin{rem}
\label{rem:existing_methods_remark}We note that the case of $\alt_{w}=1$
in Theorem \ref{thm:case_of_Un} and Theorem \ref{thm:case_of_Un_non_power}
are provable by existing methods, even though it might not seem so.
Note that the graph $R$ might not be a core graph, and even though
no two edges of the $w$-cycles of $R$ which are in the same position
within $w$ traverse the same edge of $R$, it may be that two vertices
of $w$-cycles may be in the same position within $w$ and also map
to the same vertex of $R$. These two missing properties prevent applying
\cite[Theorem~2]{Louder_Wilton_2017} and \cite[Theorem~4.1]{helferwise}.
However, a close reading of both proofs will show that these two missing
properties are not required in either proof with small adaptations.
In \cite[Section~4.5]{cassidy} a version of the $w$-cycle theorem
that also does not require $R$ to be a core graph is developed, although
it is not written as a standalone theorem.
\end{rem}

\subsection{\label{subsec:extension_to_multiple_words}Extension to multiple
words}

The algebraic methods of \cite{cassidy} and \cite{MageeDeLaSalle2024}
can be extended to the case of multiple words. Let $G_{n}$ be either
$S_{n}$ or $\U n$. For a sequence $w_{1},\dots,w_{k}$ of words
over an alphabet $x_{1},\dots,x_{r}$ and a sequence $\psi^{1},\dots,\psi^{k}$
of stable irreducible characters over $G_{n}$, consider the character
expectation
\[
\expectation[x_{1},\dots,x_{r}\in G_{n}]{\psi_{n}^{1}\smallparens{w_{1}}\cdot\dots\cdot\psi_{n}^{k}\smallparens{w_{k}}}.
\]

Using the methods of this paper, we can show:
\begin{thm}
Let $X$ be the presentation complex of $\left\langle x_{1},\dots,x_{r}\vert w_{1},\dots,w_{k}\right\rangle $,
and let $G_{n}$ be either $S_{n}$ or $\U n$. If $X$ admits an
alternating bislim structure, then for every series of stable irreducible
characters $\psi^{1},\dots,\psi^{k}$ of $G_{n}$,
\[
\expectation[x_{1},\dots,x_{r}\in G_{n}]{\psi_{n}^{1}\smallparens{w_{1}}\cdot\dots\cdot\psi_{n}^{k}\smallparens{w_{k}}}=O\largeparens{\smallparens{\dim\psi^{1}}^{\alt_{w_{1}}}\cdot\dots\cdot\smallparens{\dim\psi^{k}}^{\alt_{w_{k}}}}.
\]
\end{thm}

However, the algebraic details required to reproduce the relevant
parts of \cite{cassidy} and \cite{MageeDeLaSalle2024} are too involved
to detail here, so we leave this without a formal proof.

\section{\label{sec:nonpositive_and_negative_immersions}Nonpositive and negative
immersions}

The original motivation of the $w$-cycle theorem was to prove the
nonpositive immersions, negative immersions and coherence properties
for various group presentations. We show that generic $r$-generator
$k$-relator presentations have negative immersions when $r-k\ge2$
and nonpositive immersions when $r-k\ge1$. We additionally derive
a new proof of the fact that when $r-k\ge1$ generic presentations
are coherent, and prove some results on the maximal irreducible curvature
(see \cite[Definition~6.3]{wilton_curvature_invariants}) of a generic
presentation complex.

Before we begin, we point the curious reader to \cite{wilton_curvature_invariants}
for an elegant survey of nonpositive immersions, negative immersions,
and many interesting conjectures.

\subsection{Nonpositive immersions}
\begin{defn}[{\cite[Definition~1.2]{Wise2020}$\,$}]
A complex $X$ has (contracting) \emph{nonpositive immersions} if
for every immersion $Y\looparrowright X$ with a finite connected
complex $Y$, either $\chi\smallparens Y\le0$ or $Y$ is contractible.
\end{defn}

This concept is intimately related to bislim structures.
\begin{thm}[{\cite[Corollary 4.4]{helferwise}$\,$}]
If $X$ is bislim then it has nonpositive immersions.
\end{thm}

For completeness, we reprove their connection:
\begin{proof}
If all edges of $Y$ are internal then Theorem \ref{thm:w_cycles_complex_version_two_complexes}
shows that $\chi\smallparens Y\le0$ and we are done. If $Y$ has
a non-internal edge then removing its neighboring 2-cell is a homotopy
equivalence. Keep removing 2-cells until only internal edges remain,
in which case we are done, or $Y$ becomes a graph, in which case
it is contractible.
\end{proof}
By Theorem \ref{thm:generic_bislim_structures_generic_case} a generic
$r$-generator $\smallparens{r-1}$-relator presentation admits a
good bislim structure, and therefore we obtain
\begin{thm}
\label{thm:generic-deficiency-1-NPI}A generic $r$-generator $\smallparens{r-1}$-relator
presentation has nonpositive immersions.
\end{thm}

Previously, a somewhat weaker variant of this statement was known
to hold with positive probability. By \cite[Theorem A]{Kielak2022},
with positive asymptotic probability, a generic $r$-generator $\smallparens{r-1}$-relator
presentation complex $X$ has a free-by-cyclic fundamental group.
By \cite[Theorem~6.1]{Wise2020}, any free-by-cyclic group $G$ has
a complex $X'$ with nonpositive immersions where $\pi_{1}\smallparens{X'}=G$.
However, $X'$ is not guaranteed to be the same as $X$. It is true
that with high probability both $X$ and $X'$ are $K\smallparens{G,1}$
spaces, so $X$ and $X'$ are homotopic, however, this is not enough
to preserve nonpositive immersions.

\subsection{Negative immersions}
\begin{defn}[{\cite[Definition~1.2 ]{Wise2020}$\,$}]
A complex $X$ has (contracting) \emph{negative immersions} if there
exists a positive constant $c$ such that for every immersion $Y\looparrowright X$
with finite connected $Y$ without non-internal edges, either $\chi\smallparens Y\le-c\cdot\#\text{2-cells}\smallparens Y$
or $Y$ is a single vertex.
\end{defn}

In fact there are multiple alternative definitions of negative immersions
in the literature (e.g., \cite[Definition~14.1]{Wise2004}, \cite[Section~3.3]{Louder2021}
and \cite[Definition~6.3]{wilton_curvature_invariants}), though the
differences are relatively minor and our results apply to all variants
of negative immersions with small adjustments.

Let $X$ be a generic $r$-generator $k$-relator presentation complex.
If $r>2k$ then $X$ admits a $\smallparens{2,\dots,2}$-alternating
bislim structure by Theorem \ref{thm:generic_bislim_structures_generic_case}.
By Theorem \ref{thm:w_cycles_complex_version_two_complexes}, if $Y$
is a finite complex without non-internal edges which immerses into
$X$, then $\chi\smallparens Y\le-\text{\#2-cells}\smallparens Y$,
and thus $X$ has negative immersions.

This method can only work when $r>2k$, since otherwise $X$ cannot
admit a $\smallparens{2,\dots,2}$-alternating bislim structure. In
this section we strengthen this idea into the following:
\begin{thm}
\label{thm:generic_deficiency_2}A generic $r$-generator $k$-relator
presentation with $r\ge k+2$ has negative immersions.
\end{thm}

We note that by \cite[Theorem~1.4]{Wise2020}, negative immersions
implies coherence, and therefore we get as a corollary a new proof
that generic $r$-generator $\smallparens{r-2}$-relator presentations
are coherent. This was first proven by \cite[Theorem B]{Kielak2022}.
For more about coherence, see Section \ref{subsec:Coherence}. We
now introduce a way to mix together multiple bislim structures in
order to prove Theorem \ref{thm:generic_deficiency_2}:
\begin{defn}
\label{def:euler_characteristic-polyhedron}Let $X$ be a 2-complex
with $k$ 2-cells $C_{1},\dots,C_{k}$. We say that a branched near-immersion
$Y\to X$ is \emph{admissible} if $Y$ is finite, all the edges in
$Y$ are internal, and $Y$ is not a single point. We denote by $\#_{C_{i}}Y$
the number of 2-cells of $Y$ which map to $C_{i}$. The\emph{ alternation
degree set} is the set $\mathemph{\mathcal{A}_{X}}\subseteq\mathbb{R}^{k}$
defined by
\[
\mathcal{A}_{X}=\left\{ \smallparens{\alpha_{1},\dots,\alpha_{k}}\in\mathbb{R}^{k}\,\middle\vert\,\substack{\chi\largeparens{\skeleton Y}\,\le\,-\sum_{i}\alpha_{i}\cdot\#_{C_{i}}Y\\
\text{for all admissible }Y\to X
}
\right\} .
\]
\end{defn}

It follows by definition that $\mathcal{A}_{X}$ is convex. If $\mathcal{A}_{X}$
contains the vector $\smallparens{1,\dots,1}$ then $X$ has nonpositive
immersions. Similarly, if $\mathcal{A}_{X}$ contains a vector where
all entries are greater than $1$ then $X$ has negative immersions.
\begin{lem}
\label{lem:euler_characteristic_polyhedra_generic_case}For integers
$r$ and $k$, a generic $r$-generator $k$-relator presentation
complex $X$ satisfies 
\[
\largeparens{\frac{r-1}{k},\dots,\frac{r-1}{k}}\in\mathcal{A}_{X}.
\]
\end{lem}

\begin{proof}
Theorems \ref{thm:generic_bislim_structures_generic_case} and \ref{thm:w_cycles_complex_version_two_complexes}
show that all nonnegative integer vectors $\smallparens{\alpha_{1},\dots,\alpha_{k}}$
where $\sum_{i}\alpha_{i}<r$ are contained in $\mathcal{A}_{X}$
for generic $X$. Thus for a generic $r$-generator $k$-relator 2-complex
$X$ we know that all vectors of the form $\smallparens{0,\dots,0,r-1,0,\dots,0}$
are contained in $\mathcal{A}_{X}$. Their average is $\largeparens{\frac{r-1}{k},\dots,\frac{r-1}{k}}$,
so it is contained in $\mathcal{A}_{X}$.
\end{proof}
Theorem \ref{thm:generic_deficiency_2} follows: indeed, if $r-k\ge2$,
we have that $\frac{r-1}{k}>1$, so a generic $r$-generator $k$-relator
presentation complex $X$ has negative immersions.

\subsection{\label{subsec:Hyperbolicity-and-Negative-immersions}Hyperbolicity
and Negative immersions}

In this section we show how our techniques can be used to prove that
2-complexes have hyperbolic fundamental groups. This is related to
a conjecture of Wise that if a finite 2-complex $X$ has negative
immersions then $\pi_{1}\smallparens X$ is hyperbolic \cite[Conjecture~14.2]{Wise2004}.
Louder and Wilton conjectured another slightly more general version
of the same conjecture \cite[Conjecture~1.9]{Louder2018}. Linton
proved this conjecture for one-relator groups \cite[Theorem~7.2]{Linton2022}.

We show that whenever our techniques prove that a 2-complex $X$ has
negative immersions, we can also prove that $\pi_{1}\smallparens X$
is hyperbolic, corroborating the conjecture. Additionally, this leads
to algorithms which can certify hyperbolicity and negative immersions,
which we detail in Section \ref{subsec:Checking-for-hyperbolicity}.

We expand the definition of the Euler characteristic polyhedron by
allowing complexes $Y$ with non-internal edges:

\global\long\def\aa{\overline{\mathcal{A}}}%

\begin{defn}
\label{def:expanded_euler_characteristic_polyhedron}Let $X$ be a
complex with $k$ 2-cell $C_{1},\dots,C_{k}$. We say that a branched
near-immersion $Y\to X$ is \emph{boundary-admissible} if $Y$ is
finite and $Y$ is not a single point. The\emph{ expanded alternation
degree set} is the set $\mathemph{\aa_{X}}\subseteq\mathbb{R}^{k+1}$
defined by
\[
\mathcal{\aa}_{X}=\left\{ \smallparens{\beta;\alpha_{1},\dots,\alpha_{k}}\in\mathbb{R}^{k+1}\,\middle\vert\,\substack{\chi\largeparens{\skeleton Y}\,\le\,-\sum_{i}\alpha_{i}\cdot\#_{C_{i}}Y+\beta\left|\boundary Y\right|,\,\beta>0\\
\text{for all boundary-admissible }Y\to X
}
\right\} ,
\]
where $\boundary Y$ is the set of non-internal edges of $Y$.
\end{defn}

As before, $\aa_{X}$ is a convex set. By Theorem \ref{thm:cycle_counting_complex_version_with_boundary}
it follows that:
\begin{prop}
\label{prop:bislim_structure_to_expanded_euler_characteristic_polyhedron}If
a 2-complex $X$ admits an $\smallparens{\alpha_{1},\dots,\alpha_{k}}$-alternating
bislim structure, then $\smallparens{2;\alpha_{1},\dots,\alpha_{k}}\in\aa_{X}$.
\end{prop}

\begin{lem}
\label{lem:proving_hyperbolicity}Let $X$ be a 2-complex. If $\smallparens{\beta;\alpha_{1},\dots,\alpha_{k}}\in\aa_{X}$
where all $\alpha_{i}>1$, then $\pi_{1}\smallparens X$ is hyperbolic.
\end{lem}

\begin{proof}
Let $Y$ be a reduced disk diagram in $X$. Since it is topologically
a disk, we have $\chi\smallparens Y=1$. Since it is reduced, the
map $Y\to X$ is a near-immersion, so $Y\to X$ is boundary-admissible.
Denote $\min_{i}\alpha_{i}$ by $\alpha$, and denote by $\text{\#Area}\smallparens Y$
the number of 2-cells in $Y$. 
\begin{eqnarray*}
1 & = & \chi\smallparens Y\\
 & = & \text{\#Area}\smallparens Y+\chi\largeparens{\skeleton Y}\\
 & \le & \text{\#Area}\smallparens Y-\alpha\sum_{i}\#_{C_{i}}Y+\beta\left|\boundary Y\right|\\
 & = & -\smallparens{\alpha-1}\text{\#Area}\smallparens Y+\beta\left|\boundary Y\right|
\end{eqnarray*}
and we get that
\[
\smallparens{\alpha-1}\text{\#Area}\smallparens Y\ <\ \beta\left|\boundary Y\right|.
\]
Since $\beta>0$ and $\alpha>1$, this is a linear isoperimetric inequality,
that applies to all reduced disk diagrams in $X$. Thus $X$ is hyperbolic.
\end{proof}
Therefore, whenever bislim structures show that a vector $\smallparens{\alpha_{1},\dots,\alpha_{k}}$
belongs to $\mathcal{A}_{X}$, then by using Theorem \ref{thm:cycle_counting_complex_version_with_boundary},
it can also be shown that $\smallparens{2;\alpha_{1},\dots,\alpha_{k}}\in\aa_{X}$.
Thus whenever bislim structures show that $X$ has negative immersions,
they also show that $X$ is hyperbolic.

The following lemma expands the range of complexes which can be shown
to have negative immersions and hyperbolicity.
\begin{lem}
\label{lem:expanded_euler_characteristic_polyhedron_adding_cells}Let
$X$ be a 2-complex with $k$ 2-cells $C_{1},\dots,C_{k}$. Let $X'$
be the complex obtained by removing $C_{k}$ from $X$, so $X'$ has
$k-1$ 2-cells. If $\smallparens{\beta;\alpha_{1},\dots,\alpha_{k-1}}\in\aa_{X'}$,
then 
\[
\largeparens{\beta;\alpha_{1},\dots,\alpha_{k-1},-\beta\cdot\left|\boundary C_{k}\right|}\in\aa_{X}.
\]
\end{lem}

\begin{proof}
By definition of $\aa_{X}$ we need to show that for all boundary-admissible
$Y\to X$, 
\[
\chi\largeparens{\skeleton Y}\ \le\ -\sum_{i=1}^{k-1}\alpha_{i}\#_{C_{i}}Y+\beta\left|\boundary C_{k}\right|\#_{C_{k}}Y+\beta\left|\boundary Y\right|.
\]
 Construct a new complex $Y'$ with a new boundary-admissible map
$Y'\to X'$ by removing from $Y$ all 2-cells mapping to $C_{k}$.
Since $\smallparens{\beta;\alpha_{1},\dots,\alpha_{k-1}}\in\aa_{X'}$
we know that 
\begin{equation}
\chi\largeparens{\skeleton{Y'}}\ \le\ -\sum_{i=1}^{k-1}\alpha_{i}\#_{C_{i}}Y+\beta\left|\boundary Y'\right|.\label{eq:hyperbolicity_something}
\end{equation}
 Removing a 2-cell which maps to $C_{k}$ from $Y$ can create at
most $\left|\boundary C_{k}\right|$ new non-internal edges, when
counted correctly with multiplicity, so $\left|\boundary Y'\right|\le\left|\boundary Y\right|+\left|\boundary C_{k}\right|\cdot\#_{C_{k}}Y$.
Together with (\ref{eq:hyperbolicity_something}) we get the desired
result.
\end{proof}
\begin{example}
Consider a presentation complex $X=\left\langle F\vert w_{1},w_{2}\right\rangle $
where $X'=\left\langle F\vert w_{1}\right\rangle $ admits a 2-alternating
bislim structure, and $X=\left\langle F\vert w_{1},w_{2}\right\rangle $
admits a $\smallparens{1,2}$-alternating bislim structure, but not
a $\smallparens{2,1}$-alternating bislim structure. Thus $\smallparens{2;2}\in\aa_{X'}$,
so by the above lemma we have $v_{1}=\largeparens{2;2,-2\left|w_{2}\right|}\in\aa_{X}$.
We also know that $v_{2}=\smallparens{2;1,2}\in\aa_{X}$. Thus for
some small enough $\varepsilon>0$ the vector $\varepsilon v_{1}+\smallparens{1-\varepsilon}v_{2}\in\aa_{X}$
has all entries bigger than $1$, so $X$ has negative immersions
and $\fundamental X$ is hyperbolic.
\end{example}

We will now write $1^{+}$ for entries that are larger than $1$.
In general, we obtain this proposition:
\begin{lem}
\label{lem:relative_negative_imm_prop}Let $X$ be a 2-complex with
$k$ 2-cells and let $X'$ be a subcomplex of $X$ obtained by removing
the $k$'th 2-cell from $X$. If $X$ has a vector of the form $\smallparens{\beta;1,\dots,1,1^{+}}\in\aa_{X}$,
and $X'$ has a vector $\smallparens{\beta';1^{+},\dots,1^{+}}\in\aa_{X'}$,
then $X$ also has a vector of the form $\smallparens{\beta'';1^{+},\dots,1^{+}}\in\aa_{X}$.
\end{lem}

This suggests that there should be a notion of relative negative immersions.
Additionally, this also suggests that if $X$ has negative immersions
relative to $X'$, then $X$ should also be hyperbolic relative to
$X'$.

\subsection{\label{subsec:Wilton's-curvature-invariant}Wilton's curvature invariant}

In \cite[Definition~6.3]{wilton_curvature_invariants}, Wilton introduces
the \emph{maximal irreducible curvature} $\rho_{+}\smallparens X$
of a combinatorial 2-complex $X$, as the supremum of $1+\frac{\chi\largeparens{\skeleton Y}}{\#\text{2-cells}\smallparens Y}$
over certain well-behaved combinatorial branched maps $Y\to X$, which
in particular, are branched near-immersions where all the edges of
$Y$ are internal. In this context, $\#\text{2-cells}\smallparens Y$
counts 2-cells with multiplicity according to their branching degree
when mapped to $X$.

The maximal irreducible curvature $\rho_{+}\smallparens X$ offers
a quantitative variant of nonpositive and negative immersions: A version
of nonpositive immersions can be defined by $\rho_{+}\smallparens X\le0$,
and a version of negative immersions can be defined by $\rho_{+}\smallparens X<0$.
We obtain the following:
\begin{lem}
\label{lem:polyhedron_to_maximal_irreducible_curvature}Let $X$ be
a 2-complex.
\end{lem}

\begin{itemize}
\item If $X$ admits an $\smallparens{\alpha_{1},\dots,\alpha_{r}}$-alternating
bislim structure, and $\alpha_{i}\ge c$ for all $i$, then $\rho_{+}\smallparens X\le1-c$.
\item If $\smallparens{c,\dots,c}\in\mathcal{A}_{X}$ then $\rho_{+}\smallparens X\le1-c$.
\end{itemize}
\begin{proof}
The first item follows from Theorem \ref{thm:w_cycles_complex_version_two_complexes},
and the second item follows directly from the definition of $\rho_{+}\smallparens X$.
\end{proof}
This proves a conjecture of Wilton (\cite[Conjecture~11.13]{wilton_curvature_invariants}):
\begin{thm}
\label{thm:generic_bound_on_maximal_irreducible_curvature}For integers
$r$ and $k$, a generic $r$-generator $k$-relator presentation
complex $X$ satisfies $\rho_{+}\smallparens X=1-\frac{r-1}{k}$.
\end{thm}

\begin{proof}
The bound $\rho_{+}\smallparens X\le1-\frac{r-1}{k}$ follows directly
from Lemmas \ref{lem:euler_characteristic_polyhedra_generic_case}
and \ref{lem:polyhedron_to_maximal_irreducible_curvature}. Additionally,
as noted in the discussion of \cite[Conjecture~11.13]{wilton_curvature_invariants},
the identity map $X\to X$ is considered in the supremum in the definition
of $\rho_{+}$ with probability 1, and therefore the bound $\rho_{+}\smallparens X\le1+\frac{\chi\left(\skeleton X\right)}{\#\text{2-cells}\smallparens X}=1-\frac{r-1}{k}$
holds for a generic $r$-generator $k$-relator presentation complex
$X$.
\end{proof}
The same bound as Lemma \ref{lem:polyhedron_to_maximal_irreducible_curvature}
also holds for a weighted version: \cite[Definition~6.3]{wilton_curvature_invariants}
also defines $\rho_{+}\smallparens X$ for \textit{weighed} 2-complexes.
\begin{lem}
Let $X$ be a 2-complex where the $i$'th 2-cell is weighed by $w_{i}$.
\begin{itemize}
\item If $X$ admits an $\smallparens{\alpha_{1},\dots,\alpha_{r}}$-alternating
bislim structure, and $\alpha_{i}\ge w_{i}c$ for all $i$, then $\rho_{+}\smallparens X\le1-c$.
\item If $\smallparens{cw_{1},\dots,cw_{r}}\in\mathcal{A}_{X}$ then $\rho_{+}\smallparens X\le1-c$.
\end{itemize}
\end{lem}

\subsection{\label{subsec:Coherence}Coherence}
\begin{defn}
A group $G$ is \emph{coherent} if every finitely generated subgroup
$H\le G$ is finitely presented.
\end{defn}

\begin{thm}
\label{cor:generic-deficency-1-coherent}A generic $r$-generator
$\smallparens{r-1}$-relator presentation is coherent.
\end{thm}

This was first shown to hold with positive asymptotic probability
by \cite[Theorem A]{Kielak2022}, and was later fully proven by the
recent \cite[Corollary~C]{Fisher2024}. Both show the stronger property
of being virtually free-by-cyclic, which implies coherence by \cite{Feighn1999}.
However, we present a simpler proof. 
\begin{proof}
Let $X$ be a generic $\smallparens{r+1}$-generator $r$-relator
presentation. By Theorem \ref{thm:generic-deficiency-1-NPI} $X$
has nonpositive immersions. By \cite[Theorem~1.2]{JaikinZapirain2023}
this implies that $\fundamental X$ is homologically coherent.

We also know that $X$ satisfies the small cancellation condition
$C\smallparens 7$, so $X$ is also aspherical, and thus $X$ is a
$K\smallparens{G,1}$ space and $\pi_{1}\smallparens X$ is of cohomological
dimension $2$. Together with hyperbolicity and homological coherence,
by \cite{Gersten1996} it follows that $\pi_{1}\smallparens X$ is
coherent.
\end{proof}

\section{\label{subsec:Algorithms}Algorithms}

\subsection{Finding bislim structures}

By \cite{partial_stackings_and_bislim_structures}, $\smallparens{a_{1},\dots,a_{k}}$-alternating
bislim structures are equivalent to $\smallparens{a_{1},\dots,a_{k}}$-alternating
partial stackings. The definition of alternating partial stackings
is combinatorial and finite, so it is possible to find whether a given
2-complex $X$ admits a $\smallparens{a_{1},\dots,a_{k}}$-alternating
bislim structure in finite time. In \cite{partial_stackings_and_bislim_structures}
we present the following algorithm, which does this in polynomial
time:
\begin{lyxalgorithm}[Algorithm from \cite{partial_stackings_and_bislim_structures}]
\label{alg:check_bislimness}It can be determined whether a given
heightened complex $\heightened$ is bislim in time $O\largeparens{n^{3}}$,
where $n=\sum_{C\in X}\left|\boundary C\right|$ and the sum ranges
over all 2-cells $C$ of $X$.
\end{lyxalgorithm}

Therefore, we can search for bislim structures like so:
\begin{lyxalgorithm}
\label{alg:find_best_bislim_structure}Given a 2-complex $X$ with
$k$ 2-cells and an alternation vector $\smallparens{a_{1},\dots,a_{k}}\in\mathbb{Z}_{\ge0}^{k}$,
we can find whether $X$ admits an $\smallparens{a_{1},\dots,a_{k}}$-alternating
bislim structure in time $O\largeparens{n^{2\sum_{i}a_{i}+3}}$, where
$n=\sum_{C\in X}\left|\boundary C\right|$ and the sum ranges over
all 2-cells $C$ of $X$.
\end{lyxalgorithm}

\begin{proof}
There are at most $n^{2\sum_{i}a_{i}}$ heightened complexes $\heightened$
which are $\smallparens{a_{1},\dots,a_{k}}$-alternating. Check each
of them using algorithm \ref{alg:check_bislimness}.
\end{proof}
In order to expand the usefulness of alternating bislim structures,
we note the following observations. Let $a$ and $b$ be two independent
uniformly random permutations. Then $\left(ab,b^{-1}\right)$ are
also two independent uniformly random permutations. Therefore, we
always have that $\expectation[a^{2}b^{2}]{\psi}=\expectation[\smallparens{ab}^{2}b^{-2}]{\psi}=\expectation[abab^{-1}]{\psi}$
for all characters $\psi$. In general, we say that two words $w_{1}$
and $w_{2}$ in a free group $F$ are \emph{automorphic} if there
is an automorphism $\eta\in\aut F$ such that $\eta\smallparens{w_{1}}=w_{2}$
(e.g., the automorphism defined by $\eta\smallparens a=ab$ and $\eta\smallparens b=b^{-1}$).
Automorphic words always induce the same word measures, and thus Conjectures
\ref{conj:asymptotic_growth_conjecture_Un}, \ref{conj:asymptotic_growth_conjecture_Sn_pi},
\ref{conj:spi_equals_pi-1} and \ref{conj:asymptotic_growth_conjecture_Sn_spi}
hold for $w_{1}$ if and only if they hold for an automorphic word
$w_{2}$. Additionally, $\pi\smallparens{w_{1}}=\pi\smallparens{w_{2}}$
and $\spi\smallparens{w_{1}}=\spi\smallparens{w_{2}}$ are also invariant
(see \cite[Claim~4.3]{Puder2023}).

However, alternating stackings and bislim structures are not invariant
under automorphisms. As an example, $abcab^{-1}c^{-1}$ does not admit
a $2$-alternating partial stacking, although the automorphic word
$a^{2}b^{2}c^{2}$ does. Therefore, one may prove that Conjectures
\ref{conj:asymptotic_growth_conjecture_Un}, \ref{conj:asymptotic_growth_conjecture_Sn_pi},
\ref{conj:spi_equals_pi-1} or \ref{conj:asymptotic_growth_conjecture_Sn_spi}
hold for a word $w$ by finding an automorphic word $w'$ which admits
an appropriate alternating stacking, even if $w$ does not admits
one. This is in fact done in Section \ref{subsec:Algorithms} as part
of the proof of Theorem \ref{thm:conjecture_holds_for_small_words}.

\subsection{\label{subsec:Checking-for-hyperbolicity}Checking for hyperbolicity
and negative immersions}

The computational problem of checking hyperbolicity of a given group
presentation is undecidable by the Adian-Rabin theorem (\cite{Rabin1958})
and is in $\text{RE}$ (\cite{Papasoglu_1995}), i.e., there is an
algorithm which eventually accepts a given group presentation if and
only if the group presentation is hyperbolic, but otherwise might
never terminate. One such algorithm is implemented in the software
package \cite{kbmag}.

We say that an algorithm $A$ \emph{certifies} that a property $P$
holds if the property $P$ holds for every possible input $X$ which
is accepted by $A$. We say that $A$ is \emph{complete} if it is
guaranteed to accept every input $X$ for which the property $P$
holds.

Since the running time of complete algorithms for certifying hyperbolicity
is inherently unbounded and slow in practice, there are a number of
incomplete certifying algorithms and criteria for hyperbolicity. Although
such algorithms may not work on many hyperbolic inputs, they can often
be used to certify that groups are hyperbolic faster than complete
certifying algorithms. They range from fast algorithms which cover
relatively few cases (such as small cancellation conditions) to slower
algorithms which cover a larger set of cases (e.g., \cite{Holt2019}).

As for negative immersions, \cite{Wilton_rationality_theorem} constructs
an algorithm which fully determines whether a given 2-complex has
negative immersions.\footnote{There are multiple alternative definitions of negative immersions;
This applies to a specific definition.} However, this algorithm requires doubly-exponential running time
and it has not been fully run on nontrivial examples to the best of
our knowledge.

In this section we provide algorithms which certify that a given 2-complex
$X$ has negative immersions and has a hyperbolic fundamental group.
We clump these two properties together since when our results can
prove one property, they can prove both.

Recall the results of Section \ref{subsec:Hyperbolicity-and-Negative-immersions},
which can be used to show that a 2-complex $X$ has negative immersions
and hyperbolicity given that $X$ has sufficiently alternating bislim
structures. We obtain the following two algorithms, which can certify
that a given 2-complex $X$ has negative immersions and hyperbolicity.
The first version is a more lightweight one:
\begin{lyxalgorithm}
Given a 2-complex $X$, check if $X$ admits a $\smallparens{1,\dots,1,2,1,\dots,1}$-alternating
bislim structure, by going over all such heightened complexes and
checking if any one is bislim. If there is a $\smallparens{1,\dots,1,2,1,\dots,1}$-alternating
bislim structure, pick one. Remove the 2-cell which is 2-alternating
from $X$ to obtain a new 2-complex $X'$, and then recursively start
from the beginning of the algorithm with $X'$.

If we do not find such a bislim structure, we finish and declare that
we do not know whether the original 2-complex $X$ is hyperbolic nor
whether it has negative immersions.

If we remove the last 2-cell, we finish and declare that the original
2-complex $X$ is hyperbolic and has negative immersions.
\end{lyxalgorithm}

It follows by repeated applications of Lemma \ref{lem:relative_negative_imm_prop}
that if the algorithm accepts, $X$ indeed has negative immersions
and hyperbolicity. A-priori, the choice of which bislim structure
is picked in the algorithm influences which 2-cell is removed, which
may change the future execution of the algorithm and its final result.
It can be proved that regardless of which choices are picked, the
result will be the same. This algorithm has time complexity $O\largeparens{n^{2r+6}}$
where $r$ is the number of 2-cells in $X$, and $n$ is $\sum_{C\in X}\left|\boundary C\right|$
where the sum is over all 2-cells $C$ in $X$.

\begin{lyxalgorithm}
Given a 2-complex $X$, go over all alternating heightened complexes
$\heightened[X']$, for complexes $X'$ which are formed by removing
some 2-cells from $X$. For each alternating heightened complex $\heightened[X']$,
check if $\heightened[X']$ is bislim.

Recall the expanded Euler characteristic polyhedron $\aa$ (Definition
\ref{def:expanded_euler_characteristic_polyhedron}). By Proposition
\ref{prop:bislim_structure_to_expanded_euler_characteristic_polyhedron}
and Lemma \ref{lem:expanded_euler_characteristic_polyhedron_adding_cells},
each heightened complex $\heightened[X']$ which is bislim induces
a vector in $\aa$.

Gather the generated vectors in $\aa$, and check whether their convex
hull contains a vector $v=\smallparens{\beta,\alpha_{1},\dots,\alpha_{r}}$
with coefficients $\alpha_{1},\dots,\alpha_{r}$ all larger than $1$.
If it does, declare that $X$ is hyperbolic and has negative immersions.
\end{lyxalgorithm}

If $v$ exists, then $v$ is in $\aa$, and by definition of $\aa$
the 2-complex $X$ has negative immersions. Additionally, by Lemma
\ref{lem:proving_hyperbolicity} the 2-complex $X$ is hyperbolic.

This algorithm has a much larger complexity. Let $r$ be the number
of 2-cells in $X$, let $k$ be the number of edges in $X$, and let
$n$ be $\sum_{C\in X}\left|\boundary C\right|$ for all 2-cells $C$
in $X$. Then the number of possible alternating heightened complexes
that need to be checked is at most $2^{r}{n \choose 2k-1}$ and the
overall time complexity of the algorithm is $\text{poly}\largeparens{2^{r}{n \choose 2k-1}}$.

\subsection{\label{subsec:Computational-experiments}Computational results}

\global\long\def\slcpi{\text{PCI}^{\pm}}%
\global\long\def\slpci{\text{PCI}^{\pm}}%
\global\long\def\pci{\text{PCI}^{\pm}}%

Using Algorithm \ref{alg:find_best_bislim_structure} we conducted
a computational experiment to find which words admit alternating bislim
structures in practice. We prove Theorem \ref{thm:conjecture_holds_for_small_words},
i.e., that Conjecture \ref{conj:asymptotic_growth_conjecture_Sn_spi}
holds for all words $w\in F_{4}$ of length at most $10$. In fact,
for all words $w\in F_{4}$ of length at most $10$, there is an automorphic
word $w'$ which admits a $\smallparens{\pi\smallparens w-1}$-alternating
bislim structure. This also holds for all words of length at most
$11$ except six equivalence classes with the following representatives:
\begin{itemize}
\item $\left(aba^{-1}b^{-1}\right)^{2}cac^{-1}$
\item $\left(abab^{-1}\right)^{2}cac^{-1}$
\item $\left(abab^{-1}\right)^{2}ca^{-1}c^{-1}$
\item $\left(abab^{-1}\right)^{2}cbc^{-1}$
\item $a^{2}\left(bc\right)^{2}a^{-1}\left(bc^{-1}\right)^{2}$
\item $a^{2}\left(bc\right)^{2}a^{-1}\left(c^{-1}b\right)^{2}$.
\end{itemize}
In the previous statement, we consider two words $w_{1}$ and $w_{2}$
to be \emph{equivalent} if $w_{1}$ is automorphic to either $w_{2}$
or $w_{2}^{-1}$. Since $\sigma$ and $\sigma^{-1}$ are conjugate
for every $\sigma\in S_{n}$, the words $w$ and $w^{-1}$ induce
the same word measures on $S_{n}$ for every word $w$, and so they
can be considered to be equivalent for our purposes.

The first four of these words $w$ satisfy $\pi\smallparens w=3$,
$\spi\smallparens w=1.5$. Indeed, Figure \ref{fig:spi_counterexample}
shows a core graph $\Gamma$ with two $w$-cycles where $w=\left(abab^{-1}\right)^{2}cac^{-1}$
and $\chi\smallparens{\Gamma}=-3$, and $\Gamma$ satisfies the conditions
of the definition of $\spi$ given in \cite[Definition~4.1]{Puder2023},
showing that $\spi\smallparens w\le1.5$. Moreover, it can also be
computed for this word that 
\[
\expectation{\psi_{n}}=\frac{2n^{5}-24n^{4}+94n^{3}-96n^{2}-112n-80}{n\smallparens{n-1}^{2}\smallparens{n-2}^{2}\smallparens{n-3}\smallparens{n-4}\smallparens{n-5}}=\Omega\left(n^{-3}\right)
\]
\begin{wrapfigure}[15]{o}{0.3\columnwidth}%
\begin{centering}
\includegraphics[width=0.3\columnwidth]{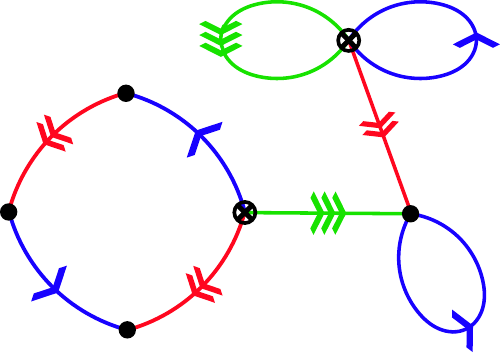}
\par\end{centering}
\caption{\label{fig:spi_counterexample}A core graph $\Gamma$ with two $w$-cycles
for the word $w=\left(abab^{-1}\right)^{2}cac^{-1}$, which start
on the two marked vertices}
\end{wrapfigure}%
where $\psi_{n}=\bigwedge^{2}\chi_{\text{std},n}$ is the exterior
power of the standard character of $S_{n}$. Thus this word is a counterexamples
to Conjectures \ref{conj:asymptotic_growth_conjecture_Sn_pi} and
\ref{conj:spi_equals_pi-1}, proving Theorem \ref{thm:counterexample}.

However, for these four words $w$ it can be proven that $\spi\smallparens w=1.5$
using a concrete combinatorial argument, which together with the analysis
of Section \ref{subsec:Character-expectations-on-Sn} also implies
that Conjecture \ref{conj:asymptotic_growth_conjecture_Sn_spi} holds
for them.\footnote{We intend to include these arguments in a separate paper.}

We also note that the probability bounds of Theorem \ref{thm:generic_bislim_structures_one_relator_case}
seem to be far from tight, as a large number of words $w$ admit an
$\spi\smallparens w$-alternating bislim structure directly.

We now show how these results were computed. Recall that if words
$w_{1}$ and $w_{2}$ are equivalent then Conjectures \ref{conj:asymptotic_growth_conjecture_Sn_pi},
\ref{conj:spi_equals_pi-1} and \ref{conj:asymptotic_growth_conjecture_Sn_spi}
hold for $w_{1}$ if and only if they hold for $w_{2}$, while it
may be that $w_{1}$ admits a $\smallparens{\pi\smallparens w-1}$-alternating
bislim structure even though $w_{2}$ does not.\footnote{E.g., $w_{1}=a^{2}b^{2}c^{2}$ and $w_{2}=abcab^{-1}c^{-1}$.}
However, there are some symmetries which do preserve bislim structures:
\begin{defn}[{\cite[Section~4.1]{Cashen_experimental_verification_of_relation_of_primitivity_rank_and_hyperbolicity}}]
Two words $w_{1}$ and $w_{2}$ in $F_{r}$ are $\slcpi$-equivalent
if they can be made equivalent by the following operations:
\begin{itemize}
\item Permuting the generators of $F_{r}$
\item Replacing all occurrences of a generator $x$ of $F_{r}$ by $x^{-1}$
\item Replacing $w$ with a cyclic rotation of $w$
\item Replacing $w$ by $w^{-1}$.
\end{itemize}
\end{defn}

Words $w$ which are $\pci$-equivalent have isomorphic presentation
complexes $\left\langle F\vert w\right\rangle $, and therefore have
the same bislim structures, so we only need to check one representative
of each $\pci$-equivalence class. Note that although $w$ and $w^{-1}$
are not necessarily automorphic, they are still equivalent.

\subsubsection{An Algorithm}

Given a word $w$, let $\mathemph{\alt_{w}}$ be the maximal integer
for which $w$ admits an $\alt_{w}$-alternating bislim structure.
It is possible to compute $\alt_{w}$ (Algorithm \ref{alg:find_best_bislim_structure})
and the primitivity rank $\pi\smallparens w$ (\cite{Cashen_experimental_verification_of_relation_of_primitivity_rank_and_hyperbolicity,Puder2011})
in polynomial time, assuming that the rank of the ambient free group
is fixed. In this section we define Algorithm \ref{alg:exhaustive_search},
which checks for all equivalence classes of words $w\in F_{r}$ of
length at most $\ell$ whether they admit $\left(\pi\smallparens w-1\right)$-alternating
bislim structures. Since this is true for most equivalence classes,
the algorithm only returns representatives of equivalence classes
for which this might not hold. Every equivalence class which is not
returned by this algorithm is guaranteed to have some word $w$ which
admits a $\smallparens{\pi\smallparens w-1}$-alternating bislim structure,
and therefore, Conjectures \ref{conj:asymptotic_growth_conjecture_Sn_pi},
\ref{conj:spi_equals_pi-1} and \ref{conj:asymptotic_growth_conjecture_Sn_spi}
all hold for any such equivalence class.

The paper \cite[Section~4.1]{Cashen_experimental_verification_of_relation_of_primitivity_rank_and_hyperbolicity}
details an algorithm which iterates over the representatives of $\pci$
equivalence classes without iterating over all words, which is instrumental
in exhaustively checking a large number of words. Additionally, since
most words do admit optimally alternating bislim structures, we can
save a significant amount of memory by storing only the words which
do not admit optimally alternating bislim structures. The resulting
algorithm is as follows:
\begin{lyxalgorithm}
\label{alg:exhaustive_search}For every $\slcpi$-representative $w\in F_{r}$
of length $\ell$, compute $\alt_{w}$, the primitivity rank $\pi\smallparens w$,
and whether $w$ is whitehead minimal.
\begin{itemize}
\item If $w$ is not whitehead minimal, skip $w$.
\item If $\alt_{w}=\pi\smallparens w-1$, skip $w$.
\end{itemize}
Collect all of the remaining words into a list of the whitehead minimal
words which might not satisfy Conjecture \ref{conj:asymptotic_growth_conjecture_Sn_spi}.
For each such word $w$, compute the set of whitehead minimal words
which are automorphic to $w$. If any of them is missing from the
list, then they admit a $\smallparens{\pi\smallparens w-1}$-alternating
bislim structure, so we can remove $w$ from the list as well.

For each equivalence class for which Conjecture \ref{conj:asymptotic_growth_conjecture_Sn_spi}
has not yet been verified, consider the words in its equivalence class
which are not necessarily whitehead minimal, and check whether they
admit a $\smallparens{\pi\smallparens w-1}$-alternating bislim structure.
Since each equivalence class is infinite, cap the search at some arbitrary
bound.

Finally, return one representative for each remaining equivalence
class.
\end{lyxalgorithm}

The space complexity of this algorithm is dominated by the number
of words which do not admit optimal bislim structures, which is a
small fraction of all of the searched words in practice. The algorithm
is implemented in \cite{my_code}. 


\let\origpath\path
\let\origurl\url
\renewcommand{\path}[1]{\texttt{\detokenize{#1}}}
\renewcommand{\url}[1]{\href{#1}{\texttt{\detokenize{#1}}}}

\bibliographystyle{alphaurl}
\bibliography{citation_library}

\let\path\origpath
\let\url\origurl

\appendix

\section{\label{sec:equivalence-to-bi-slim}Equivalence to the original definition
of bislim structures}

We show that our definition of good bislim structures is equivalent
to the original definition of bislim structures, \cite[Definition 2.1]{helferwise}.
\begin{defn}[Preorders and partial orders]
A \emph{preorder} $\preceq$ is a transitive reflexive relation.
A \emph{partial order} is a preorder which is antisymmetric\emph{.
}We denote by $a\prec b$ the statement that $a\preceq b$ and $b\not\preceq a$.
\end{defn}

\global\long\def\setboundary{\boundary^{*}}%

\begin{defn}[Multiset boundary]
For a 2-cell $C$ in a complex $X$, denote by $\mathemph{\setboundary C}$
the multiset of edges of $X$ which are adjacent to $C$. This differs
from $\boundary C$ in that $\boundary C$ consists of a cyclic path
sides of $X$, and $\setboundary C$ is the multiset of the corresponding
edges of $X$. This is required to state the original definition of
bislim structures precisely.
\end{defn}

Instead of working with the original definition (\cite[Definition 2.1]{helferwise})
directly, we start from \cite[Definition~3.1]{Bamberger2024}, which
is a slightly reworked version (see \cite[Remark~3.2]{Bamberger2024}).
\begin{defn}[{bislim structure \cite[Definition~3.1]{Bamberger2024}}]
\label{def:original_bislim}A connected combinatorial 2-complex $X$
is bislim if the following conditions hold:
\begin{enumerate}
\item There is a $\pi_{1}X$-invariant preorder $\preceq$ on the set of
edges of $\universalCover X$ (the universal cover of $X$).
\item \label{enu:distinguished_high_and_low_edge_in_original_bislim_definition}The
boundary $\setboundary C$ has two distinguished edges $e_{C}^{+}$
and $e_{C}^{-}$ for each 2-cell $C$ of $\universalCover X$.

Moreover, $e_{C}^{+}$ is traversed exactly once by the boundary $\setboundary C$.
\item The edges $e_{C}^{+}$ and $e_{C}^{-}$ are $\pi_{1}X$-invariant,
i.e., $ge_{C}^{\pm}=e_{gC}^{\pm}$ for every 2-cell $C$ of $\universalCover X$
and $g\in\pi_{1}X$ acting on $\universalCover X$.
\item \label{enu:original_bislim_e+}For distinct 2-cells $C_{1}$ and $C_{2}$
of $\universalCover X$, if $e_{C_{1}}^{+}\in\setboundary C_{2}$
then $e_{C_{1}}^{+}\prec e_{C_{2}}^{+}$.
\item \label{enu:original_bislim_e-}For distinct 2-cells $C_{1}$ and $C_{2}$
of $\universalCover X$, if $e_{C_{2}}^{-}\in\setboundary C_{1}$
then $e_{C_{1}}^{+}\prec e_{C_{2}}^{+}$.
\end{enumerate}
\end{defn}

By \cite[Remark 3.4]{Bamberger2024} we can assume without loss of
generality that $\preceq$ is antisymmetric (i.e. $\preceq$ is a
partial order). We will prove that Definition \ref{def:original_bislim}
is equivalent to our definition of bislim structures (Definition \ref{def:bislim_structure_for_good_complex})
by using Condition \ref{enu:bislim-definition-universal-cover} of
Theorem \ref{thm:equivalent_conditions_for_bislim_structure}, which
states that a good heightened complex $\heightened$ is bislim if
and only if $\Gamma\largeparens{\universalCover X}$ is acyclic. In
order to show the similarity of the two definitions, we define a preorder
$\trianglelefteq$ which parallels $\preceq$ in Definition \ref{def:original_bislim}.
\begin{defn}
Let $\heightened$ be a heightened complex. Define\emph{ }$\mathemph{\trianglelefteq}$
to be the preorder on the set of 2-cells of $\universalCover X$ which
is the transitive closure of $\Gamma\largeparens{\universalCover X}$.
Namely, $C_{1}\trianglelefteq C_{2}$ if there is a directed path
in $\Gamma\largeparens{\universalCover X}$ from $C_{1}$ to $C_{2}$.
\end{defn}

Before proving the equivalence between our definition of bislim structure
and the original definition, we will informally explain how the two
definitions correspond to each other. First, note that if $\heightened$
is a good bislim structure, we may assume without loss of generality
that $\left|H_{C}\right|=\left|L_{C}\right|=1$ for every 2-cell $C$
of $X$ by removing extra sides from $H$ and $L$. Now, $e_{C}^{+}$
corresponds to the only high side of $C$ and $e_{C}^{-}$ corresponds
to the only low side of $C$. Additionally, the preorders $\preceq$
and $\trianglelefteq$ capture the same information: the statement
$C_{1}\trianglelefteq C_{2}$ roughly corresponds to the statement
$e_{C_{1}}^{+}\preceq e_{C_{2}}^{+}$. Namely, the edge $e_{C}^{+}$
in the preorder $\preceq$ represents the 2-cell $C$ in the preorder
$\trianglelefteq$.

We now prove that both definitions are indeed equivalent.
\begin{thm}
\label{thm:equivalence_between_new_and_old_definitions_of_bislim_structures}A
connected complex $X$ admits a good bislim structure in the sense
of Definition \ref{def:bislim_structure_for_good_complex} if and
only if $X$ is bislim in the sense of Definition \ref{def:original_bislim}.
\end{thm}

\begin{proof}[\textbf{Part I}]
If $X$ admits a good bislim structure in the sense of Definition
\ref{def:bislim_structure_for_good_complex} then $X$ is bislim in
the sense of Definition \ref{def:original_bislim}.

Assume without loss of generality that $\left|H_{C}\right|=\left|L_{C}\right|=1$
for every 2-cell $C$ of $X$ by removing extra sides from $H$ and
$L$. The induced bislim structure of $\universalCover X$ is $\largeparens{\universalCover X,\phi^{-1}\smallparens H,\phi^{-1}\smallparens L}$
where $\phi:\universalCover X\to X$ is the canonical projection.
Thus $\phi^{-1}\smallparens H$ and $\phi^{-1}\smallparens L$ are
$\pi_{1}X$-invariant, and so $\Gamma\largeparens{\universalCover X}$
and the preorder $\trianglelefteq$ are also $\pi_{1}X$-invariant.
By Theorem \ref{thm:equivalent_conditions_for_bislim_structure},
$\Gamma\largeparens{\universalCover X}$ is acyclic, and so $\trianglelefteq$
is antisymmetric.

It also follows that $\left|H_{C}\right|=\left|L_{C}\right|=1$ for
every 2-cell $C$ of $\universalCover X$, so we denote $H_{C}=\left\{ s_{C}^{+}\right\} $,
$L_{C}=\left\{ s_{C}^{-}\right\} $ and denote by $e_{C}^{+}$ ($e_{C}^{-}$)
the edge of $\universalCover X$ corresponding to the side $s_{C}^{+}$
($s_{C}^{-}$, respectively).

Denote $f\smallparens C=e_{C}^{+}$ for every 2-cell $C$ of $\universalCover X$.
The mapping $f$ is injective, since otherwise $\Gamma\largeparens{\universalCover X}$
would have a directed cycle of length 2. We define a preorder $\preceq$
on the set of edges of $\universalCover X$ like so: $e_{1}\preceq e_{2}$
if either $e_{1}=e_{2}$ or $f^{-1}\smallparens{e_{1}}$ and $f^{-1}\smallparens{e_{2}}$
exist and $f^{-1}\smallparens{e_{1}}\trianglelefteq f^{-1}\smallparens{e_{2}}$.
This preorder satisfies Definition \ref{def:original_bislim}:
\begin{enumerate}
\item The preorder $\preceq$ is $\fundamental X$-invariant since $\trianglelefteq$
and the edges $e_{C}^{+}$ are $\fundamental X$-invariant.
\item The edge $e_{C}^{+}$ is traversed exactly once by the boundary path
$\setboundary C$: otherwise $\Gamma\largeparens{\universalCover X}$
would have a self loop at $C$, in contradiction with $\Gamma\largeparens{\universalCover X}$
being acyclic.
\item The edges $e_{C}^{+}$ and $e_{C}^{-}$ are indeed chosen in a $\pi_{1}X$-invariant
way.
\item If $C_{1}$ and $C_{2}$ are distinct 2-cells in $\universalCover X$
and $e_{C_{1}}^{+}\in\setboundary C_{2}$ then $e_{C_{1}}^{+}\prec e_{C_{2}}^{+}$
: since $e_{C_{1}}^{+}\in\setboundary C_{2}$ , the side $s_{C_{1}}^{+}$
induces an edge from $C_{1}$ to $C_{2}$ in $\Gamma\largeparens{\universalCover X}$
by the definition of $\Gamma\largeparens{\universalCover X}$. Thus
$C_{1}\trianglelefteq C_{2}$. Since $C_{1}\neq C_{2}$ we have $C_{1}\triangleleft C_{2}$,
and so $e_{C_{1}}^{+}\prec e_{C_{2}}^{+}$.
\item If $C_{1}$ and $C_{2}$ are distinct 2-cells in $\universalCover X$
and $e_{C_{2}}^{-}\in\setboundary C_{1}$ then $e_{C_{1}}^{+}\prec e_{C_{2}}^{+}$:
since $e_{C_{2}}^{-}\in\setboundary C_{1}$ , the side $s_{C_{2}}^{-}$
induces an edge from $C_{1}$ to $C_{2}$ in $\Gamma\largeparens{\universalCover X}$
by the definition of $\Gamma\largeparens{\universalCover X}$. Thus
$C_{1}\trianglelefteq C_{2}$. Since $C_{1}\neq C_{2}$ we have $C_{1}\triangleleft C_{2}$,
and so $e_{C_{1}}^{+}\prec e_{C_{2}}^{+}$.
\end{enumerate}
\end{proof}
\begin{proof}[\textbf{Part II}]
 If $X$ is bislim in the sense of Definition \ref{def:original_bislim},
then it admits a good bislim structure in the sense of Definition
\ref{def:bislim_structure_for_good_complex}.

By Condition \ref{enu:distinguished_high_and_low_edge_in_original_bislim_definition}
of Definition \ref{def:original_bislim}, the complex $X$ is necessarily
non-degenerate. For each 2-cell $\universalCover C$ of $\universalCover X$
let $s_{\universalCover C}^{+}$ be the side of $\universalCover C$
corresponding to $e_{\universalCover C}^{+}$. Additionally let $s_{\universalCover C}^{-}$
be any side of $\universalCover C$ corresponding to $e_{\universalCover C}^{-}$,
picked in a $\pi_{1}X$-invariant manner. Let $\pi:\universalCover X\to X$
be the canonical projection, and for a 2-cell $C$ of $X$ let $s_{C}^{+}=\pi\largeparens{s_{\universalCover C}^{+}}$
and $s_{C}^{-}=\pi\largeparens{s_{\universalCover C}^{-}}$ where
$\universalCover C$ is any lift of $C$ to $\universalCover X$,
so $\pi\largeparens{\universalCover C}=C$. Note that $s_{C}^{+}$
and $s_{C}^{-}$ are indeed sides of $C$.

Let $H=\left\{ s_{C}^{+}\vert C\text{ is a 2-cell of }X\right\} $
and $L=\left\{ s_{C}^{-}\vert C\text{ is a 2-cell of }X\right\} $.
Then $\heightened$ is a good heightened complex, $H_{C}=\left\{ s_{C}^{+}\right\} $,
$L_{C}=\left\{ s_{C}^{-}\right\} $. To show that $\heightened$ is
bislim, it remains to show that the induced graph $\Gamma\largeparens{\universalCover X}$
is acyclic.

Recall that the induced heightened structure of $\universalCover X$
is $\largeparens{\universalCover X,\pi^{-1}\smallparens H,\pi^{-1}\smallparens L}$.
Since the choices of $s_{\universalCover C}^{\pm}$ are $\pi_{1}X$-invariant,
we have $\pi^{-1}\smallparens H=\left\{ s_{\universalCover C}^{+}\ \vert\ \universalCover C\text{ is a 2-cell of }\universalCover X\right\} $
and $\pi^{-1}\smallparens L=\left\{ s_{\universalCover C}^{-}\ \vert\ \universalCover C\text{ is a 2-cell of }\universalCover X\right\} $.

Suppose that some 2-cell $\universalCover C$ in $\Gamma\left(\universalCover X\right)$
has a self loop. By the definition of $\Gamma$, it follows that either
$e_{\universalCover C}^{+}$ or $e_{\universalCover C}^{-}$ is traversed
twice by $\setboundary\universalCover C$, which is impossible by
\cite[Corollary~4.6]{Bamberger2024}.

Now suppose that there is a directed cycle $\universalCover{C_{1}},\universalCover{C_{2}},\dots,\universalCover{C_{n}}$
in $\Gamma\largeparens{\universalCover X}$ with $\universalCover{C_{i}}\neq\universalCover{C_{i+1}}$
for all $i$. By the definition of $\Gamma\largeparens{\universalCover X}$,
the edge from $\universalCover{C_{i}}$ to $\universalCover{C_{i+1}}$
in $\Gamma\text{\ensuremath{\left(\universalCover X\right)}}$ is
induced by a side $s$ in $H_{\universalCover{C_{i}}}\cup L_{\universalCover{C_{i+1}}}$
which is adjacent to both $C_{i}$ and $C_{i+1}$. Thus $s=s_{\universalCover{C_{i}}}^{+}$
or $s=s_{\universalCover{C_{i+1}}}^{-}$. In both cases, $e_{\universalCover{C_{i}}}^{+}\prec e_{\universalCover{C_{i+1}}}^{+}$
by either Condition \ref{enu:original_bislim_e+} or \ref{enu:original_bislim_e-}
of Definition \ref{def:original_bislim}. Therefore,
\[
e_{\universalCover{C_{1}}}^{+}\prec e_{\universalCover{C_{2}}}^{+}\prec\dots\prec e_{\universalCover{C_{n}}}^{+}\prec e_{\universalCover{C_{1}}}^{+},
\]
a contradiction. Thus $\Gamma\largeparens{\universalCover X}$ is
acyclic and $\heightened$ is bislim by Condition \ref{enu:bislim-definition-universal-cover}
of Theorem \ref{thm:equivalent_conditions_for_bislim_structure}.
\end{proof}

\end{document}